\documentclass[a4paper, english, reqno]{amsart}
\usepackage[utf8]{inputenc}

\usepackage{amsmath, amsthm, amssymb}
\usepackage{mathtools}

\usepackage{tabularx}
\usepackage{makecell}
\usepackage{adjustbox}

\usepackage{paralist}  
\usepackage{enumitem}

\usepackage{graphicx}
\usepackage{float}

\usepackage{tikz}
\usetikzlibrary{positioning,fit}

\usepackage[all,cmtip]{xy}

\usepackage{tikz-cd} 
\newcommand\mlnode[1]{\fbox{{\renewcommand{\arraystretch}{1.2}\begin{tabular}{@{}c@{}}#1\end{tabular}}}} 

\usepackage{nameref}
\usepackage{hyperref}
\usepackage[capitalise, nameinlink]{cleveref}

\renewcommand\le\leqslant
\renewcommand\ge\geqslant

\calclayout

\makeatletter
\newcommand{\pushright}[1]{\ifmeasuring@#1\else\omit\hfill$\displaystyle#1$\fi\ignorespaces}
\makeatother

\usepackage[backend=bibtex, style=numeric]{biblatex} 
\setlist[itemize]{align=parleft,left=0pt..18pt,topsep=5pt,itemsep=2pt,label={\(-\)}}
\setlist[enumerate]{align=parleft,left=0pt..18pt,topsep=5pt, itemsep=2pt}

\newcommand{\fg}[0]{\mathfrak{g}}

\newcommand{\fn}{\mathfrak{n}}
\newcommand{\fh}{\mathfrak{h}}

\newcommand{\fT}{\mathfrak{T}}
\newcommand{\fD}{\mathfrak{D}}

\newcommand{\kf}{\kappa}

\newcommand{\bC}[0]{\mathbb{C}}

\newcommand{\bZ}[0]{\mathbb{Z}}

\newcommand{\add}{\dotplus}
\newcommand{\ot}{\otimes}

\newcommand{\fot}{\,\overline{\otimes}\,}
\newcommand{\uot}{\,\widetilde{\otimes}\,}

\DeclareMathOperator{\Hom}{Hom}
\DeclareMathOperator{\Aut}{Aut}

\DeclareMathOperator{\tr}{tr}
\DeclareMathOperator{\tOp}{t}

\newcommand{\dual}[1]{#1^{\vee}}

\DeclareMathOperator{\spanOp}{span}
\newcommand{\Span}{\mathord{\spanOp}}

\DeclareMathOperator{\adOp}{ad}
\newcommand{\ad}{\mathord{\adOp}}

\DeclareMathOperator{\imOp}{im}
\newcommand{\im}{\mathord{\imOp}}

\numberwithin{equation}{section}
\renewcommand*{\theequation}{%
  \ifnum\value{section}>0 %
    \thesection.%
  \fi
    \arabic{equation}%
}

\newtheorem{thm}{Theorem}[section]
\theoremstyle{plain}

\newtheorem{corollary}[thm]{Corollary}
\newtheorem{lemma}[thm]{Lemma}
\newtheorem{proposition}[thm]{Proposition}
\newtheorem{theorem}[thm]{Theorem}

\newtheorem{maintheoremcounter}{Theorem}
\theoremstyle{plain}

\newtheorem{maincorollary}[maintheoremcounter]{Corollary}
\newtheorem{maintheorem}[maintheoremcounter]{Theorem}

\newtheoremstyle{example_style}
  {5pt} 
  {5pt} 
  {} 
  {} 
  {\itshape} 
  {.} 
  {10pt} 
  {} 

\theoremstyle{example_style}
\newtheorem{examplex}[thm]{Example}
\newenvironment{example}
  {\pushQED{\qed}\examplex}
  {\popQED\endexamplex}
  
  \crefname{examplex}{Example}{Examples}

\newtheorem{remarkx}[thm]{Remark}
\newenvironment{remark}
  {\pushQED{\qed}\remarkx}
 {\popQED\endremarkx}

\theoremstyle{definition}

\title{Topological Lie bialgebra structures and their classification
over \( \fg[\![x]\!] \).}

\author{Raschid Abedin}
\address[R.A.]{Institut für Mathematik,
Universität Paderborn, 33098 Paderborn,
Germany}
\email{rabedin@math.upb.de}

\author{Stepan Maximov}
\address[S.M.]{Department of Mathematical Sciences, 
Chalmers University of Technology  and the University of Gothenburg, 412 96 Gothenburg, Sweden
}
\email{maximov@chalmers.se}

\author{Alexander Stolin}
\address[A.S.]{Department of Mathematical Sciences, 
Chalmers University of Technology  and the University of Gothenburg, 412 96 Gothenburg, Sweden
}
\email{astolin@chalmers.se}

\author{Efim Zelmanov}
\address[E.Z.]{Department of Mathematics, UCSD, 9500 Gilman Drive, La Jolla, CA 92093-0112, USA}
\email{ezelmano@math.ucsd.edu}

\date{\today}

\begin{document}

\maketitle

\begin{abstract}
This paper is devoted to a classification of topological
Lie bialgebra structures
on the Lie algebra $\mathfrak{g}[\![x]\!]$,
where $ \mathfrak{g} $ is a finite-dimensional simple Lie algebra
over an algebraically closed field $ F $ of characteristic $ 0 $.

We introduce the notion of a topological
Manin pair $(L, \mathfrak{g}[\![x]\!])$ and 
present their classification by relating
them to trace extensions of 
\( F[\![x]\!] \).
Then we recall the
classification of topological doubles of Lie bialgebra
structures on $\mathfrak{g}[\![x]\!]$
and view the latter as a special case
of the classification of Manin pairs.

The classification of topological doubles
states that up to some notion of equivalence
there are only three non-trivial doubles.
It is proven that topological
Lie bialgebra structures on $\mathfrak{g}[\![x]\!]$ are in bijection
with certain Lagrangian Lie subalgebras
of the corresponding doubles. 
We then attach algebro-geometric data to such Lagrangian subalgebras
and, in this way, obtain a classification of all topological Lie bialgebra structures with non-trivial doubles.
When $F = \mathbb{C}$ the classification becomes explicit.
Furthermore, this result enables us to classify formal
solutions of the classical Yang-Baxter equation.

\end{abstract}

\tableofcontents

\section{Introduction}
Let \(\fg\) be a finite-dimensional simple Lie algebra over an algebraically closed field \(F\) of characteristic 0. 
Topological Lie bialgebra structures on \(\fg[\![x]\!]\)
arise naturally as completions of several important 
classical Lie bialgebras, e.g.\
any Lie bialgebra structure on
\( \fg[x] \)
or \( \fg[x,x^{-1}] \) extends to
\( \fg[\![x]\!] \).
Moreover, these structures
are closely related to
Hopf-algebra deformations of the universal enveloping algebra of \(\fg\); see \cite{etingof_Kazhdan_quantization_I,etingof_Kazhdan_quantization_II}.
This makes them important in the theory of integrable systems and quantum groups. This paper is dedicated to the classification of topological Lie bialgebra structures on \(\fg[\![x]\!]\).

Let us describe one possible
approach to the classification of Lie bialgebra structures.
Recall that a Lie bialgebra over \( F \) is a vector space  \( L \) over \( F \)
(not necessarily finite-dimensional)
which is equipped with a Lie algebra structure \( [\cdot \, ,\cdot] \colon L \ot L \to L \)
and a Lie coalgebra structure \( \delta \colon L \to L \ot L \)
subject to the compatibility condition
\begin{equation}\label{eq:compatibility}
\delta([x,y]) = [x \ot 1 + 1 \ot x, \delta(y)] - [y \ot 1 + 1 \ot y, \delta(x)]\qquad \forall x,y \in L.
\end{equation}
Each Lie bialgebra structure
\( \delta \) on \( L \) is
associated with the object
\( \fD(L, \delta) \)
called the classical double of
\( \delta \).
It is the vector space
\( L \add \dual{L} \)
equipped with the bilinear form
\begin{equation}\label{eq:classical_form}
  B(x + f, y + g) \coloneqq f(y) + g(x) \qquad \forall x,y \in L, \ \forall f,g \in \dual{L},
\end{equation}
and the commutator map
\begin{equation}\label{eq:classical_bracket}
  [x,f] \coloneqq -f\circ \ad_x + (f \ot 1)(\delta(x)) \qquad \forall x \in L, \ \forall f \in \dual{L}.
\end{equation}
Having a Lie bialgebra structure
\( \delta \) on \( L \)
we can obtain new Lie bialgebra structures by a procedure called
twisting.
More explicitly, 
for any skew-symmetric tensor
\( s \in L \ot L \) such that
\begin{equation}\label{eq:classical_twist_condition_intro}
\textnormal{CYB}(s) = \textnormal{Alt}((\delta \ot 1)s),
\end{equation}
where
\(
\textnormal{CYB}(s) \coloneqq [s^{12}, s^{13}] + [s^{12}, s^{23}] + [s^{13}, s^{23}]
\)
and
\(
  \textnormal{Alt}(a \ot b  \ot  c) \coloneqq a \ot b \ot c + b \ot c \ot a + c \ot a \ot b,
\)
the linear map
\( \delta_s \coloneqq \delta + ds \)
defines another Lie bialgebra
structure on \( L \).
Skew-symmetric tensors satisfying
\cref{eq:classical_twist_condition_intro}
are called twists of \( \delta \).
One important property of the
double \( \fD(L, \delta) \) is its
invariance under twisting, i.e.\
\[ \fD(L, \delta) = \fD(L, \delta + ds) \] for any twist
\( s \) of \( \delta \).
For this reason doubles are sometimes called twisting classes.
This observation leads to the following simple classification scheme for Lie bialgebra structures on \( L \):
\begin{enumerate}
\item Classify all possible classical doubles \( \fD(L, \delta) = L \add \dual{L} \);

\item Choose a representative \( \delta \) inside each twisting class and describe all its twists.
By \cite[Theorem 2.4]{Raschid_Stepan_classical_twists}
this is equivalent to describing
Lagrangian Lie subalgebras \( W \subset \fD(L, \delta) \)
such that 
\[ \fD(L, \delta) = L \add W \ \text{ and } \
\dim(W + \dual{L})/(W \cap \dual{L}) < \infty. \]
\end{enumerate}

The second step of the scheme,
in the most interesting cases, is equivalent to the 
classification of certain \( r \)-matrices, i.e.\
solutions to the classical Yang-Baxter equation (CYBE)
\begin{equation}\label{eq:2_cybe_intro}
[r^{12}(x_1, x_2), r^{13}(x_1, x_3)] +
    [r^{12}(x_1, x_2), r^{23}(x_2, x_3)] +
    [r^{13}(x_1, x_3), r^{23}(x_2, x_3)] = 0.
\end{equation}
These were thoroughly studied by Belavin and Drinfeld
\cite{BD_first,BD_second} and by Stolin \cite{Stolin_rational_sln, Stolin_rational}.

\begin{example}
It was shown in 
\cite{Stolin_some_remarks} 
that the classical double
\(\fD \coloneqq \fD(\fg,\delta)\) of a Lie bialgebra structure \(\delta\) on \(\fg\)
is of the form \(\fg \otimes A\) for a two-dimensional unital associative commutative \(F\)-algebra \(A\). Consequently,
\[\fD \cong \fg\times \fg \, \text{ or } \, \fD \cong \fg[x]/x^2\fg[x] \] 
as a Lie algebra and
in both cases there is 
essentially only one possibility
for the bilinear form. 
Moreover, Whitehead's Lemma implies that in both cases \(\delta = dr\) for a constant solution 
\(r\in \fg \otimes \fg\) of the CYBE \cref{eq:2_cybe_intro}.
The isomorphism
\(\fD \cong \fg[x]/x^2\fg[x]\) holds if and only if \(r\) is skew-symmetric. 
Therefore, the classification of Lie bialgebra structures with \(\fD \cong \fg \times \fg\) (resp.\ \(\fD \cong \fg[x]/x^2\fg[x]\)) 
is equivalent to the classification of constant non-skew-symmetric (resp.\ skew-symmetric) \( r \)-matrices.
These classifications can be found in e.g.\  \cite{chari_pressley}.
\end{example}

It was proven in 
\cite{Stolin_Zelmanov_Montaner}
that any Lie bialgebra structure \( \delta \)
on \( \fg[x] \) satisfies
\begin{equation}\label{eq:inclusion_cont}
 \delta(x^n \fg[x]) \subseteq (x,y)^{n-1}(\fg \ot \fg)[x,y].
\end{equation}
In particular,
if we equip \( \fg[x] \) with
the \( (x)\)-adic topology
and \( \fg[x] \ot \fg[y] \cong (\fg \otimes \fg)[x,y]\) with
the projective
tensor product topology,
both the Lie bracket and the Lie
cobracket become continuous.
Therefore, they posses
unique continuous extensions
\begin{equation}
    [\cdot \, ,\cdot] \colon 
    (\fg \ot \fg)[\![x,y]\!]
    \to \fg[\![x]\!]
    \,
    \text{ and } \,
    \widehat{\delta} \colon
    \fg[\![x]\!] \to 
    (\fg \ot \fg)[\![x,y]\!].
\end{equation}
In general
\(\widehat{\delta}(\fg[\![x]\!]) \not\subseteq \fg[\![x]\!] \ot \fg[\![y]\!]\) and hence
it is not a Lie bialgebra
in the sense above.

A topological Lie bialgebra structure
on \( \fg[\![x]\!] \)
is a linear map 
\(
    \fg[\![x]\!] \to 
    (\fg \ot \fg)[\![x,y]\!]
\)
having properties similar
to the ones of \( \widehat{\delta} \) above.
As in the classical case,
a topological Lie bialgebra structure on \(\fg[\![x]\!]\)
gives rise to a topological
double -- topological
analogue of a classical double.
Let 
\[
\fg[\![x]\!]' \coloneqq \{f \colon \fg[\![x]\!] \to F \mid f(x^n\fg[\![x]\!]) = 0 \, \textnormal{ for some } n\in\mathbb{Z}_{+}\}.
\] 
be the space of continuous functionals on \(\fg[\![x]\!]\).
Then the topological double
of a topological Lie bialgebra structure 
\( \delta \) on \( \fg[\![x]\!] \)
is the vector space \( \fD = \fg[\![x]\!] \add \fg[\![x]\!]' \) with the continuous
commutator map defined similar to
\cref{eq:classical_bracket}
and separately continuous bilinear form \cref{eq:classical_form}.

The classification of topological
Lie bialgebra structures
on \( \fg[\![x]\!] \)
can be done using a similar scheme:
\begin{enumerate}
\item Classify all possible topological doubles \( \fD(\fg[\![x]\!], \delta) = \fg[\![x]\!] \add \fg[\![x]\!]' \);

\item Choose a representative \( \delta \) inside each twisting class and describe all its topological twists. By
\cref{thm: W <-> t <-> T correspondence}
this problem is
equivalent to classifying Lagrangian Lie subalgebras \( W \subset \fD(\fg[\![x]\!], \delta) \)
such that 
\( 
\fD(\fg[\![x]\!], \delta) = \fg[\![x]\!] \add W.
\)
\end{enumerate}

The first step of this classification
is done in \cite{Stolin_Zelmanov_Montaner}: it is shown that up to some
notion of equivalence
there are only three non-trivial
topological doubles \(\fD_1, \fD_2\) and \(\fD_3\). More explicitly,
\(\fD_i \coloneqq \fg(\!(x)\!) \times \fg[x]/x^{i-1}\fg[x]\) for \(i \in \{1,2,3\}\) with the bilinear form
\begin{equation}
    B_i((f_1,[f_2]),(g_1,[g_2])) = \textnormal{coeff}_{i-2}\{\kappa(f_1(x),g_1(x))\} - \textnormal{coeff}_{i-2}\{\kappa(f_2(x),g_2(x))\},
\end{equation}
where \(f_1,g_1 \in \fg(\!(x)\!) \), \(f_2,g_2 \in \fg[x]\) and \(\kappa \colon \fg(\!(x)\!) \times \fg(\!(x)\!) \to F(\!(x)\!)\) 
is the \(F(\!(x)\!)\)-bilinear
extension of the Killing form on \(\fg\).

The main result of this paper
is the completion of step (2) 
in the above-mentioned classification scheme.
We achieve this result
by using the connection between topological Lie bialgebra structures
on \(\fg[\![x]\!]\) and formal \(r\)-matrices, i.e.\ series 
\begin{equation}%
\label{eq:formal_rmatrices_intro}
    r(x,y) = \frac{s(y)\Omega}{x-y} + g(x,y) \in
    (\fg \ot \fg)(\!(x)\!)[\![y]\!],
\end{equation}
where \(s \in F[\![y]\!]\), \(g \in (\fg \otimes \fg)[\![x,y]\!]\) and
\(\Omega \in \fg \otimes \fg\) is the quadratic Casimir element,
solving \cref{eq:2_cybe_intro}. 
More precisely, we prove that
the assignment \(r \mapsto dr\),
where
\begin{equation}\label{eq:intro_dr}
    dr(f) \coloneqq [f(x) \fot 1 + 1 \fot f(y), r(x,y)],
\end{equation}
establishes a bijection between 
these objects.
The reflection of the classification 
of topological doubles
on the side of formal
\( r \)-matrices gives an
unexpected restriction on the
series \( s(y) \).
To be exact, there are no 
formal \( r \)-matrices of the form \cref{eq:formal_rmatrices_intro}
with a non-zero
\(s \in y^3 F[\![y]\!]\).

Following \cite{abedin2021geometrization}, to each Lagrangian Lie subalgebra
of \( \fD_i \) complementary to
\( \fg[\![x]\!] \)
we assign an algebro-geometric datum.
Then, using 
general methods of algebraic
geometry, 
we obtain significant restrictions
on formal \(r\)-matrices associated to 
Lagrangian Lie subalgebras of \( \fD_i \):

\begin{itemize}[align=parleft, left=0pt..3em]
    \item[\(i \nolinebreak = \nolinebreak 1\colon\)] The \(r\)-matrices are subject to a generalization of the Belavin-Drinfeld trichotomy:
    they are either of elliptic, trigonometric or rational type;
    
    \item[\(i \nolinebreak = \nolinebreak 2\colon\)] The corresponding \(r\)-matrices turn out to be
    quasi-trigonometric;
    
    \item[\(i \nolinebreak = \nolinebreak 3\colon\)] The associated formal
    \(r\)-matrices are of quasi-rational
    type.
\end{itemize}

\begin{figure}[H]
    \centering
    \includegraphics[scale = 0.25]{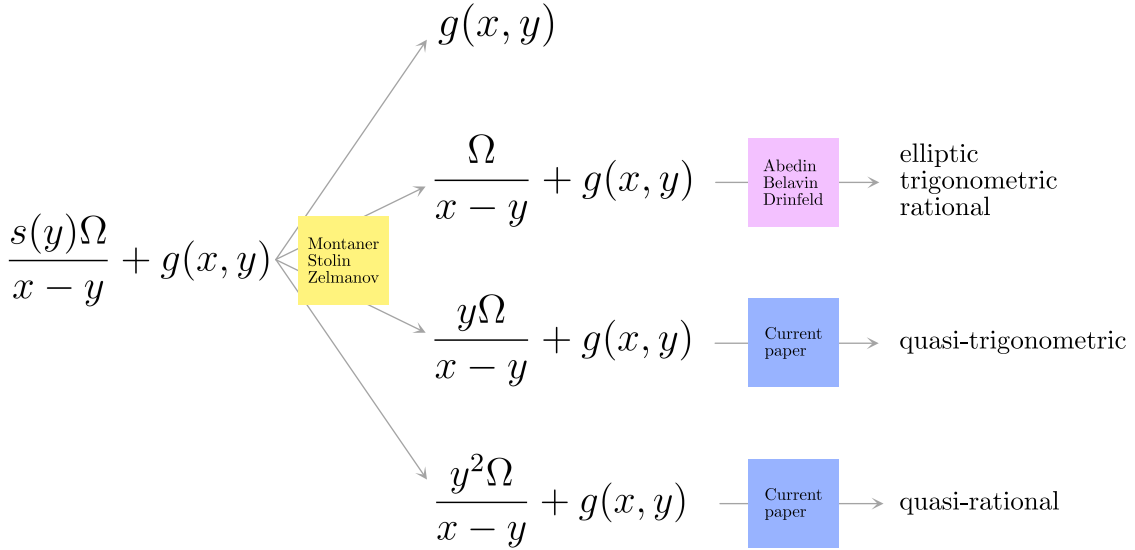}
    \caption{Restrictions on formal
    \(r\)-matrices}
    \label{fig:restrictions}
\end{figure}

When \(F = \mathbb{C}\), 
the above-mentioned five classes of
formal \(r\)-matrices
are completely classified in the available literature: Elliptic and trigonometric \(r\)-matrices are classified in \cite{BD_first}; Rational \(r\)-matrices are classified in \cite{Stolin_rational_sln,Stolin_rational};
Quasi-trigonometric \(r\)-matrices are described in \cite{Pop_Stolin_Lagrangian_quasi_trig} 
and independently in \cite{Raschid_Stepan_classical_twists};
Finally, quasi-rational \(r\)-matrices are classified in \cite{the_4_structure}.
Note that most of these descriptions
are valid for an arbitrary algebraically closed field of characteristic \(0\) without 
any adjustments.

\subsection{Structure of the paper}
\Cref{sec:Top_Lie_bialg} starts
with preliminary results
on linearly topologized
vector spaces and their completions as well as on linear
tensor product topologies.
Using these results we then
introduce the
notion of a topological Lie bialgebra structure
on an arbitrary linearly topologized Lie algebra \(L\) with continuous Lie bracket. 

\Cref{sec:review_MSZ} starts with the
the classification of Manin pairs
\( (L, \fg[\![x]\!]) \) 
via trace extensions of \( F[\![x]\!] \).
Then we review
the classification of topological
doubles over \(\fg[\![x]\!]\) from \cite{Stolin_Zelmanov_Montaner}
in the setting of Manin pairs.
In the process we generalize one of
the pivotal steps in the classification
\cite{Stolin_Zelmanov_Montaner}
and obtain
the following result.

\begin{maintheorem}
    Let \(L\) be a Lie algebra with a non-degenerate invariant symmetric bilinear form and \(A\) be a unital associative commutative reduced
    \(F\)-algebra. 
    If \(L\) contains a coisotropic subalgebra of the form \(\fg \otimes A\), then \(L \cong \fg \otimes \widetilde{A}\) for some associative 
    commutative \(F\)-algebra extension \( \widetilde{A} \supseteq A\). 
\end{maintheorem}

In \Cref{sec:topological_twists,sec:formal_r_matrices}
we introduce topological twists --
certain elements in 
\( (\fg \ot \fg )[\![x,y]\!] \)
that allow us to obtain new topological
Lie bialgebra structures on
\( \fg[\![x]\!] \) from a given
one with the same topological
double.
We relate these objects
to Lagrangian Lie subalgebras
of \( \fD_i \), \( i \in \{ 1,2,3\} \) and formal
\( r \)-matrices. 
The main results can be summarized as follows.

\begin{maintheorem}
There is a bijection between 
\begin{enumerate}
    \item Topological Lie bialgebra structures on \(\fg[\![x]\!]\) with
    the topological double \(\fD_i\);
    
    \item Formal \(r\)-matrices
    of the form
    \( y^{i-1}\Omega/(x-y) + g(x,y) \)
    and
    
    \item Lagrangian Lie subalgebras of \(\fD_i\) complementary to \(\fg[\![x]\!]\).
\end{enumerate}
Moreover, these bijections preserve
the equivalences.
\end{maintheorem}

Viewing the classification of 
topological doubles
\cite{Stolin_Zelmanov_Montaner}
through the prism of the above-mentioned
connections, we obtain an interesting
result on the side of formal
\(r\)-matrices.

\begin{maincorollary}
If the series 
\[
\frac{s(y)\Omega}{x-y} + g(x,y) \in
    (\fg \ot \fg)(\!(x)\!)[\![y]\!],
\]
with \( s \in F[\![y]\!] \)
and \( g \in (\fg \ot \fg)[\![x,y]\!], \)
solves CYBE, then \(s = 0\) or \(s\) has a zero in \(y = 0\) of multiplicity either \( 0, 1 \) or \(2\).
Furthermore, after an appropriate equivalence transformation 
\( s \in \{0, 1, y, y^2 \} \).
The existence of such a transformation in
case of multiplicity \(2 \)
is equivalent to the fact that
the coefficient of \(y^3\) in
\(s(y) \) vanishes.
\end{maincorollary}

            
            

\Cref{sec:commensurable_twists}
is about a natural subclass
of topological twists,
called commensurable twists.
These are topological twists
inside \( (\fg \ot \fg)[x,y] \).
Commensurable twists turn out to be
in bijection with rational, quasi-trigonometric and quasi-rational
\( r \)-matrices. The classifications
of these \(r \)-matrices are already known
and presented in this section.

In \Cref{sec:classification}
we state the main classification
results for topological Lie
bialgebra structures on \( \fg[\![x]\!] \). More precisely,
we show that up to a certain notion of equivalence
there are six classes of formal \(r\)-matrices:
degenerate (i.e.\ \(s(y) = 0\)), elliptic,
trigonometric, rational, quasi-trigonometric and quasi-rational.
In the case \( F = \bC \) the
descriptions of these classes become explicit.

\begin{maintheorem}
Let \(r \) be a  formal \( r \)-matrix
corresponding to a Lie bialgebra structure with the double \( \fD_i \),
\( i \in \{1,2,3\} \).
Then
\begin{itemize}[align=parleft, left=0pt..3em]
    \item[\(i \nolinebreak = \nolinebreak 1\colon\)] \(r\) is equivalent to a Taylor series 
    expansion at \( y = 0\) of either
    an elliptic or a trigonometric
    or a rational \(r\)-matrix;
    
    \item[\(i \nolinebreak = \nolinebreak 2\colon\)] \(r\) is equivalent
    to a quasi-trigonometric \(r\)-matrix;
    
    \item[\(i \nolinebreak = \nolinebreak 3\colon\)] \(r\) is equivalent
    to a quasi-rational \(r\)-matrix.
\end{itemize}
\end{maintheorem}

The algebro-geometric
proof of the classification is 
given in \cref{sec:proof}.
The case \( i = 1 \) was already thoroughly studied in
\cite{abedin2021geometrization}.
We present a sketch of its proof given there.
The cases \( i = 2 \) and \( i = 3 \)
are new and described in full detail.

We use many different notions of equivalence throughout the paper.
\Cref{sec:appendix} collects
all of these notions and can be used
as a cheat sheet while reading the work.

\subsection*{Acknowledgments}
The work of the first named author was supported by the DFG project Bu--1866/5--1. 

\section{Topological Lie bialgebras}\label{sec:Top_Lie_bialg}
We start with a brief review of linearly topologized vector spaces
and then introduce the notion of a topological Lie bialgebra. 
For a detailed exposition on linearly topologized vector spaces
we refer to \cite{Kothe}.

Let \( F \) be a field of characteristic \( 0 \) equipped with the discrete topology.
All vector spaces are considered over that field.

\subsection{Linearly topologized vector spaces}
A vector space \( V \) is said to be \emph{linearly topologized}
if it is endowed with a topology \( \fT \) satisfying the following
conditions
\begin{enumerate}
\item Vector addition \( V \times V \to V \) and
scalar multiplication \( F \times V \to V \)
are (jointly) continuous;
\item \( (V, \fT) \) is a Hausdorff topological space;
\item \( \fT \) is translation invariant and 
\( 0 \in V \) admits a fundamental system of neighbourhoods consisting
of subspaces of \( V \).
\end{enumerate}
Such a topology \( \fT \) will be referred to as 
\emph{linear topology}.

\begin{example}
Any vector space equipped with the discrete topology
is a linearly topologized vector space.
The converse is true for finite-dimensional vector spaces:\
any finite-dimensional linearly topologized vector space
is necessarily discrete.
\end{example}

We refer to complete linearly topologized vector spaces
simply as \emph{linearly complete}.
Any linearly topologized vector space \( (V, \fT) \)
can be embedded into a smallest linearly complete
vector space \( (\widehat{V}, \widehat{\fT}) \).
This space is unique up to topological isomorphism
and it is called \emph{the completion} of \( V \).
It is common to represent the completion \( \widehat{V} \)
as the following limit
\begin{equation}\label{eq:completion_as_limit}
\begin{aligned}
  \widehat{V} = \lim_{i \in I} V / N_i,
\end{aligned}
\end{equation}
where \( \{ N_i \}_{i \in I} \subseteq{\fT} \) is
a fundamental system of neighbourhoods of \( 0 \in V \)
and each quotient \( V / N_i \) is equipped with the
discrete topology.
The set \( \{ \textnormal{Cl}_{\widehat{V}}(N_i) \}_{i \in I} \)
forms a fundamental system of neighbourhoods of \( 0 \in \widehat{V} \). Here \(\textnormal{Cl}_{\widehat{V}}(N_i)\) denotes the closure of \(N_i\) in \(\widehat{V}\).

\begin{example}
Any discrete vector space \( V \) is automatically complete. Let us equip the algebraic dual \( \dual{V} \) with the weak topology, i.e. the linear topology with \( \{ U^\perp \mid U \subseteq V \text{ is finite-dimensional} \} \) as a fundamental system of neighbourhoods of \( 0 \in \dual{V} \), where \(U^\bot = \{f \in \dual{V}\mid f|_U = 0\}\).
Then \(\dual{V}\) is also
a linearly complete vector space.
In particular, this implies that both \( F[x] \) and
\( F[\![x]\!] = \dual{F[x]} \) are linearly complete spaces.
The weak topology on \( F[\![x]\!] \) is precisely the \( (x) \)-adic topology.
\end{example}

\subsection{Linear tensor product topologies}\label{sec:tensor_prod_topology}
Let  \( V \) and \( W \) be linearly topologized vector spaces.
Their usual tensor product \( V \ot W \) has no canonical linear topology
on it. 

In this work we are interested in the
finest linear topology on \( V \ot W \) making
the tensor product map \( \ot \colon V \times W \to V \ot W \) continuous.
More explicitly, we declare a subspace
\( U \subseteq V \ot W \) open if it satisfies the following three conditions
\begin{itemize}
\item There are open subspaces \( V_0 \subseteq V \) and \( W_0 \subseteq W \)
such that \( V_0 \ot W_0 \subseteq U \);
\item For every \( v \in V \) there exists an open subspace \( W_0 \subseteq W \)
such that \( v \ot W_0 \subseteq U \);
\item For every \( w \in W \) there exists an open subspace   \( V_0 \subseteq V \)
such that \( V_0 \ot w \subseteq U \).
\end{itemize}
The corresponding topology is referred to as \emph{projective tensor product topology} or \emph{the \( \fot \)-topology}.
The completion of \( V \ot W \) with respect to it is denoted by
\( V \fot W \).
The universal property of the \( \fot \)-topology immediately
implies that the unique factorization \( \bar{f} \colon V \ot W \to X \)
of any continuous bilinear map
\( f \colon V \times W \to X  \) is again continuous.

Even though we have an explicit description of open sets
in the \( \fot \)-topology, it is not easy to work with them
explicitly. For that reason we introduce another auxiliary
tensor product topology.
It is the finest linear topology
on \( V \ot W \) making the tensor product map
\( \ot \colon V \times W \to V \ot W \) uniformly continuous.
We call it \emph{the \( \uot \)-topology}. 
Suppose that \( \{ N_i \}_{i \in I} \) and \( \{ M_j \}_{j \in J} \)
are fundamental systems of neighbourhoods of 0 in \( V \) and \( W \) respectively.
Then the \( \uot \)-topology is defined by declaring
\begin{equation}\label{eq:0_neighbourhoods_in_uot} 
\{ N_i \ot W + V \ot M_j \}_{i \in I, j \in J}
\end{equation}
to be the fundamental system of neighbourhoods of \( 0 \in V \ot W \).
As with the first topology, we write \( V \uot W \) for the completion of \( V \ot W \)
with respect to this topology.
Using the representation \cref{eq:completion_as_limit}
we get the identity
\begin{equation}\label{eq:Second_topology_completion_limit}
\begin{aligned}
  V \uot W = \lim_{\substack{i \in I \\ j \in J}} V / N_i \ot W / M_j.
\end{aligned}
\end{equation}

\begin{remark}
It is clear that the first topology is finer than the
second one. 
However, when both \( V \) and \( W \) are
discrete we get the homeomorphisms
\begin{equation}\label{eq:two_topologies_coincide}
V \fot W \cong V \uot W \ \text{ and } \ \dual{V} \fot \dual{W} \cong \dual{V} \uot \dual{W} \cong \dual{(V \ot W)},
\end{equation}
where the algebraic duals are equipped with the weak topology;
see e.g.\ \cite[Lemma 24.17 and Corollary 24.25]{Bergman_Hausknecht}. 
Later we work primarily with the \( \fot \)-topology and the
isomorphisms \cref{eq:two_topologies_coincide} allow us to use the explicit 
description \cref{eq:0_neighbourhoods_in_uot} of \( 0 \)-neighbourhoods.
\end{remark}

\subsection{Topological Lie bialgebras}\label{subsec:Top_Lie_bialg}
From now on all tensor products of linearly topologized vector spaces
are taken by default with the \( \fot \)-topology.

A \emph{topological Lie coalgebra} is a pair
\( (L, \delta) \), where \( L \)
is a linearly topologized vector space and
\( \delta \colon L \to L \fot L \) is a continuous linear map
satisfying the conditions
\begin{equation}\label{eq:toplogical_Lie_coalgebra_axioms}
\begin{aligned}
  \delta(x) + \overline{\tau} \delta(x) = 0 \ \text{ and } \
  \overline{\textnormal{Alt}}((\delta \fot 1)\delta(x)) = 0,
\end{aligned}
\end{equation}
where 
\begin{equation}\label{eq:extenstions_of_key_maps}
\begin{aligned}
  \overline{\tau} &\colon L \fot L \to L \fot L, \\
  \overline{\textnormal{Alt}} &\colon L \fot L \fot L \to L \fot L \fot L, \\
  \delta \fot 1 &\colon L \fot L \to L \fot L \fot L
\end{aligned}
\end{equation}
stand for the unique continuous extensions of the continuous linear maps
\begin{equation}\label{eq:key_maps}
\begin{aligned}
  \tau &\colon L \ot L \to L \ot L, && x \ot y \mapsto y \ot x, \\
  \textnormal{Alt} &\colon L \ot L \ot L \to L \ot L \ot L, && x \ot y \ot z \mapsto x \ot y \ot z + y \ot z \ot x + z \ot x \ot y, \\
  \delta \ot 1 &\colon L \ot L \to (L \fot L) \ot L, && x \ot y \mapsto \delta(x) \ot y.
\end{aligned}
\end{equation}
We call a continuous \( F \)-linear map \( \varphi \colon L_1 \to L_2 \)
a \emph{morphism of topological Lie coalgebras} if
\begin{equation}\label{eq:top_coalgebra_morphism}
  (\varphi \fot \varphi)\delta_1 = \delta_2 \varphi.
\end{equation}

In a similar vein we define a \emph{topological Lie algebra}
as a linearly topologized vector space \( L \) together with
a continuous Lie bracket \( [\cdot \, ,\cdot] \colon L \times L \to L \).
A \emph{morphism between topological Lie algebras} is a
continuous Lie algebra morphism.

\begin{remark}
Note that by definition of the \( \fot \)-topology
we have an isomorphism between the space of continuous
bilinear maps \( L \times L \to L \) and
the space of continuous linear maps
\( L \ot L \to L \).
Moreover, when \( L \) is complete
these spaces are also isomorphic to the space of
continuous linear maps \( L \fot L \to L \).
\end{remark}

A \emph{topological Lie bialgebra} consists of the
following datum
\begin{itemize}
\item A topological Lie algebra \( (L, [\cdot \, ,\cdot]) \);
\item A topological Lie coalgebra \( (L, \delta) \);
\item A linear topology on the space \( L' \) of continuous linear functionals \( L \to F \);
\end{itemize}
And this datum is subject to the following conditions
\begin{enumerate}
\item The restriction of the dual map \( \delta' \colon (L \fot L)' \to L'  \)
to the subset \( L' \ot L' \subseteq (L \ot L)' = (L \fot L)' \) is again continuous;
\item The following compatibility condition holds
\[ \delta([x,y]) = (\ad_x \fot 1 + 1 \fot \ad_x) \delta(y) -
                  (\ad_y \fot 1 + 1 \fot \ad_y) \delta(x)
                  \qquad
               \]
for all \(x,y\in L\).
\end{enumerate}
When there is no ambiguity we simply write \( (L, \delta) \) to denote a topological Lie bialgebra. 
A  map \( \varphi \) between two topological Lie bialgebras is a \emph{topological Lie bialgebra morphism}
if it is a morphism of both topological Lie algebra and topological Lie coalgebra structures such that \( \varphi' \)
is continuous.

\begin{remark}\label{rem:bracket_notation}
In the following we use the notation \([x \fot 1,t]\coloneqq (\ad_x \fot 1)t\) and \([1 \fot x,t] \coloneqq (1 \fot \ad_x)t\) for all \(x\in L\) and \(t \in L\fot L\). For instance, the compatibility condition in (2) then reads 
\[ \delta([x,y]) = [x \fot 1 + 1 \fot x, \delta(y)] -
                  [y \fot 1 + 1 \fot y, \delta(x)]
                  \qquad
              \]
              for all \(x,y \in L\).
\end{remark}

\begin{remark}%
\label{rem:scaling_bialgebras}
If \( \delta \colon L \to L \fot L \) is a topological Lie bialgebra structure
on \( L \), then so is \( \xi \delta \) for any \( \xi \in F \).
It is natural to not distinguish structures that differ only by a non-zero scalar multiple.
Therefore, we call two topological Lie bialgebras \( (L_1, \delta_1) \) and
\( (L_2, \delta_2) \) \emph{equivalent} if there is a constant \( \xi \in F^{\times} \)
such that \( (L_1, \delta_1) \) is isomorphic to \( (L_2, \xi \delta_2) \).
We adopt the notation \[ (L_1, \delta_1) \sim (L_2, \delta_2)  \] to denote
equivalent topological Lie bialgebras.
\end{remark}

\subsection{Topological Manin triples and doubles}\label{sec:top_Man_triples}
A triple \( (L, L_+, L_-) \) is a called a 
\emph{topological Manin triple} if
\( L \) is a topological Lie algebra
equipped with a separately continuous invariant non-degenerate
symmetric bilinear form \( B \colon L \times L \to F \)
such that
\( L_{\pm} \) are isotropic topological Lie subalgebras of \( L \)
and \( L = L_+ \add L_- \).

We say that two topological Manin triples \( (L, L_+, L_-) \)
and \( (M, M_+, M_-) \) are \emph{isomorphic} if there is
an isomorphism of topological Lie algebras 
\( \varphi \colon L \to M \) such that
\begin{equation}
\varphi(L_\pm) = M_\pm 
\
\text{ and }
\
B_L(x,y) = B_M(\varphi(x), \varphi(y))
\qquad
               \forall x,y \in L.
\end{equation}

Let \( (L, \delta) \) be a topological Lie bialgebra.
Consider the triple \( (L \add L', L, L') \) with the form
\begin{equation}\label{eq:form_topological_double}
\begin{aligned}
  B(x + f, y + g) \coloneqq f(y) + g(x) \qquad \forall x,y \in L, \ \forall f,g \in L',
\end{aligned}
\end{equation}
and the bracket
\begin{equation}\label{eq:Lie_bracket_topological_double}
\begin{aligned}
  [x,f] \coloneqq -f\circ\ad_x + (f \fot 1)(\delta(x)) \qquad \forall x \in L, \ \forall f \in L'.
\end{aligned}
\end{equation}
In the classical (i.e.\ when all vector spaces are discrete) finite-dimensional
framework the construction above gives a bijection between Lie bialgebras
and Manin triples.
If we allow \( L \) to be infinite-dimensional, the triple \( (L \add L', L, L') \) is still
a Manin triple, but the converse direction is no longer true: not every Manin triple is of this form.
Passing further from the classical framework to the topological one
we entirely lose this connection:
the triple \( (L \add L', L, L')  \) is not a topological Manin triple in general.
The first subtlety is that the bracket
\cref{eq:Lie_bracket_topological_double} may not be well-defined if
\( L \) is not complete. The second problem is that it may not be continuous.
However, in the most important cases for us the triple \( (L \add L', L, L') \)
is indeed a topological Manin triple. 
The space \( L \add L' \) with the topological Lie algebra structure
and the form mentioned above is then denoted by
\( \fD(L, \delta) \) or simply by \( \fD \) and called
\emph{the topological double} of \( (L, \delta) \).

Suppose \( (L, L_+, L_-) \) is a topological
Manin triple and we have a topology on \( L'_+ \) such that
\begin{enumerate}
\item The dual map
\( [\cdot \, ,\cdot]'_{L_-} \colon L'_- \to (L_- \ot L_-)' \) 
restricts to a continuous map
\( \delta \colon L_+ \to L_+ \fot L_+ \), where \( L_+ \) is viewed
as a subset of \( L'_- \) through the bilinear form \( B \) on \( L \);
\item The restriction of \( \delta' \) to \( L'_+ \ot L'_+ \) is continuous,
\end{enumerate}
then \( (L_+, \delta) \) is a topological Lie bialgebra
\emph{defined} by the triple \( (L, L_+, L_-) \). 

\begin{remark}%
\label{rem:delta_as_dual_of_cobracket}
As in the classical case the first condition is equivalent to the following:
there exists a continuous linear map
\( \delta \colon  L_+ \to L_+ \fot L_+ \) such that
\( B^{\fot 2}(\delta(x), y \ot z) = B(x, [y,z]) \) for all \( x \in L_+ \) and \( y,z \in L_- \).
Here \( B^{\fot 2} \) is the unique continuous extension
\( L_+ \fot L_+ \to F \) of the continuous map
\( B(\cdot \, , y \ot z) = B(\cdot \, , y)B(\cdot\, , z) \colon L_+ \ot L_+ \to F \).
\end{remark}

\begin{remark}
Suppose \( (L, L_+, L_-) \) and \( (M, M_+, M_-) \) are two isomorphic topological
Manin triples.
If \( (L, L_+, L_-) \) together with a topology on \( L'_+ \) defines a topological
Lie bialgebra \( (L_+, \delta_L) \), then \( (M, M_+, M_-) \) together with the induced
topology on \( M'_+ \) defines a topological Lie bialgeba \( (M_+, \delta_M) \).
Moreover, if \( \varphi \colon L \to M \) is the isomorphism between the two Manin triples, then
its restriction \( \varphi |_{L_+} \colon L_+ \to M_+ \) is an isomorphism
of topological Lie bialgebras.
\end{remark}

\begin{remark}\label{rem:equivalence_Manin_triples_top_algebras}
Note that if a Manin triple \( (L, L_+, L_-) \) with a form \( B \)
defines a Lie bialgebra structure \( (L_+, \delta) \),
then the same triple \( (L, L_+, L_-) \) with the form \( \xi B \),
\( \xi \in F^{\times} \), defines the Lie bialgebra structure
\( (L_+, \xi^{-1} \delta) \).
Since we have identified Lie bialgebras that differ by a non-zero scalar multiple
(see \cref{rem:scaling_bialgebras}) it is natural to identify
Manin triples whose forms differ by a non-zero scalar multiple as well.
Taking that into account, we call two Manin triples  \( (L, L_+, L_-) \) and
\( (M, M_+, M_-) \) with forms \( B_L \) and \( B_M \) respectively
\emph{equivalent} if there is \( \xi \in F^{\times} \)
such that \( (L, L_+, L_-) \) is isomorphic to
\( (M, M_+, M_-) \) with the form \( \xi B_M \).
Such an equivalence of Manin triples is denoted by  \( (L, L_+, L_-) \sim  (M, M_+, M_-)\).


\end{remark}

\section{Classification of topological doubles on \(\fg[\![x]\!]\)}%
\label{sec:review_MSZ}
In this section we review the classification of topological doubles on \(\fg[\![x]\!]\) achieved in \cite{Stolin_Zelmanov_Montaner}
within our framework of topological Lie bialgebras. 
One of the pivotal steps in the derivation of this result is the fact that these doubles are of the form \(\fg \ot A\) for a so-called trace extension \(A\) of \(F[\![x]\!]\). 
However, the proof presented in  \cite{Stolin_Zelmanov_Montaner}
is not applicable for
\(\fg = \mathfrak{sl}(n,F)\) with 
\( n \ge 3 \). 
In this section we refine the proof
to include all simple Lie algebras
as well as
show that the result is still valid when
\( F[\![x]\!] \) is replaced by any
associative commutative reduced unital \( F \)-algebra.

\subsection{Extension of scalars of simple Lie algebras}\label{sec:extension_of_scalars}
Throughout this section \(F\) denotes 
an arbitrary field of characteristic 0. 
Let us recall that a finite-dimensional (not necessarily associative or Lie) algebra \(\mathfrak{a}\) over \(F\) is called \emph{central} if its \emph{centroid} 
\begin{equation}
    \Gamma_F(\mathfrak{a}) \coloneqq \{f \in \textnormal{End}_F(\mathfrak{a})\mid af(b) = f(ab) = f(a)b,\textnormal{ for all }a,b\in \mathfrak{a}\}
\end{equation}
coincides with scalar multiples
of the identity map. 
If \(\mathfrak{a}\) is simple, 
the centroid is a finite field extension of \(\mathfrak{a}\). In particular, any finite-dimensional simple \(F\)-algebra is central if \(F\) is algebraically closed.

\begin{lemma}%
\label{lem:Ann(g)V=0}
 Let
 \( \fg \) be a finite-dimensional 
 central simple Lie algebra over \( F \).
 Define
\begin{equation*}
 \textnormal{Ann}(\fg) \coloneqq \{ x \in U(\fg) \mid x \cdot \fg = 0 \},
\end{equation*}
 where \( U(\fg) \) is the universal
 enveloping algebra of \( \fg \).
 If \( V \) is a finite-dimensional irreducible
 \( \fg \)-module and \( \textnormal{Ann}(\fg) V = 0\), then
 \( V \cong \fg \) as \( \fg \)-modules.
\end{lemma}
\begin{proof}
Let \( \{ I_\alpha \}_{\alpha = 1}^d  \) be a basis of \( \fg \) and
\( O_{U(g)}(I_\alpha) \) be the orbit of \( I_\alpha \)
under the action of \( U(\fg) \). Since \( \fg \) is simple
we have \( O_{U(g)}(I_\alpha) =  \fg \).
The kernel of the surjective \( U(\fg) \)-module homomorphism 
\begin{equation}
\begin{aligned}
  U(\fg) &\longrightarrow O_{U(g)}(I_\alpha) = \fg \\
  x &\longmapsto x \cdot I_\alpha
\end{aligned}
\end{equation}
is exactly \( \textnormal{Ann}(I_\alpha) \coloneqq \{ x \in U(\fg) \mid x \cdot I_\alpha = 0 \} \).
Therefore, \( U(\fg) / \textnormal{Ann}(I_\alpha) \cong \fg \) and hence
\( \textnormal{Ann}(I_\alpha) \) is a maximal left ideal of \( U(\fg) \).
A similar argument shows
that if \( \{ v_i \}_{i = 1}^{\ell} \) 
is a basis of \( V \), then
\( \textnormal{Ann}(v_i) \)
are maximal left ideals of \( U(\fg) \)
and \( U(\fg) / \textnormal{Ann}(v_i) \cong V \).

Consider the \( U(\fg) \)-module homomorphism
\begin{equation}
\begin{aligned}
    U(\fg) &\longrightarrow \textnormal{End}_{F}(\fg) \\
    x &\longmapsto x \cdot (-).
\end{aligned}
\end{equation}
By Schur's lemma, Jacobson density theorem and the centrality of \(\fg\)
this map is surjective with kernel
\(\textnormal{Ann}(\fg) \). This means
that we have an isomorphism of
\( U(\fg) \)-modules
\( U(\fg) / \textnormal{Ann}(\fg) \cong
\textnormal{End}_{F}(\fg)\) 
and, consequently,
\( \textnormal{Ann}(\fg)  \)
is a maximal
two-sided ideal.
Moreover, \( \textnormal{Ann}(V) \)
is also a two-sided ideal and by our assumption we have
the following inclusions
\begin{equation}
    \textnormal{Ann}(\fg) \subseteq 
    \textnormal{Ann}(V) \subsetneq U(\fg)
\end{equation}
that imply the equality
\( \textnormal{Ann}(\fg) = \textnormal{Ann}(V) \).

Combining all previous observations
we obtain the following relations
between \( \fg \)-modules
\begin{equation*}
  \prod_{\alpha = 1}^d \fg \cong 
  \prod_{\alpha = 1}^d U(\fg) / \textnormal{Ann}(I_\alpha) \cong
  U(\fg) / \textnormal{Ann}(\fg) =
  U(\fg) / \textnormal{Ann}(V) \subseteq
  \prod_{i = 1}^\ell U(\fg) / \textnormal{Ann}(v_i) \cong
  \prod_{i = 1}^\ell V,
\end{equation*}
which now imply the desired isomorphism 
\( V \cong \fg \).
\end{proof}

\begin{theorem}%
\label{thm:form_of_doubles}
  Let \( \fg \) be a central simple Lie algebra over \(F\) and
  \( A \) be a reduced unital associative commutative \( F \)-algebra.
  Equip the tensor product \( \fg \ot A \)
  with the Lie algebra bracket
  \begin{equation*}
    [a \ot f, b \ot g] = [a,b] \ot fg.    
  \end{equation*}
  Let \(L\) be a Lie algebra equipped with a non-degenerate symmetric invariant
  bilinear form \(B\) such that \(\fg \otimes A \subseteq L\) 
  is a coisotropic subalgebra. Then, as a \( \fg \)-module,
  \(L \cong \fg \otimes \widetilde{A}\) for a unital associative commutative
  algebra extension \( \widetilde{A} \supseteq A \).
\end{theorem}

\begin{proof}
The action of \( \fg \cong \fg \ot 1 \subseteq \fg \ot A \) on \( L \)
extends uniquely to an action of \( U(\fg) \) on \( L \).
Define
\(
    P \coloneqq \textnormal{Ann}(\fg) = \{ x \in U(\fg) \mid x \cdot \fg = 0 \}.
\)
Let 
\begin{equation}
\begin{aligned}
i \colon U(\fg) &\longrightarrow U(\fg) \\
x_1\dots x_n &\longmapsto (-1)^n x_n \dots x_1
\end{aligned}
\end{equation}
be the antipode map and
\( \kf \) be the Killing form on \( \fg \).
By invariance of Killing form we have
\(
    \kf(i(P) \cdot \fg, \fg) = 
    \kf(\fg, P \cdot \fg) = 0,
\)
implying that \( i(P) \subseteq P \).
Similarly, by the invariance of \(B\),
we obtain
\begin{equation}
B(P \cdot L, \fg \ot A) = B(L, i(P) \cdot (\fg \ot A)) = 0.
\end{equation}
Therefore, \( P \cdot L \subseteq (\fg \ot A)^\bot \subseteq \fg \otimes A \) and \( P^2 \cdot L = 0 \).
Since \( U(\fg)\) is Notherian we can write
\begin{equation}
 P = U(\fg)p_1 + \dots + U(\fg)p_m.   
\end{equation}
The ideal \( P \) can be viewed as
the kernel of \( U(\fg) \)-module
homomorphism
\( U(\fg) \to \textnormal{End}_F(\fg) \) giving
\( \text{dim}(U(\fg)/P) = \ell < \infty \). Write 
\(
    U(\fg) = P \add \Span_F\{u_1, \dots, u_\ell\}
\)
as a vector space.
Then for any
\( u \in U(\fg) \) exists 
\(x_i, x_{ij} \in \textnormal{U}(\fg)\) and 
\(\mu_t,\lambda_k \in F\) such that 
\[
u = \sum_{i = 1}^m x_i p_i+\sum_{k = 1}^\ell \lambda_k u_k 
\, \textnormal{ and } \, 
x_i = \sum_{j = 1}^m x_{ij} p_j + \sum_{t = 1}^\ell\mu_t u_t.
\]
Thus, for all  \( d \in L \)
we have
\begin{equation}%
\label{eq:finite_orbit}
\begin{aligned}
    u \cdot d &=
    \left(\sum_{i=1}^m x_i p_i
    + \sum_{k=1}^\ell \lambda_k u_k \right) \cdot d \\
    &= \left( \sum_{i=1}^m 
    \left\{
    \sum_{j=1}^{m} x_{ij} p_j
    + \sum_{t=1}^\ell \mu_t u_t
    \right\} p_i
    + \sum_{k=1}^\ell \lambda_k u_k \right) \cdot d \\
    &= 
    \underbrace{\sum_{i=1}^m \sum_{j=1}^m (x_{ij} p_j p_i) \cdot d}_{=0}
    + \sum_{i=1}^m \sum_{t=1}^\ell \mu_t (u_t p_i)\cdot d
    + \sum_{k=1}^\ell \lambda_k u_k \cdot d.
\end{aligned}
\end{equation}
Taking other \( u \in U(\fg) \)
amounts to changing
constants \( \mu_t, \lambda_k \in F\)
in the very last line of
\cref{eq:finite_orbit}.
For that reason we have the inclusion
\begin{equation}
    U(\fg) \cdot d \subseteq \Span_F\{
    u_k \cdot d, \, (u_k p_i) \cdot d
    \mid k \in \{1, \dots, \ell\}, 
    i \in \{1, \dots, m \} \}
\end{equation}
for any \( d \in L \). In particular, \(\textnormal{U}(\fg)\cdot d\) is a finite-dimensional \(\textnormal{U}(\fg)\)-module and thus semi-simple.  If we now consider \( I = \{ d \in L \mid \textnormal{U}(\fg)\cdot d \textnormal{ is simple}\} \) we can write
\begin{equation}%
\label{eq:L_as_sum}
    L = \sum_{d \in I} U(\fg) \cdot d.
\end{equation} 
Since \( P \cdot (U(\fg) \cdot d) \) 
is a \( U(\fg) \)-submodule of 
\( U(\fg) \cdot d \), 
it is either \( 0 \)
or \( U(\fg) \cdot d \) itself for all \(d \in I\).
The latter equality is impossible,
because otherwise
\(0 = P^2 \cdot (U(\fg) \cdot d)
= U(\fg) \cdot d \neq 0\).
Thus
\( P \cdot (U(\fg) \cdot d) = 0\)
and \cref{lem:Ann(g)V=0} gives
isomorphisms of \( \fg \)-modules
\( U(\fg) \cdot d \cong \fg \) for all \(d \in I\).
We call two indices \(d,d' \in I\)
equivalent \( d \sim d' \)
if \(U(\fg) \cdot d = U(\fg) \cdot d' \).
Removing duplicates from \cref{eq:L_as_sum}
we get
\begin{equation}
    L = \bigoplus_{[d] \in I/\sim} U(\fg) \cdot d \cong \bigoplus_{[d] \in I/\sim} \fg \cong \fg \ot \widetilde{A},
\end{equation} 
for some vector space \( \widetilde{A} \).
Moreover, by choosing appropriate representatives
in each equivalence class
we can, without loss of generality, assume
\( A \subseteq \widetilde{A} \).

The next step is to equip
\( \widetilde{A} \) with a unital
commutative associative algebra structure such that
\( A \) becomes a subalgebra of
\(\widetilde{A}\).
When \( \fg \) is not of type 
\( \mathfrak{sl}(n, F), n \geq 3 \),
we have \( \textnormal{Hom}_{\fg\textnormal{-Mod}}(\fg \ot \nolinebreak \fg, \fg) = \Span_F\{ [\cdot \, ,\cdot] \} \). Repeating the proof
of \cite[Proposition 2.2]{Zelmanov_Benkart}
we obtain the desired algebra structure.
Assume 
\( \fg = \mathfrak{sl}(n, F), n \geq 3  \).
The space
\( \Hom_{\fg\textnormal{-Mod}}(\fg \ot \fg, \fg) \) is now generated
by two \( \fg \)-module homomorphisms, namely the Lie bracket
\( [\cdot \, , \cdot] \) and the map
\[ 
  a \ot b \mapsto a \circ b \coloneqq ab + ba - \frac{\tr(ab+ba)}{n}I_{n}.
\]
The Lie algebra structure on \( L \)
induces a unital algebra structure
on \( \widetilde{A} \):
the multiplication is now given by
\begin{equation}
\begin{aligned}
[a \ot f, b \ot g] &= [a,b] \ot \frac{1}{2}(fg + gf) + (a \circ b) \ot \frac{1}{2}(fg - gf).
\end{aligned}
\end{equation} 
Furthermore, the algebra \( \widetilde{A} \)
is alternative for \( n = 3 \) and associative for \( n > 3 \); for details see e.g.\ \cite{Allison_Benkart_Yun}. 
The following argument requires only the alternativity property and hence works
in both cases.

Let us denote the induced bilinear form
on \( \fg \ot \widetilde{A} \)
with the same letter \( B \).
Define the linear functional \( \tOp \colon \widetilde{A} \to F \) by
\[ \textnormal{t}(f) \coloneqq B(e_{12} \ot 1, e_{21} \ot f), \]
where \( e_{ij} \) is the \( (n \times n) \)-matrix having
\( 1 \) in position \( (i,j) \) as its only non-zero entry. Observe that \([e_{ij},e_{jk}] = e_{ik} = e_{ij}\circ e_{jk}\) for all pairwise different \(1 \le i,j,k \le n\).
By the invariance and the symmetry of \( B \) we obtain
\[ \tOp(fg) = B(e_{ij} \ot f, e_{ji} \ot g) = \tOp(gf)\]
for all \( 1 \le i \neq j \le n \) and \(f,g \in \widetilde{A}\).
We put
\begin{equation}
    A^{\perp_{\tOp}} \coloneqq \{ 
    f \in \widetilde{A} \mid 
    \tOp(Af) = 0
    \}.
\end{equation}
Since \( \fg \ot A \) is a coisotropic
subalgebra of \( L \) we have
\( A^{\perp_{\tOp}} \subseteq A\).
Again, using the invariance and the symmetry
of \( B \) we get the following chain
of identities for any
\( p,q \in A \) 
and \( f \in \widetilde{A} \)
\begin{equation}
    \tOp(p(qf)) = \tOp((pq)f) = 
    \tOp((qp)f) = \tOp(f(qp)) =
    \tOp((fq)p) = \tOp(p(fq)).
\end{equation}
In particular, this means
\( \tOp(p[q,f]) = 0 \) for all
\( p,q \in A \) and
\( f \in \widetilde{A} \).
Therefore, 
\( [A, \widetilde{A}] \subseteq A^{\perp_{\tOp}} \subseteq A \).

Now we proceed by showing that \( A \)
 lies in the center of \( \widetilde{A} \).
 For that observe that for any
 two elements \( q \in A \)
 and \( f \in \widetilde{A} \)
 we have
 \begin{equation}
 \begin{aligned}
 0 &= [q, [q, f^2]] \\
&= [q, qf^2 - f^2q + fqf - fqf] \\
&= [q, [q,f]f + f[q,f]] \\
&= 2[q,f]^2.
 \end{aligned}
 \end{equation}
Here we implicitly used Artin's theorem
stating that any subalgebra of
an alternative algebra generated by
two elements is associative.
By assumption \( A \) 
has no non-trivial
nilpotent elements implying that
\([A, \widetilde{A}] = 0\).

Finally, since \(\textnormal{char}(F) = 0\)
the commutative center of 
\( \widetilde{A} \) lies in the associative
center of \( \widetilde{A} \).
Therefore, for all \( p \in A \)
and \( f,g \in \widetilde{A} \)
we have
\( \tOp(p[f,g]) = \tOp([pf,g]) = 0 \)
showing that \( [\widetilde{A}, \widetilde{A}] \subseteq A \subseteq Z(\widetilde{A}). \)
Commutativity of \( \widetilde{A} \)
now follows from the equality
\begin{equation}
    [f,g]^2 = [[f,g]f, g] = [[f,gf], g] = 0.
\end{equation}
The proof is now complete because
any unital commutative alternative
\( F \)-algebra with \( \textnormal{char}(F)  \neq 3\) is automatically associative.
\end{proof}

To classify topological Lie bialgebras
or doubles it is important
to understand how isomoprhisms
of a Lie algebra of the form
\( \fg \ot A \) may look like.
The following theorem gives
us control over such maps.





\begin{theorem}%
\label{thm:splitting_aut}
 Let \( \mathfrak{a} \) be a finite-dimensional
 central simple (not necessarily associative or Lie) \(F\)-algebra and \( A \) be
 a unital associative commutative \( F \)-algebra. Then 
 any \( F \)-algebra automorphism \( \varphi \) of \( \mathfrak{a} \ot A \)
 is a composition of an \( F \)-algebra automorphism of \( A \) and 
 an \( A \)-algebra automorphism of \( \mathfrak{a} \ot A \). More precisely, we have \(\textnormal{Aut}_{F\textnormal{-Alg}}(\mathfrak{a} \otimes A) = \textnormal{Aut}_{F\textnormal{-Alg}}(A) \ltimes \textnormal{Aut}_{A\textnormal{-Alg}}(\mathfrak{a}\otimes A)\).
 \end{theorem}
\begin{proof}
 For any element \( a \in A \) we define
 \( \gamma_a \in \textnormal{End}_{A\textnormal{-Mod}}(\mathfrak{a} \ot A) \) by
 \( \gamma_a(x \ot b) = x \ot ab \). The set of all such endomorphism
 \( \tilde{A} \coloneqq \{ \gamma_a \mid a \in A \} \)
 is an \( F \)-algebra isomorphic to \( A \). We now show that
 for any \( \varphi \in \textnormal{Aut}_{F\textnormal{-Alg}} (\mathfrak{a} \ot A) \) we have
 \begin{equation}\label{eq:phi_A_phi}
   \varphi^{-1} \tilde{A} \varphi = \tilde{A}.
 \end{equation}
 Let \( n \coloneqq \dim_F(\mathfrak{a}) \), chose a basis of \(\mathfrak{a}\) so that we can identify \(\textnormal{End}_{F\textnormal{-Mod}}(\mathfrak{a})\) with the space of \((n\times n)\)-matrices \(\textnormal{M}_n(F)\) and let
 \( M_F(\mathfrak{a}) \) be the multiplication algebra of \( \mathfrak{a} \),
 i.e.\ the subalgebra of \(M_n(F) \) 
 generated by left and right multiplications in \( \mathfrak{a} \).
 Since \( \mathfrak{a} \) is central, the
 centroid \( \Gamma_F(\mathfrak{a}) \) of \( \mathfrak{a} \) is precisely the
 field \( F \) itself. Then by \cite[Chapter X, Theorem 4]{Jacobson_Lie_algebras}
 we have \(  M_F(\mathfrak{a}) = M_n(F) \) and, consequently,
 \( M_A(\mathfrak{a} \ot A) = M_n(A) \).
 Let \( c_n(X_1, X_2, \dots, X_N) \) be a multilinear central
 polynomial for \( M_n(F) \). Then \( c_n(M_n(F), M_n(F), \dots, M_n(F)) = F \cdot I_{n \times n} \)
 and
 \begin{equation}\label{eq:central_poly_magic}
 \begin{aligned}
     c_n(M_A(\mathfrak{a} \ot A), M_A(\mathfrak{a} \ot A), \dots, M_A(\mathfrak{a} \ot A)) &= c_n(M_n(A), M_n(A), \dots, M_n(A)) \\
    &= A \cdot I_{n \times n} = \tilde{A}.
 \end{aligned}
 \end{equation}
 The left-hand side of \cref{eq:central_poly_magic}
 is stable under conjugation by 
 \( \varphi \in \textnormal{Aut}_{F\textnormal{-Alg}} (\mathfrak{a} \ot A) \)
 implying the identity \cref{eq:phi_A_phi}.

 Let us now fix some \( \varphi \in \textnormal{Aut}_{F\textnormal{-Alg}} (\mathfrak{a} \ot A) \)
 and take \( \psi \in \textnormal{Aut}_{F\textnormal{-Alg}}(A) \) such that
 \begin{equation}
 \varphi^{-1} \gamma_a \varphi = \gamma_{\psi(a)}
 \end{equation}
 for all \( a \in A \). Such an automorphism exists because of \cref{eq:phi_A_phi}.
 Define \( f \in \textnormal{Aut}_{F\textnormal{-Alg}} (\mathfrak{a} \ot A) \)
 by \( f(x \ot a) = x \ot \psi(a) \). Then \( \varphi = (\varphi \circ f) \circ f^{-1} \),
 where \( \varphi \circ f \in \textnormal{Aut}_{A\textnormal{-Alg}} (\mathfrak{a} \ot A) \),
 which completes the proof.
\end{proof}

\subsection{Trace extensions of \( F[\![x]\!] \) and topological Manin pairs} \label{sec:manin_pairs}
Henceforth, \(F\) is an algebraically closed field of characteristic 0 and \(\fg\) is a finite-dimensional simple Lie algebra over \( F \).
Consider the Lie algebra \( \fg[\![x]\!] \coloneqq \fg \ot F[\![x]\!]  \)
with the bracket \( [a \ot f, b \ot g] \coloneqq [a, b] \ot fg \).
Endowing \( \fg[\![x]\!] \) with the \( (x) \)-adic topology or, equivalently, with the
weak topology it becomes a topological Lie algebra. 

Let \( L \) be another Lie algebra over
\( F \).
We call \((L,\fg[\![x]\!])\) a \emph{topological Manin pair} if 
\begin{enumerate}
    \item \(L\) is a Lie algebra equipped with a non-degenerate symmetric invariant bilinear form \(B\);
    \item \(\fg[\![x]\!] \subseteq L\) is a Lagrangian subalgebra with respect to \(B\);
    \item for any continuous functional \(T \colon \fg[\![x]\!] \to F\) exists an element \(f \in L\) such that \(T = B(f,-)\).
\end{enumerate}
This notion is closely related to the
notion of trace extensions of
\(F[\![x]\!]\) introduced and
classified in \cite{Stolin_Zelmanov_Montaner}. 
Let us recall its definition.
Endow \( F[\![x]\!] \) with the
\( (x) \)-adic topology.
A unital commutative associative 
\( F \)-algebra extension \( A \supseteq F[\![x]\!] \), 
equipped with a continuous linear map \( \tOp \colon A \to F \),
is called
a \emph{trace extension} of \( F[\![x]\!] \) if
\begin{enumerate}
\item the bilinear form \( (f,g) \mapsto t(fg) \) is non-degenerate;
\item \( F[\![x]\!]^{\perp_{\tOp}} \coloneqq \{ f \in A \mid \tOp(F[\![x]\!] f) = 0 \} = F[\![x]\!] \);
\item for any continuous linear functional \( T \colon F[\![x]\!] \to F \) there exists
an \( f \in A \) such that \( T(p) = \tOp(pf) \)
for all \( p \in F[\![x]\!] \).
\end{enumerate}
The above-mentioned relation is then
presented in the following lemma.
\begin{lemma}
Let \((L,\fg[\![x]\!])\) be a topological Manin pair and \(B\) be the bilinear form of \(L\).
Then there exists a trace extension
\( (A, \tOp) \) of \( F[\![x]\!] \)
such that
\begin{equation}
L \cong \fg \otimes A \ \text{ and } \
B(a \ot f, b \ot g) = \kf(a,b) \tOp(fg),
\end{equation}
where \( \kf  \) is
the Killing form on \( \fg \).
\end{lemma}

The proof is straightforward and
obtained by combining 
\cref{thm:form_of_doubles} with \cite[Lemma 2.3]{Stolin_Zelmanov_Montaner}.


\begin{example}%
\label{ex:A(na)}
Let \( n \ge 1 \) and \( \alpha = (\alpha_i \in F \mid -\infty < i \le n-2) \) be an arbitrary sequence.
Consider the algebra
\begin{align*}
  A(n, \alpha) \coloneqq F(\!(x)\!) \oplus F[x]/(x^n),
\end{align*}
and the functional \( \tOp \colon A(n, \alpha) \to F\), given by
\begin{align*}
&\tOp(x^{n-1}) \coloneqq 1, \  \tOp(x^i) \coloneqq \alpha_i  \text{ for } i \le n-2, \\
&\tOp([x]^{n-1}) \coloneqq -1, \ \tOp([x]^{i}) \coloneqq -\alpha_i \text{ for } 0 \le i \le n-2 \text{ when } n \ge 2.
\end{align*}
Then \((A(n,\alpha),t)\) is a trace extension of \( F[\![x]\!] \), where the latter is identified with the image of the inclusion \( F[\![x]\!] \to A(n, \alpha) \), \( f \mapsto (f,[f])  \).
Here \( [f] \) is the equivalence class of \( f \) in \(F[\![x]\!]/(x^n) = F[x]/(x^n) \).
\end{example}
\begin{example}%
\label{ex:A(0)}
Let \( \alpha = (\alpha_i \mid -\infty < i \le -2) \) be an arbitrary
sequence in F. Then the algebra
\( F(\!(x)\!) \) with 
functional \( \tOp \), defined by
\( \tOp(x^{-1}) \coloneqq 1 \)
and \( \tOp(x^i) \coloneqq \alpha_i  \text{ for } i \le -2 \),
is a trace extension of \( F[\![x]\!] \)
denoted by \( A(0, \alpha) \).
Later we implicitly identify
\( F(\!(x)\!) \) with
\( F(\!(x)\!) \oplus \{0 \} \) in order to write \(A(n,\alpha) = F(\!(x)\!) \oplus F[x]/(x^n)\) for all \(n \ge 0\).
\end{example}

\begin{example}\label{ex:A(inf)}
The algebra \( A(\infty) \coloneqq \sum_{i \ge 0} Fa_i + F[\![x]\!] \)
with multiplication
\begin{align*}
  a_i a_j \coloneqq 0, \ a_ix^j \coloneqq a_{i-j} \text{ for } i \ge j \text{ and } a_i x^j \coloneqq 0 \text{ otherwise},
\end{align*}
and the functional \( \tOp \colon A \to F \),
defined by \( \tOp(a_0) \coloneqq 1 \), \( \tOp(a_i) \coloneqq 0 \), \( i \ge 1 \) and
\( \tOp(F[\![x]\!]) \coloneqq 0 \),
is a trace extension called 
\emph{the trivial extension of \( F[\![x]\!] \)}.
\end{example}

Two trace extensions \( (A, \tOp) \) and \( (A', \tOp') \)
are called \emph{equivalent} if there exists an algebra isomorphism \( T \colon A \to A' \)
identical on \( F[\![x]\!] \) and a non-zero scalar \( \xi \in F \)
such that \( \tOp'(T(a)) = \xi \tOp(a) \) for any \( a \in A \).
In this case we write \( (A, \tOp) \sim (A', \tOp') \).

\begin{proposition}[Proposition 2.9, \cite{Stolin_Zelmanov_Montaner}]%
\label{prop:possible_extensions}
Let \( (A, \tOp) \) be a trace extension of \( F[\![x]\!] \).
Then either \( (A, \tOp) \sim  A(\infty) \) or \( (A, \tOp) \sim A(n, \alpha) \)
for some \( n \ge 0 \) and some
sequence \( \alpha = (\alpha_i \in F \mid -\infty < i \le n-2) \).
\end{proposition}

\begin{remark}
The equivalence \( T \colon A \to A' \) of two
trace extensions \( (A, \tOp) \) and \( (A', \tOp') \)
does not extend in general to an isomorphism of Lie algebras
\( \fg \ot A \) and \( \fg \ot A' \) with forms
\( B(a \ot f, b \ot g) \coloneqq \kf(a,b) \tOp(fg) \) and
\( B'(a \ot f', b \ot g') \coloneqq \kf(a,b) \tOp'(f' g') \) respectively.
Indeed, by definition of equivalence for trace extensions we have
\begin{align*}
  B'(a \ot T(f), b \ot T(g)) = \xi B(a \ot f, v \ot g).
\end{align*}
In other words, \( T \) does not intertwine the corresponding
bilinear forms, but it gives an isomorphism between
\((\fg \ot A, \xi B)\) and \((\fg \ot A', B')\).
\end{remark}

\begin{corollary}\label{cor:possible_extensions}
Let \((L,\fg[\![x]\!])\) be a topological Manin pair with the
bilinear form \(B\).
Then there exists a non-zero \(\xi \in F\) such that either \((L,\xi B) \cong \fg \otimes A(\infty)\) or \((L,\xi B) \cong \fg \otimes A(n,\alpha)\) for some \(n \ge 0\) and some sequence \( \alpha = (\alpha_i \in F \mid -\infty < i \le n-2) \).
\end{corollary}

Let \( (A, \tOp) \) be a trace extension of \( F[\![x]\!] \).
One can produce new trace extensions from the given one in the following way: 
let \( \varphi \) be an algebra automorphism of \( A \) such that
\( \varphi(F[\![x]\!]) = F[\![x]\!] \), then
\( (A, \tOp)^{(\varphi)} \coloneqq (A, t \circ \varphi ) \)
is another trace extension of \( F[\![x]\!] \). 

We call automorphisms of \(F[\![x]\!]\) given by
\(x \mapsto a_1x + a_2 x^2 + \dots\), \(a_i \in F\), \(a_1 \neq 0\),
\emph{coordinate transformations}.  
Note that coordinate transformations
\(x \mapsto x + a_2x^2 + \dots \)
form a group under substitution which we denote by \(\Aut_0 F[\![x]\!] \).

\begin{lemma}
A coordinate transformation \( \varphi \in \Aut_0 F[\![x]\!] \) induces
an automorphism of \[ A(n, \alpha) = F(\!(x)\!) \oplus F[x]/(x^n) \] by
\( f/g \mapsto \varphi(f)/\varphi(g) \) and 
\( [x] \mapsto [\varphi(x)] \).

Moreover, there exists another sequence
\( \beta = (\beta_i \in F \mid -\infty < i \le n-2) \)
such that \( A(n, \alpha)^{(\varphi)} = A(n, \beta) \).
\end{lemma}
\begin{proof}
Let \( \varphi = x + a_2x^2 + a_3 x^3 + \dots \in \Aut_0 F[\![x]\!] \).
The first part of the statement is clear. Denote the induced
automorphism with the same letter \( \varphi \).
Consider the trace extension
\( \tOp \coloneqq \tOp_\alpha \circ \, \varphi \),
where \( \tOp_\alpha \) is the linear functional given
by the sequence \( \alpha \).
Since \( \tOp_\alpha(x^k) = 0 \) for \( k \ge n \)
we have \( \tOp(x^{n-1}) = 1 \) and a well defined sequence \( \beta \)
given by
\[ \beta_{i} \coloneqq \tOp (x^{i}), \ i \le n-2.  \]
Moreover, when \( n \ge 1 \)
we see that \( \tOp([x]^{n-1}) = \tOp_\alpha([\varphi(x)]^{n-1}) = \tOp_\alpha([x]^{n-1}) = -1 \)
and for \( 0 \le i \le n-2 \) we compute
\( \tOp([x]^{i}) = \tOp_\alpha([\varphi(x)]^{i}) = -\tOp_\alpha(\varphi(x)^{i}) = -\beta_i \).
Therefore \( \tOp \) is given by the sequence \( \beta \) and
\( A(n, \alpha)^{(\varphi)} = A(n, \beta) \) as we wanted.
\end{proof}

It was mentioned in
\cite[Section 2]{Stolin_Zelmanov_Montaner}
that by applying
\( \varphi \in \textnormal{Aut}_0 F[\![x]\!] \)
to a trace extension
\( A(n, \alpha), 0 \le n \le 2, \)
we get a full control over
sequence \( \alpha \). Namely, we can
make it into a zero sequence with 
\( \alpha_0 \neq 0 \).
The following proposition extends
these facts to \( n \ge 3 \).

\begin{proposition}
\label{prop:manin_pairs}
Let \(n \ge 0\) and \( \alpha = (\alpha_i \in F \mid -\infty < i \le n-2) \) be a sequence. There exists a  \(\varphi \in \textnormal{Aut}_0F[\![x]\!]\) such that \(A(n,\alpha) \cong A(n,\beta)^{(\varphi)}\), where \(\beta\) is the sequence satisfying \(\beta_i = 0\) for all \( i \neq 0\) and \(\beta_0 = \alpha_0\).
\end{proposition}

\begin{proof}
The cases \( n \in \{ 0, 1, 2\} \)
are considered in \cite{Stolin_Zelmanov_Montaner}.
Assume \( n>2 \).
Denote by \(\tOp\) the trace form corresponding
to \( A(n, \alpha) \).
We want to find a
\(u = x(1 + u_1 x + \dots ) \in F[\![x]\!]\) 
such that \(\tOp(u^k) = 0\) for all \(k \in \mathbb{Z}\setminus\{0,n-1\}\).
We do that by defining the coefficients \( u_1, u_2, \dots \) inductively.

We start by considering the
positive powers of \( u\).
Since \(\tOp(x^k) = 0\) for \(k\ge n\), the equality \(\tOp(u^k) = 0\) is automatic for \(k \ge n\). 
Next, we observe that
\begin{equation}
    \tOp(u^{n-2}) =
    \tOp(x^{n-2} + (n-2)u_1 x^{n-1}) = 
    \alpha_2 + (n-2)u_1
\end{equation}
and defining \( u_1 \coloneqq - \alpha_{n-2}/(n-2) \)
we get the desired \( \tOp(u^{n-2}) = 0 \). For \(n = 3\) this concludes the positive powers of \(u\).
Assume now that for \(n > 3\) we have fixed
\(u_1, \dots, u_{k-1} \),
\( 2 \le k \le n-2 \),
such that
\(\tOp(u^{n-2}) = \dots = \tOp(u^{n-k}) = 0\).
Using the convention \(u_0 \coloneqq 1\) we can write
\begin{equation}
    u^{n-k-1} = x^{n-k-1}\sum_{j = 0}^\infty \sum_{\substack{\ell_1+\dots+\ell_{n-k-1} = j\\ \ell_1,\dots,\ell_{n-k-1} \ge 0}}u_{\ell_1}\dots u_{\ell_{n-k-1}}x^j.
\end{equation}
Therefore, letting \(\alpha_{n-1} \coloneqq 1\) we get
\begin{equation}
    \tOp(u^{n-k-1}) = \sum_{j = 0}^{k} \sum_{\substack{\ell_1+\dots+\ell_{n-k-1} = j\\ \ell_1,\dots,\ell_{n-k-1} \ge 0}}u_{\ell_1}\dots u_{\ell_k}\alpha_{j + n-k-1}.
\end{equation}
Defining
\begin{equation}
    u_k \coloneqq - \frac{1}{n-k-1}\sum_{j = 0}^k \sum_{\substack{\ell_1+\dots+\ell_{n-k-1} = j\\ 0\le \ell_1,\dots,\ell_{n-k-1} < k}}u_{\ell_1}\dots u_{\ell_{n-k-1}}\alpha_{j + n-k-1}.
\end{equation}
we obtain the identity
\( \tOp(u^{n-k-1}) = 0 \)
completing the induction step
for positive powers of \( u \).

Now we proceed with negative powers
of \( u \).
For the base case, we notice that
\begin{equation}
\begin{aligned}
\tOp(u^{-1}) &= 
    \tOp(x^{-1} - u_1  - (u_2 - u_1^2)x - \dots - (u_n - p(u_1, \dots, u_{n-1}))x^{n-1}) \\
    &= \alpha_1 - u_1 \alpha_0 - 
    (u_2 - u_1^2) \alpha_1 - \dots
    - (u_n - p(u_1, \dots, u_{n-1})),
\end{aligned}
\end{equation}
where \( p \) is a polynomial
depending only on
\( u_1, \dots, u_{n-1} \).
In other words, 
if we fix \( u_1, \dots, u_{n-2} \) as above and choose an arbitrary
\( u_{n-1} \in F \), then
we can put
\[ u_{n} \coloneqq \alpha_1 - u_1 \alpha_0 - 
    (u_2 - u_1^2) \alpha_1 - \dots +  p(u_1, \dots, u_{n-1}) \]
and obtain \( \tOp(u^{k}) = 0\)
for \( -1 \le k \le n-2 \) and \( k \neq 0 \).
For the inductive step, write \(v \coloneqq u^{-1} \coloneqq x^{-1}(1 + v_1x + \dots)  \in x^{-1}F[\![x]\!]\). 
The identity \(1 = uv\) implies
\(\sum_{j = 0}^k u_jv_{j-k} = 0\) for all \(k \in \mathbb{Z}_{> 0}\).
Therefore, \(u_1,\dots,u_k\) uniquely determine \(v_1,\dots,v_k\) and vice versa.
We have 
\begin{equation}
    v^k = x^{-k}\sum_{j = 0}^\infty \sum_{\substack{\ell_1+\dots+\ell_k = j\\ \ell_1,\dots,\ell_k \ge 0}}v_{\ell_1}\dots v_{\ell_k}x^j
\end{equation}
and, as a consequence, 
\begin{equation}
    \tOp(u^{-k}) = \sum_{j = 0}^{n-1+k} \sum_{\substack{\ell_1+\dots+\ell_k = j\\ \ell_1,\dots,\ell_k \ge 0}}v_{\ell_1}\dots v_{\ell_k}\alpha_{j - k}.
\end{equation}
In particular,
if parameters
\( u_1, \dots, u_{n-2}, \dots, u_{n+k}, \)
\( k \ge 0 \) are fixed in
such a way, that
\( \tOp(u^{m}) = 0 \)
for \( -k-1 \le m \le n-2 \),
\( m \neq 0 \)
then by defining
\begin{equation}
v_{n+k+1} \coloneqq -\frac{1}{k+2}\sum_{j = 0}^{n+k+1} \sum_{\substack{\ell_1+\dots+\ell_{k+2} = j\\ 0\le \ell_1,\dots,\ell_{k+2} < n+k+1}}v_{\ell_1}\dots v_{\ell_{k+2}}\alpha_{j -k-2},
\end{equation}
we get the identities
\( \tOp(u^{m}) = 0 \)
for all \( -k-2 \le m \le n-2 \),
\( m \neq 0 \).
This concludes the proof.
\end{proof}

Combining \cref{cor:possible_extensions} and \cref{prop:possible_extensions} results in the following statement.

\begin{corollary}\label{rem:Manin_pairs}
Let \((L,\fg[\![x]\!])\) be a topological Manin pair and \(B\) be the bilinear form on \(L\).
Then there exists a \(\varphi \in \textnormal{Aut}_0 F[\![x]\!]\) such that either \((L,B) \cong \fg \otimes A(\infty)^{(\varphi)}\) or \((L,B) \cong \fg \otimes A(n,\alpha)^{(\varphi)}\) for some \(n \ge 0\) and some sequence \( \alpha = (\alpha_i \in F \mid -\infty < i \le n-2) \) satisfying \(\alpha_i = 0\) for all \(i \neq 0\).
\end{corollary}

\begin{remark}%
\label{rem:B_i}
Let us explicitly describe the bilinear
forms on 
\begin{equation}
    \fg \otimes A(i-1,0) = \fg(\!(x)\!) \times \fg[x]/x^{i-1}\fg[x],   
\end{equation}
where \( i \in \mathbb{Z}_{>0} \).
We start with the bilinear form \(\mathcal{K}_i\colon \fg(\!(x)\!) \times \fg(\!(x)\!) \to F\) 
defined by
\begin{equation}
    \mathcal{K}_i(f,g) \coloneqq \textnormal{res}_{x = 0}\{x^{1-i}\kappa(f(x),g(x))\} = \textnormal{coeff}_{i-2}\{\kappa(f(x),g(x))\}, 
\end{equation}
where \( \kf \) stands for the Killing form of \(\fg(\!(x)\!)\) over \(F(\!(x)\!)\).
Observe that \(\mathcal{K}_i(x^{i-1}\fg[x],x^{i-1}\fg[x]) = 0\). Therefore, 
\begin{equation}
    \mathcal{K}_i([f],[g]) \coloneqq \mathcal{K}_i(f,g)
\end{equation}
gives a well-defined bilinear form
on \(\fg[x]/x^{i-1}\fg[x] = \fg[\![x]\!]/x^{i-1}\fg[\![x]\!]\).
The bilinear form \(B_i\) on \(\fg \otimes A(i-1,0)\) is now defined by
\begin{equation}
    B_i((f_1,[f_2]),(g_1,[g_2])) \coloneqq \mathcal{K}_i(f_1,g_1) - \mathcal{K}_i(f_2,g_2)
\end{equation}
for all \(f_1,g_1 \in \fg(\!(x)\!)\) and \(f_2,g_2 \in \fg[\![x]\!]\).
\end{remark}

\begin{remark}
Manin pairs are related to so-called quasi-Poisson structures in a similar fashion as Manin triples are related to Poisson-Lie groups via their connection to Lie bialgebra structures. These quasi-Poisson structures are of independent interest, as they capture interesting phenomena which go beyond the realm of usual Poisson structures; see e.g. \cite{alekseev_kosmann}. Interestingly, contrary to Poisson case, it is open if the quasi-Poisson structures associated to Manin pairs can be appropriately quantized; see e.g. \cite[Remark in Section 16.2.2]{etingof_schiffmann}. 
\end{remark}

\subsection{Topological Lie bialgebra structures on \( \fg[\![x]\!] \)}
As before, we equip the Lie algebra
\( \fg[\![x]\!] \)
with the \( (x) \)-adic
topology.
The continuous dual
\[ \fg[\![x]\!]' = \{ \text{linear maps } f \colon \fg[\![x]\!] \to F \mid f(x^n \fg[\![x]\!]) = 0 \text{ for some } n \in \bZ_{\ge 0}   \} \]
of \(\fg[\![x]\!]\) is isomorphic as a vector space to the space of polynomials with coefficients in \( \fg \).
We endow it with the discrete topology and, to avoid any confusion, denote
it by \( \fg\{ x \} \).
According to our definition in \cref{sec:Top_Lie_bialg}
a topological Lie bialgebra structure on \( \fg[\![x]\!] \) is a continuous
linear map
\[
  \delta \colon \fg[\![x]\!] \to \fg[\![x]\!] \fot \fg[\![y]\!] \cong (\fg \ot \fg)[\![x,y]\!]
\]
such that
\begin{enumerate}
\item the dual map \( \delta' \colon \fg\{ x \} \ot \fg\{ y \} \cong ((\fg \ot \fg)[\![x,y]\!])' \to \fg\{ x \} \) is
a Lie bracket on \( \fg\{ x \} \) and
\item\label{it:comp_condition} the compatibility condition \( \delta([f,g]) = [f \fot 1 + 1 \fot f, \delta(g)] -
                  [g \fot 1 + 1 \fot g, \delta(f)] \) holds for all
                  \( f,g \in \fg[\![x]\!] \).
\end{enumerate}

\begin{remark}
We omitted the continuity condition on \( \delta' \)
because \( \fg\{ x \} \) and hence \( \fg\{ x \} \ot \fg\{ y \} \)
are discrete.
Moreover, repeating the proof of
\cite[Lemma 4.1]{Stolin_Zelmanov_Montaner}
we can easily show that the compatibility condition
guarantees the inclusion
\[ \delta(x^n \fg[\![x]\!]) \subseteq (x,y)^{n-1} (\fg \ot \fg)[\![x,y]\!] \]
for any positive integer \( n \) and thus automatically implies
the continuity of \( \delta \).
Therefore, our definition of a topological Lie bialgebra
structure on \( \fg[\![x]\!] \)
coincides with the
one given in \cite{Stolin_Zelmanov_Montaner}.
\end{remark}

The following result shows that, as in the classical case,
topological Lie bialgebra structures on \( \fg[\![x]\!] \)
produce topological Manin triples.

\begin{lemma}\label{lem:top_bialg_gives_manin_triple_taylor}
Let \( (\fg[\![x]\!], \delta) \) be a topological Lie bialgebra.
The corresponding triple
\[
\left(\fg[\![x]\!] \add \fg\{ x \}, \fg[\![x]\!], \fg\{ x \} \right)
\]
described in \cref{sec:top_Man_triples} is a topological Manin triple,
i.e.\ the Lie bracket on \( \fg[\![x]\!] \add \fg\{ x \} \) and the form \( B \),
defined in \cref{eq:form_topological_double,eq:Lie_bracket_topological_double},
are continuous and separately continuous respectively.
\end{lemma}
\begin{proof}
Fix some arbitrary \( X,Y \in \fg[\![x]\!] \) and
\( f,g \in \fg\{ x \} \).
The separate continuity follows from the identity
\[
B\left(X + f, (Y + x^{\deg(f)+1} \fg[\![x]\!]) + g \right) = B(X + f, Y + g).
\]
To prove the continuity of the Lie bracket it is enough
to show that for \( 0 \)-neighbourhood \( x^n \fg[\![x]\!] \subseteq  \fg[\![x]\!] \add \fg\{ x \}\)
there is a \( 0 \)-neighbourhood \( x^m \fg[\![x]\!] \subseteq \fg[\![x]\!] \) such that
\[
\left[ X + x^m \fg[\![x]\!], g \right] \subseteq [X,g] + x^n \fg[\![x]\!].
\]
Let \( m \coloneqq \max \{2\deg(g) + 2, 2n \} \), then 
\[
\underbrace{g\left([x^m \fg[\![x]\!], \cdot]\right)}_{= 0} +
\underbrace{(g \fot 1)\delta(x^m \fg[\![x]\!])}_{\subseteq x^n \fg[\![x]\!]} \subseteq x^n \fg[\![x]\!],
\]
providing the continuity of the bracket.
\end{proof}

\begin{remark}\label{rem:doubles_in_series_case}
Using the homeomorphisms \cref{eq:two_topologies_coincide} 
it is not hard to show that we have the following converse
to \cref{lem:top_bialg_gives_manin_triple_taylor}.
Let \( M = (L, L_+, L_-) \) be a topological Manin triple 
with form \( B \) such that
\begin{enumerate}
\item \( L_- \) is equipped with the discrete topology,
\item As a topological Lie algebra \( L_+ \cong \fg[\![x]\!] \) and finally
\item \( B \) identifies \( L'_+ \) with \( L_- \).
\end{enumerate}
Then \( M \) defines a unique topological Lie bialgebra structure \(\delta\) 
on \( \fg[\![x]\!] \) and \(M\) is isomorphic to \((\fD(\fg[\![x]\!],\delta),\fg[\![x]\!],\fg\{x\})\) as a topological Manin triple.
Consequently, the classification of topological
Lie bialgebras on \( \fg[\![x]\!] \) coincides with the classification
of topological Manin triples with the above-mentioned properties.
Note that each such Manin
triple also gives rise to a
topological Manin pair \((L,L_+)\) in the sense of \cref{sec:manin_pairs}.
\end{remark}

\subsection{Classification of topological doubles}\label{sec:classification_of_doubles_MSZ}
In this section we recall the classification of topological doubles of \(\fg[\![x]\!]\) from \cite{Stolin_Zelmanov_Montaner}, which can be understood as a refinement of \cref{rem:Manin_pairs}.
Indeed, a topological Manin pair \((L,\fg[\![x]\!])\) is a topological double of \(\fg[\![x]\!]\) if and only if there exists a Lagrangian subalgebra \(W \subseteq L\) complementary to \(\fg[\![x]\!]\). If we choose a representation \(L \cong \fg \otimes A(n,\alpha)^{(\varphi)}\) according to \cref{rem:Manin_pairs}, the existence of such a \(W\) implies that \(0 \le n \le 2\) and \(\alpha_0 = 0\). Thus, \cref{rem:Manin_pairs} takes the following form in this setting.

\begin{theorem}[Proposition 2.10 \cite{Stolin_Zelmanov_Montaner}]%
\label{thm:limitations_on_n}
Let \( D = \fg \ot A(n, \alpha) \) for some \( n \ge 0 \) and a
sequence \( \alpha = (\alpha_i \in F \mid -\infty < i \le n-2) \).
Then
\begin{enumerate}
\item \( D \) contains a Lagrangian Lie subalgebra \( W \) such that
\( D = \fg \ot F[\![x]\!] \add W \) if and only if \( 0 \le n \le 2 \) and
\item in this case there is a unique \( \varphi \in \Aut_0 F[\![x]\!] \)
such that \( A(n, \alpha) \sim A(n, 0)^{(\varphi)} \).
\end{enumerate}
\end{theorem}

Let us note that there is no ambiguity in the topologies of the respective objects according to the following result.
\begin{lemma}%
\label{lem:limitations_topoology}
Let \(\delta_1\) and \(\delta_2\) be two topological Lie bialgebra structures
on \(\fg[\![x]\!]\). If there is a Lie algebra isomorphism \(\varphi\) between
the corresponding topological doubles
\(\fg[\![x]\!] \add \fg\{x\}_1\) and \(\fg[\![x]\!] \add \fg\{x\}_2\)
that is identical on \(\fg[\![x]\!]\) and intertwines the forms, 
then \(\varphi\) is a homeomorphism.
\end{lemma}
\begin{proof}
Let \(\fD\) be the topological double of \(\delta_1\).
As a Lie algebra it has two decompositions into Lagrangian Lie subalgebras
coming from \(\delta_1\) and \(\delta_2\)
\[
\fD = \fg[\![x]\!] \add \fg\{x\}_1 = \fg[\![x]\!] \add \varphi\left(\fg\{x\}_2\right).
\]
These decompositions give rise to two different product topologies generated by
\begin{align*}
\fT_1 &\coloneqq \left\{ U \times \{f\} \mid U \text{ is open in } \fg[\![x]\!] \text{ and } f \in \fg\{x\}_1 \right\}, \\
\fT_2 &\coloneqq \left\{ U \times \{g\} \mid U \text{ is open in } \fg[\![x]\!] \text{ and } g \in \varphi\left(\fg\{x\}_2\right) \right\}.
\end{align*}
For any \(U \times \{f\} \in \fT_1 \) we can find \(n \in \mathbb{Z}_+\),
\(u, u' \in \fg[\![x]\!]\) and \(g' \in \varphi\left(\fg\{x\}_2\right)\)
such that \(u + x^n \fg[\![x]\!] \subseteq U\) and \(f = u' + g'\).
Then \((u + u' + x^n \fg[\![x]\!]) \times \{ g' \} \) is an element in \(\fT_2\)
that is contained in \(U \times \{f\} \).
Similarly, for any element \(V\) in \(\fT_2\) there is an element in \(\fT_1\) contained in \(V\).
Therefore, the topologies generated by \(\fT_1 \) and \(\fT_2\)
are equal.
\end{proof}

\begin{remark}\label{rem:four_doubles}
Combining \cref{lem:limitations_topoology,thm:limitations_on_n} we see that
for any topological Lie bialgebra \( (\fg[\![x]\!], \delta) \) with the
topological double \( \fD(\fg[\![x]\!], \delta) \)
there exist a scalar \( \xi \in F^{\times} \) and either
\begin{itemize}
\item An \( F[\![x]\!] \)-linear isomorphism of Lie algebras
\( \fD(\fg[\![x]\!], \xi \delta) \cong \fg \ot A(\infty) \) identical on \(\fg[\![x]\!]\)
and intertwining the corresponding forms
or
\item An integer \( 0 \le n \le 2 \), a series (change of variable)
\( \varphi \in \Aut_0 F[\![x]\!] \)
and an \( F[\![x]\!] \)-linear isomorphism of Lie algebras 
\( \fD(\fg[\![x]\!], \xi \delta) \cong \fg \ot A(n, 0)^{(\varphi)} \)
identical on \(\fg[\![x]\!]\) and
intertwining the corresponding forms.
\end{itemize}
Moreover, there are unique (independent of \(\delta\) and \(\xi\))
topologies on \(\fg \ot A(\infty)\)
and \( \fg \ot A(n, 0)^{(\varphi)} \) making the above-mentioned
isomorphisms into homeomorphisms. 
We will see in
\cref{subsec:realization_of_doubles_polynomials} 
that all these Lie algebras with forms
can indeed be realized as topological doubles
of some topological Lie bialgebra structures on 
\( \fg[\![x]\!] \).
\end{remark}

\section{Topological twists of \( (\fg[\![x]\!], \delta) \)}%
\label{sec:topological_twists}
Given a classical Lie bialgebra structure \( \delta \colon L \to L \ot L \)
on a Lie algebra \( L \) we can obtain new Lie bialgebra structures
by means of so called twisting. More precisely, any skew-symmetric
tensor \( s \in L \ot L \), satisfying 
\begin{equation}%
\label{eq:extra_twist_condition}
    \textnormal{CYB}(s) = \textnormal{Alt}((\delta \ot 1)s),
\end{equation}
gives rise to a twisted Lie bialgebra structure \( \delta_s \coloneqq \delta + ds \) on \( L \).
Here \( ds(a) \coloneqq [a \ot 1 + 1 \ot a, s] \), \(\textnormal{Alt}\) was defined in \cref{eq:key_maps}
and
\begin{equation}
    \textnormal{CYB}(s) \coloneqq [s^{12},s^{13}] + [s^{12},s^{23}] + [s^{13},s^{23}],
\end{equation}
is the classical Yang-Baxter equation,
where e.g.\ \([(a \otimes b)^{13},(c\otimes d)^{23}] = a \otimes c \otimes [b,d]\).
Since the classical doubles for \( \delta_s \) and \( \delta \) are equal
this procedure can be used for classification of Lie bialgebra structures
within a fixed double \( \fD \).

We now adopt this procedure to our topological
setting by introducing topological twists
for a topological Lie bialgebra
structure on \( \fg[\![x]\!] \).
These are just certain elements in
\( (\fg \ot \fg)[\![x,y]\!] \)
satisfying a condition similar to
\cref{eq:extra_twist_condition}. 
In subsequent sections, 
by classifying topological twists,
we obtain
a full description of all topological
Lie bialgebra structures \(\delta\) on \( \fg[\![x]\!] \) with
doubles \( \fD(\fg[\![x]\!], \delta) \cong \fg \otimes A(i-1,0)\), where \(i \in \{ 1, 2, 3 \} \).
\subsection{Topological twists}\label{subsec:topological_twists}

We say that an element \( s \in \fg[\![x]\!] \fot \fg[\![y]\!] = \fg[\![x]\!] \uot \fg[\![y]\!] = (\fg \ot \fg)[\![x,y]\!] \)
is a \emph{topological twist} of a topological Lie bialgebra \( (\fg[\![x]\!], \delta) \)
if
\begin{enumerate}
\item \( s + \overline{\tau}(s) = 0 \);
\item \( \overline{\textnormal{CYB}}(s) = \overline{\textnormal{Alt}}((\delta \fot 1)s) \).
\end{enumerate}

\begin{remark}
For a topological Lie algebra \(L\) the function \( \overline{\textnormal{CYB}} \) is the unique continuous extension
of the continuous composition \[ L \ot L \xrightarrow{\textnormal{CYB}} L \ot L \ot L \xhookrightarrow{\phantom{\textnormal{CYB}}} L \fot L \fot L.  \] 
To see that \( \textnormal{CYB} \) is continuous, 
note that e.g.\ \(s \mapsto [s^{12},s^{13}]\) is the composition 
of the diagonal map 
\(L \otimes L \to (L \otimes L) \times (L \otimes L)\), 
the canonical map \((L \otimes L) \times (L \otimes L) \to L^{\otimes 4}\), 
the endomorphism of \(L^{\otimes 4}\),
which switches the second and third factors,
and \([\cdot \, ,\cdot] \otimes 1 \otimes 1 \colon L^{\otimes 4} \to L^{\otimes 3}\). All of these maps are continuous, so \(s \mapsto [s^{12},s^{13}]\) is too. Similar arguments apply to \(s \mapsto [s^{12},s^{23}]\) and \(s \mapsto [s^{13},s^{23}]\).

\end{remark}

Let \( s \in (\fg \ot \fg)[\![x,y]\!] \) be a topological twist of \( \delta \). Define the linear map
\( ds \colon \fg[\![x]\!] \to (\fg \ot \fg)[\![x,y]\!] \) by 
\( f \mapsto [f \fot 1 + 1 \fot f,s]. \)
As in the classical case the linear functional 
\( \delta_s \coloneqq \delta + ds \) is again a topological Lie bialgebra 
structure on \( \fg[\![x]\!] \).

We say that two topological twists \( s_1 \) and \( s_2 \)
of \( \delta \) are \emph{formally isomorphic}, if there is an
\( F[\![x]\!] \)-linear isomorphism \( \phi \)
of topological Lie bialgebras \( (\fg[\![x]\!], \delta + ds_1) \)
and \( (\fg[\![x]\!], \delta + ds_2) \).

\begin{remark}
Observe, that by 
\cref{thm:splitting_aut}
any \(F\)-linear automorphism
of \( \fg[\![x]\!] \) is a composition
of an \( F[\![x]\!] \)-linear
automorphism of \( \fg \) and 
an \( F \)-linear automorphism of
\( F[\![x]\!] \). 
It is not hard to prove that any
\( F \)-linear automorphism of
\( F[\![x]\!] \) is given by
\begin{equation}%
\label{eq:coordinate_transformation_mapping}
x \longmapsto a_1 x + a_2 x^2 + \dots
\end{equation}
for some non-zero \(a \in F \).
These observations immediately imply
that
\begin{enumerate}
    \item Any automorphism of \( \fg[\![x]\!] \) is automatically
    continuous with respect to the \((x)\)-adic topology as a composition of continuous automorphisms and
    
    \item Applying coordinate
    transformations \cref{eq:coordinate_transformation_mapping}
    and scaling to the
    isomorphism classes of topological
    twists, we obtain all possible
    Lie bialgebra structures on
    \( \fg[\![x]\!] \).
\end{enumerate}
Therefore,
we have reduced the classification of
topological Lie bialgebras
to the classification of
topological twists.
\end{remark}


\subsection{Lagrangian Lie subalgebras}
Topological twists of \( \delta \)
are closely related to
Lagrangian Lie subalgebras of \( \fD(\fg[\![x]\!], \delta) \). 
The following
theorem describes this relation.

\begin{theorem}%
\label{thm: W <-> t <-> T correspondence}
  Let \( (\fg[\![x]\!], \delta) \) be a topological Lie bialgebra with the
  classical double \( \fD \).
  Then there are the following one-to-one correspondences:
  
 \begin{tikzpicture}[every node/.style={inner sep=.15cm,outer sep=0cm}]
  \node[] (A) {\mlnode{
    Topological twists of \( \delta \), i.e. \\
    skew-symmetric tensors \( s \in \fg[\![x]\!] \fot \fg[\![y]\!] \) \\
    satisfying \( 
    \overline{\textnormal{CYB}}(s)
    =
    \overline{\textnormal{Alt}}\left((\delta \fot 1)s\right) \)
  }};
  \node[below left=1cm and -2cm of A] (B) {\mlnode{
    Lagrangian \\
    Lie subalgebras \( W \subseteq \fD \) \\
    complementary to \( \fg[\![x]\!] \)
  }};
    \node[below right=0.8cm and -2cm of A] (C) {\mlnode{
    Linear maps 
    \( T \colon \fg\{ x \} \longrightarrow \fg[\![x]\!] \) such that \\
    for all \( p_1, p_2, p_3 \in \fg\{ x \} \) \\
    \( B(Tp_1, p_2) + B(p_1, Tp_2) = 0 \) and \\
    \( B([Tp_1 - p_1, Tp_2 - p_2], Tp_3 - p_3) = 0 \) 
  }};
  
\draw[<->] (A) edge (B) (B) edge (C) (C) edge (A);
\end{tikzpicture}  
\end{theorem}

\begin{proof}
The proof is done by repeating the arguments in
\cite[Theorem 2.4]{Raschid_Stepan_classical_twists}
and \cite[Theorem 7]{Karolinsky_Stolin} within our topological setting.
The only step that deserves a clarification is the construction of a
topological twist \( s \)
from a linear map \( T \colon \fg\{ x \} \to \fg[\![x]\!] \).

Let \( \{ Tp_i \}_{i \ge 1} \) be a basis for \( \im(T) \).
We write \( \textnormal{mindeg}(Tp_i) \) for the minimal exponent of \( x \)
contained in \( Tp_i \) and \( (Tp_i)_m \) for the \( m \)-th degree part of \( Tp_i \),
i.e.\ \( (\sum_{i \ge 0} a_i \ot x^i)_m = a_m \ot x^m \).
Since a basis of \( \im(T) \) is at most countable, we can without loss of generality assume
that for any non-negative integer \( m \) there is a non-negative integer \( N_m \)
such that
\[ \begin{cases}
\textnormal{mindeg}(Tp_i) \le m & i \le N_m; \\
\textnormal{mindeg}(Tp_i) > m & i > N_m; \\
\{ (Tp_i)_m \mid N_{m-1} < i \le N_m \} & \text{is a linearly independent set.}
\end{cases} \]

Now let us construct a dual set \( \{ v_i \} \subseteq \fg\{ x \} \).
For each \( m \ge 0 \) and \( N_{m-1} < i \le N_{m} \)
define \( v_i' \coloneqq a_i x^m \),
where \( a_i \in \fg \) are chosen in such a way that 
\[ B(Tp_j, v_i') = \begin{cases}
1 & N_{m-1} < i = j \le N_{m}; \\
0 & N_{m-1} < i \neq j \le N_{m}.
\end{cases} \]
For each \( m \ge 0 \) and \( N_{m-1} < i \le N_{m} \)
we recursively define \[ v_i \coloneqq v'_i - \sum_{k = 1}^{N_{m-1}} B(Tp_k, v_i')v_k. \]
The corresponding twist is then defined as
\( s \coloneqq - \sum_{i} Tp_i \ot Tv_i \). Since for any 
\( m \in \mathbb{Z}_{\ge 0} \) there is at most \( N_m \) elements
\( Tp_i \) containing \( x^m \) this is a well-defined
element of \( (\fg \ot \fg)[\![x,y]\!] \).
Moreover, we have
\[ Tp_k = \sum_{i} B(v_i, Tp_k)Tp_i = -\sum_{i} B(Tv_i, p_k) Tp_i \]
for all \( k \). Because \( T \) is completely defined by its action
on \( \{ p_i \} \) we have the equality
\begin{equation*}
    T = -\sum_{i} B(Tv_i, \cdot) Tp_i.
\end{equation*} 
\end{proof}

\begin{remark}\label{rem:construction_of_Ws}
Let \(s = \sum_{k = 0}^\infty \sum_{\alpha = 1}^n s_{k,\alpha} \otimes I_\alpha y^k \in (\fg \otimes \fg)[\![x,y]\!]\) be a classical twist of \(\delta\), for some basis \(\{I_\alpha\}_1^n\) of \(\fg\). The corresponding 
\(T \colon \fg\{x\} \to \fg[\![x]\!]\) is defined by \begin{equation}
Tw \coloneqq \sum_{k = 0}^\infty \sum_{\alpha = 1}^n B(w,I_\alpha x^k)s_{k,\alpha},    
\end{equation}
where \(B\) is the bilinear form of \(\fD(\fg[\![x]\!],\delta)\). The Lagrangian Lie subalgebra associated to \(s\) is given by 
\begin{equation}
    W \coloneqq \{Tw-w\mid w \in \fg\{x\}\}.
\end{equation}
\end{remark}

We say that two topological Lie bialgebras
\( (\fg[\![x]\!], \delta) \) and \( (\fg[\![x]\!], \tilde{\delta}) \)
are \emph{in the same twisting class} if there is an \( F[\![x]\!] \)-linear isomorphism of
topological Lie algebras \( \varphi \colon \fD(\fg[\![x]\!], \delta) \to \fD(\fg[\![x]\!], \tilde{\delta}) \)
intertwining the corresponding forms and such that the diagram
\[
\begin{tikzcd}
 & \fg[\![x]\!] \arrow[hookrightarrow]{dr} \arrow[hookrightarrow]{dl} \\
\fD(\fg[\![x]\!], \delta) \arrow{rr}[swap]{\varphi} && \fD(\fg[\![x]\!], \tilde{\delta})
\end{tikzcd}
\]
commutes.
It is clear that if \( (\fg[\![x]\!], \delta) \) and \( (\fg[\![x]\!], \tilde{\delta}) \)
are in the same twisting class, then \( \delta \) is completely 
determined by a Lagrangian topological Lie subalgebra
\( W \coloneqq \varphi(\fg\{ x \}) \subseteq \fD(\fg[\![x]\!], \tilde{\delta}) \).
Combining
\cref{thm: W <-> t <-> T correspondence,rem:delta_as_dual_of_cobracket}, 
we get the following statement.
\begin{lemma}
  If \( (\fg[\![x]\!], \delta) \) and \( (\fg[\![x]\!], \tilde{\delta}) \) are in the same
  twisting class, then there exists a topological
  twist \( s \in (\fg \ot \fg)[\![x,y]\!] \)
  of \( \delta \) such
  that \( \tilde{\delta} = \delta + ds \).
  Conversely, if \( \tilde{\delta} = \delta + ds \),
  for a topolgoical twist \( s \),
  then \( \fD(\fg[\![x]\!], \delta) = \fD(\fg[\![x]\!], \tilde{\delta}) \).
\end{lemma}

Let \( \phi \) be an \( F[\![x]\!] \)-linear
automorphism of \( \fg[\![x]\!] \),
then \( (\phi \times [\phi]) \)
is a Lie algebra automorphism of
\( \fg \ot A(n, 0) = \fg(\!(x)\!) \times \fg[x]/x^n \fg[x] , n \ge 0\).
Here we use the same
notation \( \phi \) for the unique extension
of \( \phi \) to an automorphism of \( \fg(\!(x)\!) \)
and
\( [\phi] \) for the induced
map on \( \fg[x]/x^n \fg[x] \)
given by
\( [a \ot x^k] \mapsto [\phi(a) \ot x^k] \).
We call two Lagrangian Lie subalgebras
\( W \text{ and } \widetilde{W} \subset \fD(\fg[\![x]\!], \delta_i) \), \( i \in \{ 1,2,3 \} \),
\emph{formally isomorphic} if there
is an \( F[\![x]\!] \)-linear automorphism 
\( \phi \) of \( \fg[\![x]\!] \) such that
\begin{equation}
    \widetilde{W} = (\phi \times [\phi])(W).
\end{equation}

\section{Formal \(r\)-matrices}%
\label{sec:formal_r_matrices}

In this section we show that there is a deep connection between topological Lie bialgebra
structures on \(\fg[\![x]\!]\) and formal \(r\)-matrices leading to important observations
for both structures.
For instance, we use formal \(r\)-matrices to
show that all topological Lie algebras with forms
described in \cref{rem:four_doubles} 
are obtainable as topological doubles of certain 
topological Lie bialgebra structures on \(\fg[\![x]\!]\). 
Furthermore, Lagrangian Lie subalgebras determined
by topological twists have a very useful description
in terms of formal \( r \)-matrices, 
which is used in the study of equivalences. 
On the other hand, the classification of topological
doubles, results in unexpected restrictions on the form of formal \(r\)-matrices. 

\subsection{Formal \(r\)-matrices and topological Lie bialgebra structures on \(\fg[\![x]\!]\).}
Let \( \{ I_\alpha \}_{1}^{n} \) be an orthonormal basis for
\( \fg \) with respect to the Killing form \( \kf \) on it. Recall that the quadratic Casimir element 
is defined as \(\Omega \coloneqq \sum_{\alpha = 1}^n I_\alpha \ot I_\alpha \in \fg \ot \fg\).
Consider the formal Yang's \( r \)-matrix
\begin{equation}
  r_{\textnormal{Yang}}(x,y) \coloneqq \frac{\Omega}{x-y} = \sum_{k = 0}^\infty \sum_{\alpha = 1}^n x^{-k-1} I_\alpha \ot y^k I_\alpha \in (\fg \ot \fg)(\!(x)\!)[\![y]\!].
\end{equation}
A series \( r \in (\fg \ot \fg)(\!(x)\!)[\![y]\!] \) of the form
\begin{equation}\label{eq:standard_form_formal_r_matrix}
  r(x,y) = s(y)r_{\textnormal{Yang}}(x,y) + g(x,y),
\end{equation}
where \( s \in F[\![y]\!] \) and \( g \in (\fg \ot \fg)[\![x,y]\!] \) is called a \emph{formal \( r \)-matrix}
if it solves the formal classical Yang-Baxter equation (CYBE)
\begin{equation*}
    \overline{\textnormal{CYB}}(r) \coloneqq [r^{12}(x_1, x_2), r^{13}(x_1, x_3)] +
    [r^{12}(x_1, x_2), r^{23}(x_2, x_3)] +
    [r^{13}(x_1, x_3), r^{23}(x_2, x_3)] = 0.
\end{equation*}
Here for instance
\begin{equation}
    [a^{12}(x_1,x_3),b^{23}(x_2,x_3)] \coloneqq \sum_{j,k = 0}^\infty \sum_{\alpha,\beta = 1}^n a_{j,\alpha}(x_1) \otimes [x_2^jI_\alpha, b_{k,\beta}(x_2)] \otimes I_\beta x_3^k
\end{equation}
for all \(a = \sum_{j = 0}^\infty \sum_{\alpha= 1}^n a_{j,\alpha}(x) \otimes y^jI_\alpha\) and \( b = \sum_{j = 0}^\infty \sum_{\alpha= 1}^n b_{j,\alpha}(x) \otimes y^jI_\alpha \) in \( (\fg \ot \fg)(\!(x)\!)[\![y]\!]\).
The other two commutators are defined in a similar way. 
As the notation suggests, \(\overline{\textnormal{CYB}}\) coincides with the map defined in \cref{sec:topological_twists} for \(r \in (\fg \otimes \fg)[\![x,y]\!]\).

\begin{remark}
We can equivalently define a 
formal \( r \)-matrix as a
series
\[
  r(x,y) = s(x,y)r_{\textnormal{Yang}}(x,y) + g(x,y) \in (\fg \ot \fg)(\!(x)\!)[\![y]\!],
\]
where \( s \in F[\![x,y]\!] \) and \( g \in (\fg \ot \fg)[\![x,y]\!] \). The equivalence of two
definitions follows from
the equality 
\begin{equation}
    r(x,y) = s(y,y) r_{\textnormal{Yang}}(x,y)  +\underbrace{\frac{s(x,y)-s(y,y)}{x-y}}_{\in F[\![x,y]\!]} \Omega + g(x,y).
\end{equation}
\end{remark}

The connection between formal \(r\)-matrices and topological Lie bialgebra structures is based on the following observation.

\begin{lemma}\label{lem:rmat_defines_delta}
For any formal \(r\)-matrix \( r \in (\fg \ot \fg)(\!(x)\!)[\![y]\!] \) the formula
\begin{equation}\label{eq:standard_bialgebra_def}
  dr(f) \coloneqq [f \fot 1 + 1 \fot f, r]
\end{equation}
defines a topological Lie bialgebra structure 
\( dr \) on \(\fg[\![x]\!]\).
\end{lemma}

\begin{proof} 
Consider a formal \( r \)-matrix
\(r(x,y) = s(y)r_{\textnormal{Yang}}(x,y) + g(x,y)\),
for some \( s \in F[\![y]\!] \) and \( g \in (\fg \ot \fg)[\![x,y]\!] \). First of all, observe that the invariance of \(\Omega\) implies that for all \(f = \sum_{j = 0}^\infty f_jx^j\in \fg[\![x]\!]\)
\begin{align*}
    [f\fot 1 + 1 \fot f,r_{\textnormal{Yang}}] = \sum_{j = 0}^\infty\Bigg( \underbrace{\frac{x^j-y^j}{x-y}}_{\in F[x,y]}[f_j\otimes 1,\Omega] + \frac{y^j}{x-y} \underbrace{[f_j \otimes 1 + 1 \otimes f_j,\Omega]}_{= 0}\Bigg)
\end{align*}
is an element of \((\fg \otimes \fg)[\![x,y]\!]\). Therefore, \(dr \colon \fg[\![x]\!] \to (\fg \otimes \fg)[\![x,y]\!]\) is well-defined.

Let us now prove that \( dr\) is skew-symmetric.
For that it is enough to prove
the skew-symmetry of \( r \).
The latter result is obtained by
repeating the arguments from 
the proof of
\cite[Proposition 4.1]{BD_first}. 
Namely, by definition \(\overline{\textnormal{CYB}}(r) \in (\fg \otimes \fg \otimes \fg)(\!(x_1)\!)(\!(x_2)\!)[\![x_3]\!]\).
Swapping variables \(x_1 \) and \(x_2\) 
and applying the \(F(\!(x_1)\!)(\!(x_2)\!)[\![x_3]\!]\)-linear extension of \(\tau \otimes 1\) to \(\overline{\textnormal{CYB}}(r)\) results in 
\begin{equation}
    -[\overline{r}^{12}(x_1, x_2), r^{23}(x_2, x_3)]
    -[\overline{r}^{12}(x_1, x_2), r^{13}(x_1, x_3)] +
    [r^{23}(x_2, x_3), r^{13}(x_2, x_3)] = 0,
\end{equation}
where \(\overline{r}(x,y) \coloneqq s(x)r_{\textnormal{Yang}}(x,y)-\tau g(y,x)\). Adding this to \(\overline{\textnormal{CYB}}(r)\) yields
\begin{equation}
    [r^{12}(x_1,x_2) - \overline{r}^{12}(x_1, x_2), r^{13}(x_1,x_3) + r^{23}(x_2, x_3)] = 0.
\end{equation}
Multiplying with \((x_1-x_3)\) and putting \(x_1 = x_3\) results in 
\begin{equation}
    [r^{12}(x_1,x_2) - \overline{r}^{12}(x_1, x_2),\Omega^{13}] = 0.
\end{equation}
Using the fact that \(\sum_{\alpha = 1}^n\textnormal{ad}(I_\alpha)^2 = 1\), we obtain \(r = \overline{r}\) by applying the map \(a \otimes b \otimes c \mapsto [a,c]\otimes b\) coefficientwise. 

It remains to prove the co-Jacobi identity for \(\delta\). 
For any \( f \in \fg[\![x]\!] \)
and \( a \in (\fg \ot \fg \ot \fg)(\!(x_1)\!)(\!(x_2)\!)[\![x_3]\!] \) we define 
\begin{equation}
    f \cdot a \coloneqq [f \fot 1 \fot 1 + 1 \fot f \fot 1 + 1 \fot 1 \fot f, a].
\end{equation}
Then 
\begin{equation}%
\label{eq:f_applied_to_r12_13}
\begin{aligned}
    f \cdot [r^{12}(x_1, x_2), r^{13}(x_1, x_3)] = \sum_{j,k = 0}^\infty \sum_{\alpha,\beta = 1}^n&\left( [f(x_1),[r_{j,\alpha}(x_1),r_{k,\beta}(x_1)]] \otimes x_2^jI_{\alpha} \otimes x_3^kI_{\beta} \right. \\&\left.+[r_{j,\alpha}(x_1),r_{k,\beta}(x_1)] \otimes [f(x_2),x_2^jI_{\alpha}] \otimes x_3^kI_{\beta} \right.\\&\left.+[r_{j,\alpha}(x_1),r_{k,\beta}(x_1)] \otimes x_2^jI_{\alpha} \otimes [f(x_3),x_3^kI_{\beta}] \right).
\end{aligned}
\end{equation}
Using the identity
\[
[f(x_1),[r_{j,\alpha}(x_1),r_{k,\beta}(x_1)]] = [[f(x_1),r_{j,\alpha}(x_1)],r_{k,\beta}(x_1)] + [r_{j,\alpha}(x_1),[f(x_1),r_{k,\beta}(x_1)]]
\] we can rewrite
\cref{eq:f_applied_to_r12_13}
in the form
\begin{equation}\label{eq:first_action}
    f \cdot [r^{12}(x_1, x_2), r^{13}(x_1, x_3)] = [dr(f)^{12}(x_1,x_2),r^{13}(x_1,x_3)] + [r^{12}(x_1,x_2),dr(f)^{13}(x_1,x_3)].
\end{equation}
Similarly, we have
\begin{equation}%
\label{eq:other_two_actions}
    \begin{aligned}
    f \cdot [r^{12}(x_1, x_2), r^{23}(x_2, x_3)] &= [dr(f)^{12}(x_1,x_2),r^{23}(x_2,x_3)] + [r^{12}(x_1,x_2),dr(f)^{23}(x_2,x_3)],\\
    f \cdot [r^{13}(x_1, x_3), r^{23}(x_2, x_3)] &= [dr(f)^{13}(x_1,x_3),r^{23}(x_2,x_3)] + [r^{13}(x_1,x_3),dr(f)^{23}(x_2,x_3)].
    \end{aligned}
\end{equation}
Summing up left and right 
sides of \cref{eq:other_two_actions,eq:first_action}
respectively and using
relations of the form
\begin{equation}
    [r^{12}(x_1,x_2),dr(f)^{13}(x_1,x_3) + dr(f)^{23}(x_2,x_3)] = - (dr \fot 1)dr(f)(x_1,x_2,x_3)
\end{equation}
we obtain the desired equality
\(
0 = f \cdot \overline{\textnormal{CYB}}(r) = 
   - \overline{\textnormal{Alt}}((dr \fot 1) dr(f))
\).
\end{proof}

\subsection{Formal
\( r \)-matrices
and Lagrangian Lie
subalgebras}%
\label{subsec:W_r_from_r}

Consider a formal
\( r \)-matrix 
of the form
\begin{equation}%
\label{eq:ym_r_matrix}
r(x,y) = 
\frac{y^m \Omega}{x-y} + \widetilde{g}(x,y)
= \sum_{k = 0}^\infty \sum_{\alpha = 1}^n x^{-k-1} I_\alpha \ot y^{k+m} I_\alpha  + \widetilde{g}(x,y),
\end{equation}
for some \( m \in \mathbb{Z}_{\ge 0} \)
and 
\(  \widetilde{g} \in (\fg \ot \fg)[\![x,y]\!] \). 
According to \cref{lem:rmat_defines_delta}
it defines a Lie bialgebra
structure \( dr \) 
on \( \fg[\![x]\!] \).
The goal of this section is
two-fold: First,
we want to prove that
\( \fD(\fg[\![x]\!], dr)
\sim \fg \ot A(m, 0)\).
More precisely,
\( \fD(\fg[\![x]\!], dr)
\cong \fg \ot A(m, 0)\)
as Lie algebras and the bilinear
form on \( \fD \) corresponds
to \( -B_{m+1} \) on 
\( \fg \ot A(m, 0) \);
Secondly, we want to construct
the Lagrangian Lie subalgebra
\( W_r \subset \fg \ot A(m, 0) \)
corresponding to \( dr \).

\paragraph{\textbf{Construction of \( W_r \)}}
Using the given \( r \)-matrix \( r \)
we now construct a Lagrangian
Lie subalgebra \( W_r \subset \fg \ot A(m, 0) = \fg(\!(x)\!) \times \fg[x]/x^m\fg[x] \)
such that
\( \fg \ot A(m, 0) = \fg[\![x]\!] \add W_r \). 
We recall
that \(\fg \ot A(m, 0)\) is equipped with the bilinear form \(B_{m+1}\) defined in \cref{rem:B_i}.

We start by writing
\begin{equation}\label{eq:gxy_expansion}
g(x,y) 
= \sum_{k = 0}^{\infty} \sum_{\alpha = 1}^n  p^{\alpha}_k (x) \ot I_\alpha y^k.
\end{equation}
Substituting this into \cref{eq:ym_r_matrix} we obtain
\begin{equation}\label{eq:Taylor_at_y}
r(x,y) =
\sum_{k = 0}^{m-1} \sum_{\alpha = 1}^n p^{\alpha}_k(x) \ot I_\alpha y^k + 
\sum_{k = m}^{\infty} \sum_{\alpha = 1}^n \left(I_\alpha x^{m-k-1} + p^{\alpha}_k(x) \right) \ot I_\alpha y^k.
\end{equation}
Similarly, representing \( (x-y)^{-1} \) as series \( - \sum_{k = 0}^\infty x^k y^{-k-1} \) we get another formal \( r \)-matrix in \( (\fg \ot \fg)(\!(y)\!)[\![x]\!] \):
\begin{equation}\label{eq:Taylor_at_x}
\begin{aligned}
r(x,y) =
&-\sum_{k = 0}^{\infty} \sum_{\alpha = 1}^n I_\alpha x^{m + k} \ot I_\alpha y^{-k-1} \\
&-\sum_{k = 0}^{m-1} \sum_{\alpha=1}^n \left( I_\alpha x^{m-k-1} + p^\alpha_k(x) \right) \ot I_\alpha y^k
+\sum_{k = m}^\infty \sum_{\alpha = 1}^n p^\alpha_k(x) \ot I_\alpha y^k.
\end{aligned}
\end{equation}
We define \( W_r \) by combining the coefficients in front of \( I_\alpha y^k, k \ge 0 \)
in \cref{eq:Taylor_at_y,eq:Taylor_at_x}, namely
\begin{equation}
\begin{aligned}
W_r \coloneqq \Span_F \big\{w_{k,\alpha} \big| \, k \ge 0, \ 1\le \alpha \le n \big\}\subset \fg(\!(x)\!) \times \fg[x]/x^m\fg[x],
\end{aligned}
\end{equation}
where 
\begin{equation}
    w_{k,\alpha} = \begin{cases}
     \left(p^\alpha_k(x), - I_\alpha [x]^{m - k - 1} + p^\alpha_k([x])\right)&0 \le k \le m-1, \ 1 \le \alpha \le n,
     \\
     \left(I_\alpha x^{-k+m-1} + p^\alpha_{k}(x), p^\alpha_{k}([x]) \right)&k \ge m, \ 1 \le \alpha \le n ;
    \end{cases}
\end{equation}
Here \( [x] \) denotes the equivalence class of \( x \) in \( F[x]/(x^m) \).
It is clear that \( W_r \) is complementary to the diagonal embedding
\( \Delta \coloneqq \{(f,[f])\mid f\in \fg[\![x]\!]\} \)
of \( \fg[\![x]\!] \) into \( \fg(\!(x)\!) \times \fg[x]/x^m\fg[x] \).
Let \(\{b_{k,\alpha} \coloneqq (I_\alpha x^k, I_\alpha [x]^k) \mid 1 \le \alpha \le n \text{ and } k \ge 0 \}\) 
be a topological basis
for \( \Delta \).
Then by direct calculation we see
that
\begin{equation}\label{eq:topological_basis_duality}
    B_{m+1}(w_{j,\alpha},b_{k,\beta}) = \delta_{jk}\delta_{\alpha\beta}  
\end{equation}
for all \(j,k\ge 0\) and \(1\le \alpha,\beta \le n\).

\paragraph{\textbf{Lagrangian property of \( W_r \)}}
Let us expand each series \( p^\alpha_k(x) \) in
\cref{eq:gxy_expansion} as
\begin{equation}
p^{\alpha}_k(x) = \sum_{s = 0}^\infty \sum_{\beta = 1}^n c^{\alpha, \beta}_{k, s} I_\alpha x^s.
\end{equation}
The skew-symmetry of \( r(x,y) \) gives restrictions on the coefficients
\( c^{\alpha, \beta}_{k, s} \). More precisely,
\begin{align*}
  g(x,y) + \tau(g(y,x)) 
  &= \sum_{i,j = 0}^\infty \sum_{\alpha, \beta = 1}^n (c^{\alpha, \beta}_{i,j} + c^{\beta,\alpha}_{j,i}) I_\alpha x^i \ot I_\beta y^j \\
  &= \frac{(x^m - y^m)\Omega}{x-y} \\
  &= \sum_{k = 0}^{m-1} \sum_{\alpha = 1}^n I_\alpha x^{m-k-1} \ot I_\alpha y^{k}.
\end{align*}
Therefore, we have
\begin{equation}
c^{\alpha, \beta}_{i,j} + c^{\beta, \alpha}_{j,i} = \delta_{\alpha,\beta} \delta_{i+j,m-1}.
\end{equation}
Using these identities we can show that \( W_r \) is isotropic by proving \(B_{m+1}(w_{j,\alpha},w_{k,\beta}) = 0\) for all \(k,j \ge 0\) and \(1\le \alpha,\beta \le n\). 
For example, when \(j,k \le m-1\) we get
\begin{align*}
  B_{m+1} \Big((p^\alpha_i(x), - I_\alpha [x]^{m - i - 1} + p^\alpha_i([x])), \, &(p^\beta_j(x), - I_\beta [x]^{m - j - 1} + p^\beta_j([x])) \Big) \\
  &= c^{\alpha, \beta}_{i,j} + c^{\beta, \alpha}_{j,i} - \underbrace{\kappa(I_\alpha, I_\beta)}_{= \delta_{\alpha,\beta}} \underbrace{\textnormal{coeff}_{m-1}\left( x^{(m-i-1)+(m-j-1)}  \right)}_{= \delta_{i+j,m-1}} \\
  &= 0.
\end{align*}
The isotropy of \( W_r \) combined with \( \Delta \add W_r = \fg(\!(x)\!) \times \fg[x]/x^m\fg[x] \)
implies that \( W_r \) is Lagrangian.
\paragraph{\textbf{An \( r \)-matrix in \( (\fg \ot \fg)[\![x,y]\!]/(x^m, y^m) \)}}
We denote the first sum in \cref{eq:Taylor_at_x} by \( r_- \)
and the other two by \( r_+ \), i.e.
\begin{equation}
\begin{aligned}
r(x,y) =
&\overbrace{-\sum_{k = 0}^{\infty} \sum_{\alpha = 1}^n I_\alpha x^{m + k} \ot I_\alpha y^{-k-1}}^{\eqqcolon r_-} \\
&\underbrace{-\sum_{k = 0}^{m-1} \sum_{\alpha=1}^n \left( I_\alpha x^{m-k-1} + p^\alpha_k(x) \right) \ot I_\alpha y^k
+\sum_{k = m}^\infty \sum_{\alpha = 1}^n p^\alpha_k(x) \ot I_\alpha y^k}_{\eqqcolon r_+}.
\end{aligned}
\end{equation}
The Classical Yang-Baxter equation for \( r \) then writes as
\begin{align*}
  0 = \overline{\textnormal{CYB}}(r) = \overline{\textnormal{CYB}}(r_+) + &\overbrace{\overline{\textnormal{CYB}}(r_-) 
  + [r_+^{12}, r_-^{13}] + [r_-^{12}, r_+^{13}] + [r_+^{12}, r_-^{23}]}^{\text{negative powers in one of the components}} \\
  &+ \underbrace{[r_+^{13}, r_-^{23}] + [r_-^{13}, r_+^{23}] + [r_-^{12}, r_+^{23}]}_{\text{each tensor contains } x^k, k \ge m}.
\end{align*}
It can be considered as an element in \( ( \fg \ot \fg \ot \fg)(\!(y)\!)(\!(z)\!)[\![x]\!] \).
Projecting it first onto \( ( \fg \ot \fg \ot \fg)[\![x, y, z]\!] \) and then
taking the quotient by \( (x^m, y^m, z^m) \)
we obtain the equality \( \overline{\textnormal{CYB}}(r_+) = 0 \).
In particular, \( r_+([x], [y]) \) is a formal \( r \)-matrix
in \( (\fg \ot \fg)[\![x,y]\!]/(x^m, y^m) \).

\paragraph{\textbf{Subalgebra property of \( W_r \)}}
The previous result implies that
\begin{align*}
  R\coloneqq \sum_{j = 0}^\infty \sum_{\alpha = 1}^n w_{j,\alpha} \otimes b_{j,\alpha} \in 
  (W \ot \fg) [\![ \, (y,[y]) \, ]\!]
\end{align*}
is an \( r \)-matrix, i.e.\ it 
formally solves the CYBE because
both right and left components
of \( R \) do that.
The fact that \( W_r \) is a subalgebra now follows from the identity
\([R^{12}, R^{13}] = -[R^{12} + R^{13}, R^{23}]\) 
that can be rewritten as
\begin{align}\label{eq:CYBE_in_R}
   \sum_{j,k = 0}^\infty \sum_{\alpha,\beta = 1}^n [w_{j,\alpha},w_{k,\beta}] \otimes b_{j,\alpha} \otimes b_{k,\beta} = -\sum_{j = 0}^\infty \sum_{\alpha = 1}^n w_{j,\alpha} \otimes dR(b_{j,\alpha}) \in (W \ot \fg \ot \fg) [\![ \, (x_2, [x_2]), (x_3, [x_3]) \, ]\!].
\end{align}
Here \(dR \colon \Delta \to \Delta \fot \Delta\) is defined by \(
    dR(f) = [f \fot 1 + 1 \fot f,R]
    \)
for all \(f \in \Delta\). Observe that the diagonal embedding \(\fg[\![x]\!] \to \Delta\) defines an isomorphism \((\fg[\![x]\!],dr) \cong (\Delta,dR)\) of topological Lie bialgebras.

\paragraph{\textbf{The topological double of \(dr\)}} 
Applying 
\begin{equation}
    B^{\fot 3}_{m+1}(b_{j_1,\alpha_1}\otimes w_{j_2,\alpha_2}\otimes w_{j_3,\alpha_3},-)    
\end{equation}
to \cref{eq:CYBE_in_R} and using \cref{eq:topological_basis_duality} results in
\begin{align*}
    B^{\fot 2}_{m+1}(dR(b_{j_1,\alpha_1}),w_{j_2,\alpha_2} \otimes w_{j_3,\alpha_3}) = - B_{m+1}(b_{j_1,\alpha_1},[w_{j_2,\alpha_2}, w_{j_3,\alpha_3}])
\end{align*}
for all \(j_1,j_2,j_3 \ge 0\) and \(1 \le \alpha_1,\alpha_2,\alpha_3 \le n\).
This implies that 
\begin{align*}
    B^{\fot 2}_{m+1}(dR(f),v \otimes w) = -B_{m+1}(f,[v, w])
\end{align*}
for all \(f\in\Delta\) and \(v,w\in W_r\). It is easy to see that \(\fg[\![x]\!]' \cong \Delta' \cong W_r\) using \(B_{m+1}\). \Cref{rem:doubles_in_series_case} now states that
\begin{equation}
    (\fD(\fg[\![x]\!],-dr),\fg[\![x]\!],\fg[\![x]\!]') \cong (\fg(\!(x)\!)\times \fg[x]/x^m\fg[x],\Delta,W_r).  
\end{equation}

\subsection{Realization of all topological doubles}\label{subsec:realization_of_doubles_polynomials}
\label{sec:realization_polynomials}
It is well-known that the following four series 
\begin{equation}%
\label{eq:four_main_r_matrices}
\begin{aligned}
  r_0(x,y) \coloneqq 0, \ r_1(x,y) \coloneqq \frac{\Omega}{x-y}, \ r_2(x,y) \coloneqq \frac{y\Omega}{x-y} + r_{\text{DJ}}, \ r_3(x,y) \coloneqq \frac{xy\Omega}{x-y} = \frac{y^2\Omega}{x-y}+y\Omega
\end{aligned}
\end{equation}
are formal \(r\)-matrices. Here \( r_{\text{DJ}}\in \fg \ot \fg \)
is the Drinfeld-Jimbo classical \( r \)-matrix with respect to a triangular decomposition \(\fg = \fn_+ \dotplus \fh \dotplus \fn_-\). 
Therefore, the discussion in \cref{subsec:W_r_from_r} implies that \(\delta_0 \coloneqq -dr_0 = 0\) and \(\delta_{m+1}\coloneqq -dr_{m+1}\)
define topological Lie bialgebra structures on \(\fg[\![x]\!]\) with topological doubles 
\begin{align}\label{eq:doubles_realized}
  \fD(\fg[\![x]\!], \delta_0) \cong \fg \ot A(\infty) \ \text{ and } \
  \fD(\fg[\![x]\!], \delta_{m+1}) \cong \fg \ot A(m, 0) = \fg(\!(x)\!)\times \fg[x]/x^{m}\fg[x] 
\end{align}
for \(m \in \{0,1,2\}\). These isomorphisms are identical on the images of \( \fg[\![x]\!] \) under the respective natural embeddings.

Using the construction
of \( W_r \) presented in
\cref{subsec:W_r_from_r}
we can explicitly describe
the Lagrangian
Lie subalgebras \(W_i \)
defining
\( \delta_i \), \( i \in \{0,1,2,3\} \):
\begin{itemize}
\setlength\itemsep{1em}
    \item \(W_0 \coloneqq \sum_{j \ge 0}\fg \otimes a_j \subseteq \fg \otimes A(\infty)\), where the elements \(a_j \in A(\infty)\) were defined in \cref{ex:A(inf)};
    
    \item \(W_1 \coloneqq x^{-1}\fg[x^{-1}] \subseteq \fg(\!(x)\!) \cong \fg(\!(x)\!) \times \{0\}\);
    
    \item \( W_2 \coloneqq \{(a,b) \in (\fn_+ \dotplus \fh \dotplus x^{-1}\fg[x^{-1}]) \times (\fh\dotplus\fn_-)\mid a + b \in (\fn_+ \dotplus \fn_-)\} \subseteq \fg(\!(x)\!) \times \fg \)
    where \(\fg = \fn_+ \dotplus \fh \dotplus \fn_-\)
    is the triangular decomposition
    used for the construction
    of \(r_{\textnormal{DJ}}\);
    
    \item \(W_3 \coloneqq  \fg[x^{-1}] \times [x]\fg \subseteq \fg(\!(x)\!)\times \fg[x]/x^2\fg[x]\). 
\end{itemize}

Let \( \varphi \in \Aut_0 F[\![x]\!] \) be a coordinate transformation and \(m \in \{0,1,2\}\).
Then the map \( (\varphi \fot \varphi) \circ \delta_{m+1} \circ \varphi^{-1} \)
is a topological Lie bialgebra structure on \(\fg[\![x]\!]\)
with topological double isomorphic to
\(\fg \ot A(m, 0)^{(\varphi)}\).
Therefore, all Lie algebras (with forms) mentioned in
\cref{rem:four_doubles} are realizable as
topological doubles of topological Lie bialgebra structures
on \( \fg[\![x]\!] \).

\subsection{Formal \(r\)-matrices and topological twists.}
The following statements 
describe how the conditions
on \( s \in (\fg \ot \fg)[\![x,y]\!] \) for being
a topological twist
look like on the side
of formal \(r\)-matrices.

\begin{lemma}%
\label{lem:delta_s_r_s_correspondence}
Let \(r\) be a formal
\(r\)-matrix and \( \delta \coloneqq dr \) be the
corresponding Lie bialgebra
structure on \( \fg[\![x]\!] \).
Then \( s \in (\fg \ot \fg)[\![x,y]\!] \) is a topological
twist of \( \delta \) if
and only if \( r_s \coloneqq r + s \) is a formal \(r\)-matrix.
Moreover, if this is the case,
then \( \delta_s = dr_s \).
\end{lemma}

\begin{proof}
The first part of the statement follows from the
formal equality in \( (\fg \ot \fg \ot \fg)(\!(x)\!)(\!(y)\!)[\![z]\!] \)
\begin{equation}
    \overline{\textnormal{CYB}}(r + s) = 
    \overline{\textnormal{CYB}}(r) + \overline{\textnormal{CYB}}(s) -
    \overline{\textnormal{Alt}}((\delta \fot 1)s).
\end{equation}
Indeed,
\(\overline{\textnormal{CYB}}(r) = 0\) implies that 
\(\overline{\textnormal{CYB}}(r + s) = 0\) is equivalent to \(\overline{\textnormal{CYB}}(s)
= \overline{\textnormal{Alt}}((\delta \fot 1)s)\).
    
The second part follows by
definition of \( ds \), namely
\begin{equation}
\begin{aligned}
      (\delta_s)(f)(x,y) &= \delta(f)(x,y) +
      ds(f)(x,y) \\
      &= [f(x) \fot 1 + 1 \fot f(y), r(x,y)] +
      [f(x) \fot 1 + 1 \fot f(y), s(x,y)] \\
      &= [f(x) \fot 1 + 1 \fot f(y), r_s(x,y)].
\end{aligned}
\end{equation}
\end{proof}

\begin{remark}
Let \( s \) be a topological tiwst of \( \delta_i \). By
straightforward calculations
one can prove that
the Lagrangian Lie subalgebra 
\( W_s \) associated to \(s\) in \cref{thm: W <-> t <-> T correspondence} (see \cref{rem:construction_of_Ws}) coincides with \(W_{-r_i + s}\) presented in \cref{subsec:W_r_from_r}. Therefore,
when \( i \) is fixed
we use the notations
\( W_s \) and \(W_{-r_i + s} \)
interchangeably.
\end{remark}

\begin{corollary}%
\label{cor:bijection_between_matrices_and_W}
  The assignment \(r \mapsto W_r\) presented in \cref{subsec:W_r_from_r} defines a bijection between Lagrangian Lie
  subalgebras \( W \subset \fD(\fg[\![x]\!], \delta_i), i \in \{0,1,2,3\} \) complementary to \(\fg[\![x]\!]\), and
  formal \( r \)-matrices of the form
  \( r = r_i + s \).
\end{corollary}


\subsection{Reinterpretation of the classification of doubles}
The classification of classical doubles has interesting consequences for the form of formal \(r\)-matrices.

\begin{theorem}%
\label{thm:restriction_on_powers}
Let \( r \) be a formal \( r \)-matrix 
of the form 
\begin{equation}\label{eq:rmatrix_theorem}
r(x,y) = \frac{s(y) \Omega}{x-y} + g(x,y)
\end{equation}
for some \(s \in x^m F[\![x]\!]^\times\) and \(g \in (\fg \otimes \fg)[\![x,y]\!]\).
Then \( m \in \{ 0, 1, 2 \} \) and there exists \(\psi \in xF[\![x]\!]^\times\), \(\xi \in F^\times\) and \(\widetilde{g} \in (\fg \otimes \fg)[\![x,y]\!]\) such that
\begin{equation}
    \xi r(\psi(x),\psi(y)) = \frac{y^m\Omega}{x-y} + \widetilde{g}(x,y).
\end{equation}
\end{theorem}

\begin{proof}
\Cref{lem:rmat_defines_delta} implies that \(r\) defines a topological Lie bialgebra structure \(dr \)
on \(\fg[\![x]\!]\).
By \cref{rem:four_doubles} there is a constant \(\xi \in F^\times\),
an integer \(0 \le m \le 2\), a coordinate transformation
\(\varphi \in \textnormal{Aut}_{0}F[\![x]\!]\) and an \(F[\![x]\!]\)-linear
isomorphism
\(\fD(\fg[\![x]\!], \xi \delta) \cong \fg \ot A(m, 0)^{(\varphi)}  \)
identical on \(\fg[\![x]\!]\).
This means that \( \xi (\varphi^{-1} \fot \varphi^{-1}) dr \ \varphi = \delta_{m+1} + dt\) for some topological twist \(t\).
Put \( \psi \coloneqq \varphi^{-1} \). Since \(\delta_{m+1} + dt\)
is given by the \(r\)-matrix \(-r_{m+1} + t\) we have
\begin{equation}
    \begin{split}
        &[f(x) \fot 1 + 1 \fot f(y),\xi r(\psi(x),\psi(y))] =  \xi (\psi \fot \psi) dr \ \varphi(f) \\&= (\delta_{m+1} + dt)(f) = [f(x) \fot 1 + 1 \fot f(y),-r_{m+1}(x,y) + t(x,y)]
    \end{split}
\end{equation}
Putting \(f = a \in \fg\) implies that 
\begin{equation}
    \xi r(\psi(x),\psi(y))   +r_{m+1}(x,y) - t(x,y) = h(x,y) \Omega
\end{equation}
for some \(h \in F(\!(x)\!)[\![y]\!]\), since the \(\fg\)-invariant elements of \(\fg \otimes \fg\) are spanned by \(\Omega\). 
Letting
\(f = ax \in \fg[\![x]\!]\) we see that 
\begin{equation}
    0 = [ax \otimes 1 + 1 \otimes ay, h(x,y) \Omega ] = (x-y)h(x,y) [a\otimes 1,\Omega],
\end{equation}
which implies that \(h = 0\) since \([a\otimes 1,\Omega] \neq 0\) for \(a \neq 0\). In particular,
\begin{equation}%
\label{eq:r_psi_r_i}
    \xi r(\psi(x),\psi(y)) =  -r_{m+1}(x,y) + t(x,y). 
\end{equation}
We can always write
\begin{equation}
    \frac{s(\psi(y))}{\psi(x)-\psi(y)} = \frac{s(\psi(y))}{\psi'(y)(x-y)} + u(x,y)
\end{equation}
for some \(u(x,y) \in (\fg \ot \fg)[\![x,y]\!]\). 
Combining this
with \cref{eq:r_psi_r_i}
we obtain
\begin{equation}
    \xi \frac{s(\psi(y))}{\psi'(y)} = -y^m.
\end{equation}
By taking \( - \xi \)
instead of \( \xi \)
we conclude the proof.
\end{proof}

\begin{remark}
Let
\begin{equation}%
\label{eq:r_matrix_diff_eq}
    \frac{s(y) \Omega}{x-y} + g(x,y),
\end{equation}
where \( s(y) = y^m + a_{m+1}y^{m+1} + \dots \),
\( m \in \{0,1,2\} \),
be a formal \(r\)-matrix.
One can directly see from the proof of \cref{thm:restriction_on_powers}
that the equivalence
\( \psi \in \textnormal{Aut}_{0}F[\![x]\!]\),
used to transform \( s(y) \) into 
\( 1, y \) or \( y^2 \),
is a solution to
the differential equation
\begin{equation}%
\label{eq:diff_equation_remark}
     y^m \psi'(y) = s(\psi(y)).
\end{equation}
Note, that for \( m = 0, 1 \)
\cref{eq:diff_equation_remark}
always has a solution, whereas for
\( m = 2\) a solution exists if and only
if \( a_3 = 0 \).

Conversely, if
\( s(y) = y^m + a_{m+1}y^{m+1} + \dots \),
\(m \in \{0,1,2 \}\) and
\cref{eq:diff_equation_remark} has a
solution, then we can find
and \(r\)-matrix of the form
\cref{eq:r_matrix_diff_eq}.
\end{remark}

\begin{remark}\label{rem:r->dr_is_bijection}
The methods in the proof of \cref{thm:restriction_on_powers} combined with the previous discussion of formal \(r\)-matrices imply that the assignment \(r \mapsto dr\) defines a bijection between the set of formal \(r\)-matrices and the set of topolgical Lie bialgebra structures on \(\fg[\![x]\!]\). 
\end{remark}

\subsection{Equivalences} Two formal \( r \)-matrices \( r \) and 
\( \tilde{r} \)
are called \emph{formally gauge equivalent} if
there is an \( F[\![x]\!] \)-linear automorphism
\( \phi \) of \( \fg[\![x]\!] \) such
that 
\begin{equation*}
  \tilde{r} = (\phi \fot \phi)r.
\end{equation*}
At this stage we have defined three notions
of formal equivalence: between topological twists, Lagrangian
Lie subalgebras and formal \( r \)-matrices.
Now we turn to describing the relation between these notions.

\begin{theorem}%
\label{thm:equivalence}
Let \( \phi \) be an \( F[\![x]\!] \)-linear
automorphism of \( \fg[\![x]\!] \),
\( s \) and \( \tilde{s} \) be
topological twists of the topological Lie bialgebra
structure \( \delta_i \), \( i \in \{ 1,2,3 \} \),
on \( \fg[\![x]\!] \).
The following are equivalent:
\begin{enumerate}
    \item \( (\phi \fot \phi)(\delta_i + ds) = (\delta_i + d\tilde{s}) \phi \);
    \item \( (\phi \times [\phi])(W_s) = W_{\tilde{s}} \);
    \item \((\phi \fot \phi)(-r_i + s) = -r_i + \tilde{s} \).
\end{enumerate}
\end{theorem}

\begin{proof}
\emph{1 \(\Longleftrightarrow\) 3 : }
Using \( F[\![x]\!] \)-linearity of \( \phi \) and
\cref{lem:delta_s_r_s_correspondence} condition \(1\) reads as
\begin{equation}
\begin{aligned}
(\phi \fot \phi)(\delta_i + ds)(f) &=  (\phi \fot \phi)[f \fot 1 + 1 \fot f, -r_i + s] \\
&= [\phi(f) \fot 1 + 1 \fot \phi(f), (\phi \fot \phi)(-r_i + s)] \\
&= [\phi(f) \fot 1 + 1 \fot \phi(f), -r_i + \tilde{s}].
\end{aligned}
\end{equation}
This is equivalent to
\( [f \fot 1 + 1 \fot f, (-r_i + \tilde{s}) - (\phi \fot \phi)(-r_i + s)] = 0  \) for all \( f \in \fg[\![x]\!] \), which means that
\( (r_i + \tilde{s}) - (\phi \fot \phi)(r_i + s) = 
h \Omega \), for some \( h \in F(\!(x)\!)[\![y]\!] \). Hence,
\begin{equation}
\begin{aligned}
   0 &=  h(x,y) [f(x) \fot 1 + 1 \fot f(y), \Omega] \\
   &= h(x,y) [(f(x) - f(y)) \fot 1, \Omega]
\end{aligned}
\end{equation}
for all \( f \in \fg[\![x]\!] \). 
Since \( [a \ot 1, \Omega] \neq 0 \) for any 
non-zero \( a \in \fg \) we must have \( h = 0 \)
implying the desired equality
\( (\phi \fot \phi)(-r_i + s) = -r_i + \tilde{s} \).
The converse is clear.

\emph{2 \(\Longleftrightarrow\) 3 : }
Write two series expansions of \(-r_i + s\) at \( y = 0 \) and \( x = 0 \)
respectively as follows
\begin{equation}
\label{eq:series_expansion_at_y_and_x}
\begin{dcases}
\sum_{k = 0}^{\infty} \sum_{\alpha=1}^{n} r^{\alpha}_k(x) \ot I_{\alpha} y^k, \\
\sum_{k = 0}^{\infty} \sum_{\alpha = 1}^n I_\alpha x^{k + i -1} \ot I_\alpha y^{-k-1} + \sum_{k = 0}^{\infty} \sum_{\alpha=1}^{n} \widetilde{r}^{\alpha}_k(x) \ot I_{\alpha} y^k.
\end{dcases}
\end{equation}
Let \( \phi(I_\alpha) = \sum_{j = 0}^{\infty} \sum_{\gamma = 1}^n c^{\gamma}_{j, \alpha} I_\gamma x^j\).
Applying \( \phi \fot \phi \) to \cref{eq:series_expansion_at_y_and_x}
we get the series
\begin{equation*}
\begin{dcases}
\sum_{k = 0}^{\infty} \sum_{\alpha=1}^{n} \sum_{j = 0}^{\infty} \sum_{\gamma=1}^{n} c^{\gamma}_{j, \alpha} \phi(r^{\alpha}_k(x)) \ot I_{\gamma} y^{k+j}, \\
\sum_{k = 0}^{\infty} \sum_{\alpha = 1}^n \sum_{j = 0}^{\infty} \sum_{\gamma=1}^{n} c^{\gamma}_{j, \alpha} \phi(I_\alpha x^{k + i -1}) \ot I_\gamma y^{-k-1 + j}
+ \sum_{k = 0}^{\infty} \sum_{\alpha=1}^{n} \sum_{j = 0}^{\infty} \sum_{\gamma=1}^{n} c^{\gamma}_{j, \alpha} \phi(\widetilde{r}^{\alpha}_k(x)) \ot I_{\gamma} y^{k+j}.
\end{dcases}
\end{equation*}
Put \(t \coloneqq k+j\). This gives
\begin{equation*}
\begin{dcases}
\sum_{t = 0}^{\infty} \sum_{\gamma=1}^{n} \left( \sum_{k = 0}^{\infty} \sum_{\alpha=1}^{n} c^{\gamma}_{t-k, \alpha} \phi(r^{\alpha}_k(x)) \right) \ot I_{\gamma} y^{t}, \\
\sum_{k = 0}^{\infty} \sum_{\alpha = 1}^n \sum_{j = 0}^{\infty} \sum_{\gamma=1}^{n} c^{\gamma}_{j, \alpha} \phi(I_\alpha x^{k + i -1}) \ot I_\gamma y^{-k-1 + j}
+ \sum_{t = 0}^{\infty} \sum_{\gamma=1}^{n} \left( \sum_{k = 0}^{\infty} \sum_{\alpha=1}^{n} c^{\gamma}_{t-k, \alpha} \phi(\widetilde{r}^{\alpha}_k(x)) \right) \ot I_{\gamma} y^{t},
\end{dcases}
\end{equation*}
where we let \( c^{\gamma}_{l, \alpha} \coloneqq 0 \)
for \( l < 0\). 
Therefore, the Lagrangian Lie subalgebra
in \(\fg \ot A(i-1, 0)\) corresponding to
\( (\phi \fot \phi)(-r_i + s) = -r_i + \tilde{s} \) is 

\begin{equation*}
\begin{aligned}
W_{\tilde{s}} \coloneqq \Span_F \left\{ \sum_{k=1}^{\infty} \sum_{\alpha=1}^{n} c^{\gamma}_{t-k, \alpha}
\big( \phi(r^{\alpha}_{k}(x)), [\phi]([\widetilde{r}^{\alpha}_{k}(x)]) \big) \,  \big| \, 0 \le t, \, 1 \le \gamma \le n  \right\}.
\end{aligned}
\end{equation*}
It is now clear that \( W_{\tilde{s}} \subseteq (\phi \times [\phi])W_{s} \). Since
\( \fg[\![x]\!] \add W_{\tilde{s}} = \fg[\![x]\!] \add (\phi \times [\phi])W_{s} = \fg \ot A(i-1, 0) \)
we have the equality
\begin{equation}
W_{(\phi \fot \phi)(-r_i + s)} = (\phi \times [\phi])W_{-r_i + s}.
\end{equation}
Together with bijection
\cref{cor:bijection_between_matrices_and_W}
we get the desired statement.
\end{proof}

Therefore, we have shown
that isomorphism classes of topological
twists of 
\( \delta_i, i \in \{1,2,3 \}, \)
are in one-to-one correspondence with
isomorphism classes of the
corresponding Lagrangian
Lie subalgebras
and gauge equivalence classes of
formal
\( r \)-matrices.

\section{Commensurable twists and their classification}%
\label{sec:commensurable_twists}
One natural subclass of topological twists of \(\delta_i\) for \(i \in \{0,1,2,3\}\) that admits a straightforward classification is the
class of topological twists in
\( \fg[x] \ot \fg[y] \subset (\fg \ot \fg)[\![x,y]\!] \).
We say that a subalgebra \( W \subseteq \fD(\fg[\![x]\!],\delta_i) \) is \emph{commensurable} with \( W_i \subseteq \fD \)
if \[ \dim(W_i + W)/(W_i \cap W) < \infty. \]
Adjusting slightly the argument in \cref{thm: W <-> t <-> T correspondence} we get the
following correspondence.

\begin{lemma}\label{lem:commensurable_twists}
Let \((\fg[\![x]\!],\delta_i)\), \(i \in \{0,1,2,3\}\), be the topological Lie bialgebra structures defined in \cref{subsec:realization_of_doubles_polynomials}.
  There are the following one-to-one correspondences:

 \begin{tikzpicture}[every node/.style={inner sep=.15cm,outer sep=0cm}]
  \node[] (A) {\mlnode{
    Lagrangian Lie subalgebras \( W \subseteq \fD \) \\
    complementary to \( \fg[\![x]\!] \) and \\
    commensurable with \( W_i \)
  }};
  \node[below left=1.5cm and -1.5cm of A] (B) {\mlnode{
    Topological twists \( s \) of \( \delta_i \) \\
    lying in \( \fg[x] \ot \fg[y] \)
  }};
    \node[below right=1cm and -1.5cm of A] (C) {\mlnode{
    Linear maps 
    \( T \colon W_i \longrightarrow \fg[\![x]\!] \) such that \\
    \( \dim(\im(T)) < \infty \) and for all \( p_1, p_2, p_3 \in W_i \) \\
    \( B_i(Tp_1, p_2) + B_i(p_1, Tp_2) = 0 \) and \\
    \( B_i([Tp_1 - p_1, Tp_2 - p_2], Tp_3 - p_3) = 0 \) 
  }};
  
\draw[<->] (A) edge (B) (B) edge (C) (C) edge (A);
\end{tikzpicture}

\end{lemma}
\begin{proof}
We start with the identification
\begin{equation}
   \fg[\![x]\!] \add \fg\{x\} \cong \begin{cases}
    \fg \otimes A(\infty)& \text{ if } i = 0, \\
    \fg \otimes A(i-1,0) & \text{ if } i \in \{1,2,3\},
    \end{cases}
\end{equation}
under which \( \fg\{x\} \cong W_i \) as Lie algebras.
It is enough to prove that any linear map
\( T \colon W_i \to \fg[\![x]\!] \)
satisfying the assumptions above
gives rise to a polynomial twist \( s \in \fg[x] \ot \fg[y] \).
As in the proof of \cref{thm: W <-> t <-> T correspondence}
let \( \{ Tp_j \}_{j = 1}^n \) be a basis for
\( \im(T) \) and \( \{ v_j \}_{j = 1}^n \subseteq \fg\{ x \} \)
be its dual basis.
The vector space \( \fg\{ x \}  \) decomposes as
\[ \textnormal{span}\{ p_j \mid 1 \le j \le n \} \add \ker(T). \]
Put \( m \coloneqq \max_j \{ \deg(p_j) \} \),
then it is clear that \( x^{m+1} \fg\{x\} \subseteq \ker(T) \). Here we recall that \(\fg\{x\} = \fg[\![x]\!]'\) can be identified with \(\fg[x]\) as a vector space.
Now we observe that for any \( 1 \le j \le n \)
we have the equality
\begin{align*}
  0 = B(v_j, T(\ker(T))) = -B(Tv_j, \ker(T)) = -B(Tv_j, x^{m+1} \fg\{x\}),
\end{align*}
which means that \( Tv_j \) is actually a polynomial of degree at most \( m \).
In other words, the tensor
\( s \coloneqq - \sum_{j=1}^n Tp_j \ot Tv_j \) lies in \( \fg[\![x]\!] \ot \fg[y] \).
By skew-symmetry we get \( s \in \fg[x] \ot \fg[y] \) as we wanted.
\end{proof}
With this result in mind we call topological twists in
\( \fg[x] \ot \fg[y] \) \emph{commensurable twists}.
It is clear that formal isomorphisms of twists do not
necessarily preserve their commensurability. 
Therefore, we will classify commensurable twists within
\( \fD(\fg[\![x]\!], \delta_i) \),
\( 1 \le i \le 3  \) up to \emph{polynomial isomorphisms},
i.e.\ \( F[x] \)-linear automorphisms of
\( \fg[x] \) extended to 
\( \fg[\![x]\!] \).

\subsection{Commensurable twists in \( \fg \ot A(0,0) \)}%
\label{subsec:comm_twists_double_1}
According to the results of 
\cref{sec:formal_r_matrices}
isomorphism classes of
commensurable twists of \( \delta_1 \) 
are in bijection with
gauge equivalence classes
of formal \emph{rational \( r \)-matrices}
\begin{equation}%
\label{eq:com_twist_rational_r}
r_p = r_{\textnormal{Yang}} + p \in (\fg \ot \fg)(\!(x)\!)[\![y]\!], 
\end{equation}
where \( p \in (\fg \ot \fg)[x,y] \) is
a skew-symmetric polynomial.

Let us for a moment
assume \( F = \bC \).
Then \( r_p \)
can be considered as a Taylor series
expansion at \( y = 0\)
of the non-degenerate classical \( r \)-matrix
\begin{equation}%
\label{eq:corresponding_meromorphic_function}
r_p(x,y) = \frac{\Omega}{x-y} + p(x,y).
\end{equation}
By \cite[Proposition 4.8]{abedin2021geometrization}
classical non-degenerate
\( r \)-matrices of the form
\cref{eq:corresponding_meromorphic_function}
are holomorphically gauge equivalent
if and only if their Taylor series
expansions at \( y = 0 \) are
formally gauge equivalent.
Furthermore, the statement of
\cite[Theorem A.3]{Raschid_Stepan_classical_twists}
tells us that two classical \( r \)-matrices of the form
\cref{eq:corresponding_meromorphic_function} are holomorphically gauge equivalent
if and only if they are polynomially
gauge equivalent.
Combining these results we get the
following statement.
\begin{lemma}
Put  \( F = \bC \). Then formal
\( r\)-matrices \( r_{p_1} \)
and \( r_{p_2} \) are formally gauge
equivalent if and only if
they are polynomially gauge equivalent.
More precisely, 
\(
r_{p_2} = (\phi \fot \phi)r_{p_1}
\)
for a \( \phi \in \textnormal{Aut}_{\bC[\![x]\!] \textnormal{-LieAlg}}(\fg[\![x]\!]) \) is equivalent to
\(
r_{p_2} = (\psi \fot \psi)r_{p_1}
\)
for some
\( \psi \in \textnormal{Aut}_{\bC[x] \textnormal{-LieAlg}}(\fg[x]) \).
\end{lemma}

Therefore, for \( F = \bC \)
the isomorphism classes of commensurable twists (up to polynomial isomorphisms)
are in bijection with
gauge equivalence classes of formal
\( r \)-matrices of the form
\cref{eq:com_twist_rational_r}.
Using
\cite[Theorem 1]{Stolin_rational}
we reduce the classification of
commensurable twists of 
\( \delta_1 \) to the classification of
particular Lagrangian Lie subalgebras in \( \fg(\!(x^{-1})\!) \).

\begin{theorem}%
\label{thm:stolin_orders}
Formal \( r \)-matrices of the form
\( r_p \) are in bijection with
Lagrangian Lie subalgebras \( W \subset \fg(\!(x^{-1})\!) \) such that
\begin{enumerate}
\item \( \fg[x] \add W = \fg(\!(x^{-1})\!) \);
\item \( x^{-N}g[\![x^{-1}]\!] \subseteq W\) for some integer \( N > 0 \).
\end{enumerate}
Moreover, \( r_{p_2} = (\varphi \ot \varphi)r_{p_1} \)
for some \( \varphi \in \textnormal{Aut}_{\bC[x] \textnormal{-LieAlg}}(\fg[x]) \)
if and only if \( W_2 = \varphi(W_1) \), where \( W_1 \) and \( W_2 \) 
are the corresponding Lagrangian Lie subalgebras.
\end{theorem}
The second condition on \( W \) means
that \( W \) is an \emph{order} in \( \fg(\!(x^{-1})\!) \).
Let \(\{\alpha_1, \dots, \alpha_n \}\) be the set of
simple roots of \(\fg\) associated to some triangular
decomposition of \(\fg\) and \(\alpha_0\) be
the corresponding maximal root.
Let \(k_i\) be the unique positive integers such that
\(\alpha_0 = \sum_{i=1}^n k_i \alpha_i\) holds.  
It was shown in \cite{Stolin_rational}
that an order \(W\) can be embedded by
an appropriate polynomial gauge transformation
into a so called \emph{maximal order} \(W' \subset \fg(\!(x^{-1})\!) \)
with the property \(W' + \fg[x] = \fg(\!(x^{-1})\!)\)
and that such maximal orders are in one-to-one correspondence with
roots \(\{ -\alpha_0, \alpha_1, \dots, \alpha_n \}\).
The classification of orders
within maximal orders corresponding to roots
\(-\alpha_0\) and \(\alpha_i\) with \(k_i = 1\)
reduces to the classification 
of pairs \( (L, B) \), where
\( L \) is a subalgebra of \( \fg \) such that \( L + P_k = \fg \)
for the parabolic subalgebra \( P_k \) corresponding to the simple root \( \alpha_k \)
and \( B \) is a \( 2 \)-cocycle on \( L \) non-degenerate on the intersection
\( L \cap P_k \).
In particular, when \(\fg = \mathfrak{sl}(n, \bC)\)
all coefficients \(k_i = 1\) and the classification of orders
reduces completely to the classification of such pairs \((L, B)\);
for more details see \cite{Stolin_rational_sln}.

\begin{remark}
In order to obtain the above-mentioned
description of commensurable
twists we made an assumption that
\( F = \bC \). This was done in order
to apply verbatim the statement of
\cite[Theorem A.3]{Raschid_Stepan_classical_twists}.
However, the geometric methods used
in \cite{Raschid_Stepan_classical_twists} to prove the theorem are applicable
directly to formal \( r \)-matrices
over an arbitrary algebraically
closed field \( F \) of characteristic \( 0 \).

Therefore, the description
of commensurable twists in terms of 
orders \( W \subset \fg(\!(x^{-1})\!) \)
and pairs \( (L,B) \)
is true for any algebraically closed field
\( F \) of characteristic \( 0 \).
\end{remark}

\subsection{Commensurable twists in \( \fg \ot A(1,0) \)}
As in the previous case, isomorphism classes of commensurable
twists of \(\delta_2\) are in one-to-one correspondence 
with gauge equivalence classes of \emph{quasi-trigonometric \( r \)-matrices}, i.e.\
formal \( r \)-matrices of the form
\begin{align*}
  r_2(x,y) + p(x,y) = \frac{y \Omega}{x-y} + p(x,y),
\end{align*}
where \( p \in \fg[x] \ot \fg[y] \). 

In case \( F = \bC \) quasi-trigonometric \( r \)-matrices
were classified up to polynomial gauge equivalence in \cite{Pop_Stolin_Lagrangian_quasi_trig}
and independently in \cite{Raschid_Stepan_classical_twists}.
More precisely, any quasi-trigonometric \( r \)-matrix
\( r_2 + p \) is polynomially gauge equivalent to
\( r_2 + t_Q \), where \( t_Q \) is a twist
defined by a BD quadruple \( Q = (\Gamma_1, \Gamma_2, \gamma, t) \)
such that \( \Gamma_1 \) does not contain the minimal root \(-\alpha_0\)
of \( \fg \).
The explicit construction of \( t_Q \) is presented in 
\cite[Section 4]{Raschid_Stepan_classical_twists}.

\subsection{Commensurable twists in \( \fg \ot A(2,0) \)}\label{sec:type2_commensurable_twists}
In \cite{the_4_structure} \(r\)-matrices of the form
\(r_3 + p\) for some polynomial \( p \in \fg[x] \ot \fg[y] \) were called
\emph{quasi-rational \(r\)-matrices}. It was shown there that these are in bijection with Lagrangian Lie subalgebras \(W \) of \( \fg(\!(x^{-1})\!) \times \fg[x]/x^2\fg[x] \) 
such that
\begin{enumerate}
\item \( W \cap P = 0 \);
\item \( W \oplus P = \fg \ot A(2,0) \);
\item \( x^{-N} \fg[\![x^{-1}]\!] \times \{0\} \subseteq W \) for some \( N > 0 \),
\end{enumerate}
where \(P = \{(f,[f])\mid f \in \fg[z]\}\).
It was shown in \cite{Stolin_Zelmanov_Montaner}
that the classification of such Lagrangian Lie subalgebras
up to polynomial gauge equivalence
coincides with the classification of 
orders within maximal orders in \(\fg(\!(x^{-1})\!)\)
corresponding to \(-\alpha_0\) and \(\alpha_i\)
with coefficient \(1\) in the linear decomposition
\(\alpha_0 = \sum_{i=1}^{n} k_i \alpha_i\) of \(\alpha_0\)
into the sum of simple roots. As in \cref{subsec:comm_twists_double_1}, this reduces to the classification 
of pairs \( (L, B) \), where
\( L \) is a subalgebra of \( \fg \) such that \( L + P_k = \fg \)
for the parabolic subalgebra \( P_k \) corresponding to the simple root \( \alpha_k \)
and \( B \) is a \( 2 \)-cocycle on \( L \) non-degenerate on the intersection
\( L \cap P_k \).
In particular, for \(\fg = \mathfrak{sl}(n, F)\),
the classification of commensurable twists of \(\delta_1\)
coincides with the classification of commensurable
twists of \(\delta_3\).

\subsection{Commensurable twists in \( \fg \ot A(\infty) \)}\label{sec:type3_commensurable_twists}
By \cref{lem:commensurable_twists} commensurable twists of \(\delta_0 = 0\)
are precisely the elements \(s \in \fg[x] \ot \fg[y]\)
such that \(\textnormal{CYB}(s) = 0\). Such \(r\)-matrices
are in one-to-one correspondence with finite-dimensional
quasi-Frobenius Lie subalgebras of \(\fg[x]\).
This problem is known to be ''wild'' and hence we do not
expect to see any classification in this case.

\section{Classification of topological twists}\label{sec:classification}
We now give a classification of
topological twists of \( \delta_i,
i \in \{1,2,3\}\).
Its proof is presented in the
following section.

\subsection{Twisting class \( \fg \ot A(0,0) \)}
In this section we assume that \(F = \mathbb{C}\)
in order to present the most explicit classification. 
However, we establish similar results
over arbitrary algebraically closed fields of characteristic
\(0\) in a series of remarks.

\Cref{thm:equivalence} states that isomorphism classes of
topological twists of \( \delta_1 \) 
are in bijection with gauge equivalence classes of formal \( r \) matrices \( r_1 + s \),
where \( s \in (\fg \ot \fg)[\![x,y]\!] \). Following \cite{abedin2021geometrization},
when \( F = \bC \) 
these formal series are actually Taylor series of appropriate two-parametric meromorphic \(r\)-matrices
and are subject to an adjusted form of the well-known trichotomy for one-parametric \(r\)-matrices achieved by Belavin and Drinfeld \cite[Theorem 1.1]{BD_first}. 


\begin{theorem}[Theorem 4.7, \cite{abedin2021geometrization}]\label{thm:raschid_classification}
Let \( r \) be a formal \( r \)-matrix of the
form \( r_1 + s \) for some topological twist
\( s \in (\fg \ot \fg)[\![x,y]\!] \).
Then it is gauge
equivalent to the Taylor series expansion of a classical \( r \)-matrix
\( \rho \colon \bC^2 \to \fg \ot \fg \)
with respect to the second variable in point \( 0 \),
where \( \rho \) is either 
\begin{description}
\item[Elliptic] 
In this case there is a lattice \( \Lambda \subset \bC^2 \) of rank \(2\), such that
for all \(\lambda_1, \lambda_2 \in \Lambda\)
\[\rho(x + \lambda_1, y + \lambda_2) = \rho(x,y);\]

\item[Trigonometric] 
\begin{equation}\label{eq:trigonometric_rmat}
    \rho(x,y) = \frac{1}{\exp(x-y) - 1} \sum_{k=0}^{|\sigma|-1} \exp\left( \frac{k(x-y)}{|\sigma|} \right)
\Omega_k + t \left(\exp\left( \frac{x}{|\sigma|} \right), \exp\left( \frac{y}{|\sigma|} \right) \right),
\end{equation}
for some finite-order automorphism \( \sigma \in \textnormal{Aut}_{\bC\textnormal{-LieAlg}}(\fg) \)
and an element \( t \in \mathfrak{L}(\fg, \sigma) \ot \mathfrak{L}(\fg, \sigma) \). 
Here \( \mathfrak{L}(\fg, \sigma) \) stands for the loop algebra of \( \fg \)
twisted by \( \sigma \); \\

\item[Rational] 
\begin{equation}\label{eq:rational_rmat}
    \rho(x,y) = r_p(x,y) \coloneqq  \frac{\Omega}{x-y} + p(x,y),
\end{equation}
where \( p \in (\fg \ot \fg)[x,y] \) is skew-symmetric polynomial.
\end{description}
\end{theorem}

\begin{remark}
We will see in \cref{sec:proof} 
that it is possible to assign
an irreducible cubic plane curve \(X\) to a formal \(r\)-matrix \(r = r_1 + s\)
over any algebraically
closed field of characteristic \( 0 \).
It is then well-known that \(X\) is either an elliptic curve or has a unique singular point, which is either nodal or cuspidal. The type of \(X\) is thereby an invariant of the equivalence class of \(r\). Therefore, this procedure splits (equivalence classes of) formal \(r\)-matrices of the form \(r_1 + s\) into three classes. It is shown in \cite{abedin2021geometrization} that for \(F = \mathbb{C}\), \(r\) is elliptic (resp.\ trigonometric, resp.\ rational) if \(X\) is elliptic (resp.\ nodal, resp.\ cuspidal). Therefore, this categorization of \(r\) by its associated curve is the generalization of \cref{thm:raschid_classification} to arbitrary algebraically closed fields of characteristic 0. In particular, we also call a formal \(r\)-matrix of the form \(r_1 + s\) elliptic (resp.\ trigonometric, resp.\ rational) if \(X\) is elliptic (resp.\ nodal, resp.\ cuspidal) even if \(F\) is not the field of complex numbers.
\end{remark}

By \cite[Proposition 4.10]{abedin2021geometrization}
two classical \( r \)-matrices of one of the forms above
are holomorphically gauge equivalent if and only if the corresponding 
Taylor series expansions of these \( r \)-matrices
are formally gauge equivalent.
This reduces the classification of topological twists of \( \delta_1 \)
up to formal isomorphisms to the classification of classical
\( r \)-matrices of the forms given in \cref{thm:raschid_classification}
up to holomorphic equivalence.
The latter classification, in its turn, reduces to the works
by Belavin and Drinfeld \cite{BD_first,BD_second} in the elliptic and trigonometric
cases and to works by Stolin \cite{Stolin_rational_sln,Stolin_rational} in the rational case.
\smallskip 

\paragraph{\textbf{Elliptic}}
It is known from \cite{BD_first} that elliptic \(r\)-matrices
exist only for \(\fg = \mathfrak{sl}(n, \bC)\).
Moreover, their explicit construction involves a choice
of a primitive \(n\)-th root of unity.
Therefore, all classical elliptic \(r\)-matrices are parametrized
by triples \((\lambda_1, \lambda_2, d)\), where
\(\lambda_1, \lambda_2 \in \bC\) span a two-dimensional lattice
and \( n > d > 0\) is an integer coprime with \(n\) that describes
our choice a primitive root of unity.

Extending the notion of equivalence by 
allowing scaling and coordinate transformations 
we can without loss of generality assume \(\lambda_1 = 1\) and
\(\Im(\lambda_2) > 0\).
Furthermore, by \cite[Lemma 3.3]{abedin2021geometrization},
two classical elliptic \(r\)-matrices corresponding to 
\((1, \lambda, d)\) and \((1, \lambda', d)\)
are equivalent if an only if the elliptic curves
\(\bC / \langle 1, \lambda \rangle\) and
\(\bC / \langle 1, \lambda' \rangle\) are isomorphic.
This gives the following statement.

\begin{theorem}
\label{thm:elliptic_W}
Let \(\fg = \mathfrak{sl}(n, \bC)\).
Topological twists of \(\delta_1\) of elliptic type
are parametrized by doubles
\((\lambda, d)\),
where \(\lambda \in \mathbb{H}/\mathrm{SL}(2, \bZ)\)
and \(0 < d < n\) is an integer coprime to \(n\).
\end{theorem}

\begin{remark}
It can be shown that two elliptic \(r\)-matrices corresponding to
triples
\((\lambda, d_1)\) and \((\lambda, d_2)\) are
(gauge) equivalent if and only if \(d_2 = n - d_1\).
Combining this result with \cref{thm:elliptic_W} 
we see that equivalence classes of topological twists of \( \delta_1 \)
of elliptic type are in bijection with
\[
\mathbb{H}/\mathrm{SL}(2, \bZ) \times \{ (n, d) \in \mathbb{Z}_+ \mid \gcd{(n,d)} = 1 \ \text{ and } \ 0 < d < n/2 \}.
\]
\end{remark}

\begin{remark}
For a general algebraically closed field \(F\) of characteristic 0, the methods from \cite{abedin2021geometrization} prove that elliptic formal \(r\)-matrices are parametrized by pairs \((X,\mathcal{A})\) consisting of an elliptic curve \(X\) over \(F\) and an acyclic locally free sheaf \(\mathcal{A}\) of Lie algebras on \(X\) with fiber \(\fg\) at every point of \(X\). 
\end{remark}

\paragraph{\textbf{Trigonometric}}
It was shown in \cite[Theorem 3.4]{Raschid_Stepan_classical_twists}
that any classical \(r\)-matrix of the form \cref{eq:trigonometric_rmat}
is globally holomorphically equivalent to a classical
\(r\)-matrix \(X(u-v)\) depending on the difference of its
parameters and such that the set of poles of \(X(z)\) 
is \(2 \pi i \mathbb{Z}\).
The last fact means that \(X(z)\) is a trigonometric solution
in the sense of \cite{BD_first}.
The well-known classification \cite[Theorem 6.1]{BD_first}
of trigonometric solutions
tells us that \(X(z)\) is
holomorphically equivalent to another trigonometric solution
\(X^\sigma_Q(z)\), described by a finite-order
automorphism \(\sigma \) of \(\fg\) and a Belavin-Drinfeld (BD)
quadruple \(Q = (\Gamma_1, \Gamma_2, \gamma, t)\).

Conversely, any trigonometric solution \(X^\sigma_Q(z)\)
is automatically of the form described in \cref{thm:raschid_classification}
and hence its Taylor series expansion at \( y = 0 \) gives
a normalized formal \( r \)-matrix, which in its turn produces a topological twist.

Therefore, topological twists of \( \delta_1 \) of trigonometric type
are fully described by finite-order automorphisms of \(\fg\) and BD 
quadruples \(Q\); for more details see \cite{Raschid_Stepan_classical_twists,BD_first}.
\\

\begin{remark}
For an arbitrary algebraically closed field \(F\) of characteristic 0, the formula \cref{eq:trigonometric_rmat} is well-defined and a formal \(r\)-matrix is trigonometric if and only if it is equivalent to a series of this form. The classification of trigonometric \(r\)-matrices by Belavin-Drinfeld quadruples relies completely on the structure theory of twisted loop algebras and thus remains valid if \(\mathbb{C}\) is replaced by \(F\).
\end{remark}

\paragraph{\textbf{Rational}}
Classification of twists
in this case coincides
with the 
classification
of commensurable twists
inside
\( \fg \ot A(0,0) \).
The latter is presented in
\cref{subsec:comm_twists_double_1}.

\begin{remark}
For an arbitrary algebraically closed field \(F\) of characteristic 0, the formula \cref{eq:rational_rmat} is well-defined and a formal \(r\)-matrix is rational if and only if it is equivalent to a series of this form. As already mentioned in \cref{subsec:comm_twists_double_1}, the results from \cite{Stolin_rational_sln,Stolin_rational} also apply to this general setting.
\end{remark}

\subsection{Twisting classes \( \fg \ot A(1,0) \) and \( \fg \ot A(2,0)\)}
Let \(i \in \{2,3\}\). 
Lagrangian Lie subalgebras \( W \subseteq \fg(\!(x)\!) \times \fg[x]/x^{i-1}\fg[x]\) complementary 
to the image of \( \fg[\![x]\!] \) in \(\fg(\!(x)\!) \times \fg[x]/x^{i-1}\fg[x]\) are in one-to-one correspondence
with formal \( r \)-matrices of the form \(r_i + s \) for some \(s \in (\fg \ot \fg)[\![x,y]\!]\), defined
in \cref{sec:realization_polynomials}. 

\begin{theorem}%
\label{thm:other_twisting_classes}

Up to formal isomorphism, Lagrangian
Lie subalgebras of \( \fg \ot A(1,0)\) and
\( \fg \ot A(2,0)\) are commensurable.
This implies that all topological twists of \( \delta_2 \) and \( \delta_3 \), up to formal isomorphism, are commensurable and hence their
classification coincides with the one presented in \cref{sec:type2_commensurable_twists} and  \cref{sec:type3_commensurable_twists} respectively. In particular, all \(r\)-matrices of the form
\(r_2 + s\) and \(r_3 + s\), where \(s \in (\fg \ot \fg)[\![x,y]\!]\), are formally gauge equivalent to quasi-trigonometric and quasi-rational \(r\)-matrices respectively.
\end{theorem}

\section{Algebro-geometric proof of the classification of topological twists}\label{sec:proof}
Finally, we present an algebro-geometric proof
of the classification stated in
\cref{thm:other_twisting_classes}.

\subsection{Preliminary results} 
\label{lem:preleminary_facts_for_geometrization}
Let \( i \in \{1,2,3 \} \), consider the canonical injection
\begin{equation}
    \begin{aligned}
    \iota \colon \fg[\![x]\!] &\longrightarrow \fg(\!(x)\!) \times \fg[x]/x^{i-1}\fg[x]\\
    f &\longmapsto (f,[f])
    \end{aligned}
\end{equation}
and let \(
W \subseteq \fg \otimes A(i-1,0) = \fg(\!(x)\!) \oplus \fg[x]/x^{i-1}\fg[x]
\)
be a Lagrangian subalgebra satisfying
\begin{equation}
    \iota(\fg[\![x]\!]) \dotplus W = \fg(\!(x)\!) \times \fg[x]/x^{i-1}\fg[x].    
\end{equation}
Furthermore, let \(W_+ \) and \( W_- \) be the projections
of \( W \)
onto \(\fg(\!(x)\!)\) and \( \fg[x]/x^{i-1}\fg[x]\) respectively.
The following results are true:
\begin{enumerate}
    \item \(W_\pm^\bot \subseteq W_\pm\) with respect to the bilinear form \(\mathcal{K}_i\) defined in \cref{rem:B_i};
    
    \item \(\fg(\!(x)\!) = \fg[\![x]\!] + W_+\) and \(\dim(\fg[\![x]\!]\cap W_+) < \infty\);
    
    \item \(W_+/W_+^\bot \times W_-/W_-^\bot = (\iota(\fg[\![x]\!]) \cap (W_+\times W_-)) \dotplus W/(W_+^\bot \times W_-^\bot)\) is a Manin triple satisfying \(\dim(W_+/W_+^\bot) = \dim(W_-/W_-^\bot) < \infty\);
    
    \item When \(i \in \{2,3\}\)
    we have \(\fg[\![x]\!]\cap W_+ \neq \{0\}\).
\end{enumerate}
\begin{proof}
Part (1) follows from the inclusions
\begin{equation}\label{eq:Wbot}
    W_+^\bot \times W_-^\bot = (W_+\times W_-)^\bot \subseteq W^\bot = W\subseteq W_+\times W_-.
\end{equation}
The equality 
\(\fg(\!(x)\!) = \fg[\![x]\!] + W_+\) 
follows easily from
\(\iota(\fg[\![x]\!]) + W = \fg(\!(x)\!) \times \fg[x]/x^{i-1}\fg[x]\).
For the second
part of (2) we have
\(0 = (\fg[\![x]\!] + W_+)^\bot = x^{i-1}\fg[\![x]\!] \cap W_+^\bot\)
which implies that \(\fg[\![x]\!] \cap W_+^\bot\) can be embedded into \(\fg[\![x]\!]/x^{i-1}\fg[\![x]\!]\) and is therefore finite-dimensional.
Consequently, the dimension of
\(\fg[\![x]\!]\cap W_+\) is finite if
the quotient
\((\fg[\![x]\!]\cap W_+) / (\fg[\![x]\!]\cap W_+^{\perp}) \)
is finite-dimensional. The latter space can be identified with a subspace of \(W_+/W_+^\bot\).
Therefore, the second part
of (2) follows from (3).

For (3) observe that the kernel \(K\) of the projection \(W \to W_+\) contains \(\{0\} \times W_-^\bot\) by virtue of \cref{eq:Wbot}. On the other hand, any element of \(K\) is of the form
\((0, a)\) for some \(a \in W_-\), so for all \( (w_+,w_-)\in W \)
\begin{equation}
    0 = B_i((0,a),(w_+,w_-)) = -\mathcal{K}_i(a,w_-)
\end{equation}
holds, implying \( a \in W_-^{\perp} \)
and hence
\( K = \{ 0 \} \times W_-^{\perp}\).
This provides us
with an isomorphism \(W/(W_+^\bot\times W_-^\bot) \to W_+/W_+^\bot\) and, similarly,
\(W/(W_+^\bot\times W_-^\bot) \to W_-/W_-^\bot\).
Composing them we obtain
an isomorphism
\(W_+/W_+^\bot \to W_-/W_-^\bot\).
In particular, \(\dim(W_+/W_+^\bot) = \dim(W_-/W_-^\bot) <\nolinebreak \infty\). 
Since \(W \subseteq W_+ \times W_-\), the identity \(\iota(\fg[\![x]\!]) \dotplus W = \fg(\!(x)\!) \times  \fg[x]/x^{i-1}\fg[x]\) is equivalent to
\begin{equation}\label{eq:W+W-W}
    W_+ \times W_- = (\iota(\fg[\![x]\!]) \cap (W_+\times W_-)) \dotplus W.    
\end{equation}
Quoiting out \(W_+^\bot \times W_-^\bot\) concludes (3).

Let us turn now to (4). Assume that \(i \in \{2,3\}\) and \(\fg[\![x]\!] \cap W_+ = \{0\}\). Then
\begin{equation}
    \iota(\fg[\![x]\!]) \cap (W_+\times W_-) = \{0\},    
\end{equation}
and \cref{eq:W+W-W} imply
\(W = W_+ \times W_- = W_+^\bot \times W_-^\bot\).
Since \(\fg(\!(x)\!)\times \fg[x]/x^{i-1}\fg[x] = \iota(\fg[\![x]\!]) \dotplus (W_+ \times W_-)\) and \(\fg[\![x]\!] \cap W_+ = \{0\}\) we obtain \(W_- =  \fg[x]/x^{i-1}\fg[x]\), which contradicts \(W_-^\bot = W_-\).
\end{proof}

\subsection{Algebro-geometric notation and convention}

\begin{itemize}
    \item In the following, an algebraic variety is an integral scheme of finite type over the algebraically closed field \(F\) of characteristic 0. In this setting, a projective variety \(X\) is an algebraic variety isomorphic to the projective spectrum \(\textnormal{Proj}(R)\) of a finitely-generated graded \(F\)-algebra \(R\) without zero divisors; see e.g.\ \cite[Section II.2]{hartshorne} for the definition of the projective spectrum. In particular, a projective curve is such a projective variety of dimension one.
    
    \item For an algebraic variety \(X\), the sheaf of rings of regular functions of \(X\) is denoted by \(\mathcal{O}_X\) and for any point \(p \in X\), \(\mathfrak{m}_p\) is the maximal ideal of the local ring \(\mathcal{O}_{X,p}\).

    \item For an \(\mathcal{O}_X\)-module \(\mathcal{F}\), we write \(\mathcal{F}_p\) (resp.\ \(\mathcal{F}|_p = \mathcal{F}_p/\mathfrak{m}_p\mathcal{F}_p\)) for the stalk (resp.\ fiber) of \(\mathcal{F}\) in some point \(p \in X\). Furthermore, \(\Gamma(U,\mathcal{F})\) denotes the space of sections of \(\mathcal{F}\) on an open subset \(U \subseteq X\),  \(\textnormal{H}^k(\mathcal{F})\) denotes the \(k\)-th global cohomology class of \(\mathcal{F}\), and \(\textnormal{h}^k(\mathcal{F}) \coloneqq \dim(\textnormal{H}^k(\mathcal{F}))\). In particular, \(\textnormal{H}^0(\mathcal{F}) = \Gamma(X,\mathcal{F})\).
    
    \item We call \(\mathcal{F}\) a sheaf of Lie algebras on \(X\) if it is equipped with a bilinear morphism \(\mathcal{F} \times \mathcal{F} \to \mathcal{F}\) which equips \(\mathcal{F}_p\) with the structure of a Lie algebra for all \( p\in X\). If \(\mathcal{F}\) is also locally free of finite rank, we call the unique bilinear morphism \(\mathcal{F}\times \mathcal{F} \to \mathcal{O}_X\) whose stalk at any point \(p \in X\) corresponds to the Killing form of the free \(\mathcal{O}_{X,p}\)-Lie algebra \(\mathcal{F}_p\), Killing form of \(\mathcal{F}\) (see e.g.\ \cite[Definition 2.4 \& Lemma 2.5]{abedin2021geometrization}).
\end{itemize}

\subsection{Geometrization of lattices.}\label{subsec:geometrization_of_lattices}
The methods to prove \cref{thm:raschid_classification} and \cref{thm:other_twisting_classes} are based on the following geometrization scheme of \(\fg\)-lattices from \cite[Section 2.3]{abedin2021geometrization}.
Recall that a \(\fg\)-lattice \(L \subseteq \fg(\!(x)\!)\) is a subalgebra satisfying 
\begin{equation}\label{eq:lattice}
    \dim(L \cap \fg[\![x]\!]) < \infty \textnormal{ and }\dim(\fg(\!(x)\!)/(\fg[\![x]\!]+L)) < \infty.
\end{equation}
Then for any unital subalgebra \(O \subseteq M_L \coloneqq \{f \in F(\!(x)\!)\mid fL \subseteq L\}\) of finite codimension, the graded \(F\)-algebra
\begin{equation}
    \textnormal{gr}(O) \coloneqq \bigoplus_{j = 0}^\infty t^j\left(O\cap x^{-j}F[\![x]\!]\right) \subseteq O[t]
\end{equation}
defines a projective curve \(X \coloneqq \textnormal{Proj}(\textnormal{gr}(O))\) satisfying
\begin{equation}
    \textnormal{h}^1(\mathcal{O}_X) =\dim(F(\!(x)\!)/(F[\![x]\!]+O)).
\end{equation}
The smooth point \(p = (t) \in X\) satisfies \(\textnormal{D}_+(t) = X\setminus\{p\}\) and is equipped with an isomorphism \(c \colon \widehat{\mathcal{O}}_{X,p} \to F[\![x]\!]\) such that the extension \( c \colon (\widehat{\mathcal{O}}_{X,p}\setminus \{0\})^{-1}\widehat{\mathcal{O}}_{X,p} \to F(\!(x)\!)\) has the property
\begin{equation}
    c(\Gamma(X\setminus \{p\},\mathcal{O}_X)) = O \subseteq F(\!(x)\!) .
\end{equation}
Now the graded \(\textnormal{gr}(O)\)-Lie algebra
\begin{equation}
    \textnormal{gr}(L) \coloneqq \bigoplus_{j \in \bZ} t^j(L \cap x^{-j}\fg[\![x]\!]) \subseteq L[t,t^{-1}]
\end{equation}
defines a coherent sheaf of Lie algebras \(\mathcal{L}\) on \(X\) satisfying 
\begin{equation}\label{eq:cohomology}
    \textnormal{h}^0(\mathcal{L}) = \dim(L \cap \fg[\![x]\!])  \textnormal{ and }\textnormal{h}^1(\mathcal{L}) =\dim(\fg(\!(x)\!)/(\fg[\![x]\!]+L)).
\end{equation} 
As \(\mathcal{O}_X\)-module, \(\mathcal{L}\) is simply the the sheaf associated to a graded module on a projective scheme in e.g.\ \cite[Section II.5]{hartshorne}.
There is a natural isomorphism \(\zeta \colon \widehat{\mathcal{L}}_p \to \fg[\![x]\!]\) such that the extension \((\widehat{\mathcal{O}}_{X,p}\setminus \{0\})^{-1}\widehat{\mathcal{L}}_{p} \to \fg(\!(x)\!)\) has the property \(\zeta(\Gamma(X\setminus \{p\},\mathcal{L})) = L\). 
Let us write 
\begin{equation}\label{eq:geometric_datum}
\mathbb{G}(O,L) \coloneqq ((X,\mathcal{L}),(p,c,\zeta))
\end{equation}
for the geometric datum constructed above.

\subsection{Some results on multipliers.}
\label{lemm:geometric_genus}
Take an \( i \in \{1,2,3 \} \)
and consider a Lagrangian
subalgebra
\begin{equation}
W \subseteq \fg \otimes A(i-1,0) = \fg(\!(x)\!) \times \fg[x]/x^{i-1}\fg[x]    
\end{equation}
complemetary
to \( \iota(\fg[\![x]\!]) \)
 and let \(W_+ \subseteq \fg(\!(x)\!)\) be the projection of \(W\)
 on \( \fg(\!(x)\!) \).

\begin{enumerate}
    \item The integral closure \(N\) of \(M \coloneqq M_{W_+} \coloneqq  \{f\in F(\!(x)\!)\mid f W_+ \subseteq W_+\}\) satisfies
    \begin{equation}
        g \coloneqq \dim(F(\!(x)\!)/(F[\![x]\!]+ N)) \in \{0,1\};    
    \end{equation}
    
    \item When \(g = 1 \) we have \( M = N\) and \(i = 1\);
    
    \item If \( g = 0\) and \(i = 1 \) then \( F[s',s's] \subseteq M\) for some \(s \in x^{-1}F[\![x]\!]^\times\). Here \( s' \) stands
    for the formal derivative
    of \( s \);
    
    \item For \(i \in \{2,3\}\) we have \( N = M = F[x^{-1}]\).
    \end{enumerate}
\begin{proof} The statements from \Cref{subsec:geometrization_of_lattices} imply that \(M\) has Krull dimension one and \(W_+\) is finitely-generated over \(M\). The canonical inclusion \(M \to N\) is finite (see e.g.\ \cite[Chapter I, Theorem 3.9A]{hartshorne}) and \((M\setminus\{0\})^{-1}M = (M\setminus\{0\})^{-1}N\), so \(N\) is a finitely-generated torsion-free \(M\)-module of rank one. 
Therefore, \(N/M\) is a torsion module over \(M\) and \(N/M\) is finite-dimensional since \(M\) has Krull dimension one. The subalgebra \(L \coloneqq NW_+ \subseteq \fg(\!(x)\!)\) is now also a finitely-generated \(M\)-module of the same rank as \(W_+\). Therefore, using the same argument \(L/W_+\) is finite-dimensional.
By (2) of \cref{lem:preleminary_facts_for_geometrization}
the subalgebra \( W_+ \)
is a \(\fg\)-lattice
and hence so is
\( L \coloneqq NW_+\subseteq \fg(\!(x)\!) \).
Let \(\mathbb{G}(N,L) \coloneqq ((X,\mathcal{L}),(p,c,\zeta))\) be the geometrization described in \cref{subsec:geometrization_of_lattices}.
Combining
\cref{eq:cohomology} with
\(\fg(\!(x)\!) = \fg[\![x]\!] + W_+ \subseteq \fg[\![z]\!] + L\)
we get
\(\textnormal{H}^1(\mathcal{L}) = 0\).
%

The Killing form \(K \colon \mathcal{L} \times \mathcal{L} \to \mathcal{O}_X\) induces a morphism \(K^\textnormal{a} \colon \mathcal{L} \to \mathcal{L}^*\) (where \(\mathcal{L}^*\) is the dual as \(\mathcal{O}_X\)-module).
The fiber \(K^\textnormal{a}|_p\) is an isomorphism
because \(K|_p\)
can be identified with the
Killing form of
\(\mathcal{L}|_p \cong \fg\).
Therefore,
\(\textnormal{Ker}(K^\textnormal{a})\) and \(\textnormal{Cok}(K^\textnormal{a})\) are torsion sheaves.
%
In particular, \(\textnormal{H}^1(\textnormal{Cok}(K^\textnormal{a})) = 0\) and \(\textnormal{Ker}(K^\textnormal{a})\) vanishes, since it is a torsion subsheaf of the locally free sheaf \(\mathcal{L}\). This results in the short exact sequence
\begin{align}\label{eq:tildeKexactseq}
    0 \longrightarrow \mathcal{L} \stackrel{K^\textnormal{a}}{\longrightarrow} \mathcal{L}^* \longrightarrow \textnormal{Cok}(K^\textnormal{a})\longrightarrow0.
\end{align}
The associated long exact sequence in cohomology reads
\begin{equation}
    0 \longrightarrow \textnormal{H}^0(\mathcal{L}) {\longrightarrow} \textnormal{H}^0(\mathcal{L}^*) \longrightarrow \textnormal{H}^0(\textnormal{Cok}(K^\textnormal{a}))\longrightarrow \textnormal{H}^1(\mathcal{L}) {\longrightarrow} \textnormal{H}^1(\mathcal{L}^*) \longrightarrow \textnormal{H}^1(\textnormal{Cok}(K^\textnormal{a})) \longrightarrow 0.
\end{equation}
The identities \(\textnormal{H}^1(\mathcal{L}) = 0 = \textnormal{H}^1(\textnormal{Cok}(K^\textnormal{a}))\) imply that \(\textnormal{H}^1(\mathcal{L}^*) = 0\). 

The Riemann-Roch theorem for \(\mathcal{L}\) and \(\mathcal{L}^*\) (e.g.\ in the version of \cite[Chapter 7, Exercise 3.3]{liu}) combined with the fact that \(\textnormal{h}^1(\mathcal{O}_X) = g\) implies 
\begin{align*}
    &0 \le \textnormal{h}^0(\mathcal{L})-\textnormal{h}^1(\mathcal{L}) = \textnormal{deg}(\textnormal{det}(\mathcal{L})) + (1-g)\textnormal{rank}(\mathcal{L}),\\
    &0 \le \textnormal{h}^0(\mathcal{L}^*)-\textnormal{h}^1(\mathcal{L}^*) = -\textnormal{deg}(\textnormal{det}(\mathcal{L})) + (1-g)\textnormal{rank}(\mathcal{L}),
\end{align*}
where we used that \(\textnormal{det}(\mathcal{L}^*) = \textnormal{det}(\mathcal{L})^*\) implies \(\textnormal{deg}(\textnormal{det}(\mathcal{L}^*)) = -\textnormal{deg}(\textnormal{det}(\mathcal{L}))\). We conclude \( g \le 1\).

Assume \(g = 1\), then \(X\) is an elliptic curve. Let \(\Omega^1_{X}\) be the sheaf of regular 1-forms on \(X\).
We have \(\textnormal{H}^1(\Omega^1_{X}) = F \eta\) for some global 1-form \(\eta\) on \(X\) and this choice defines an isomorphism \(\Omega^1_{X} \cong \mathcal{O}_{X}\).
Serre duality (see e.g.\ \cite[Chapter III, Corollary 7.7]{hartshorne}) provides \(0 = \textnormal{h}^1(\mathcal{L}^*) = \textnormal{h}^0(\mathcal{L})\).
In particular, by \cref{eq:cohomology}, \(W_+ \cap \fg[\![x]\!] \subseteq L \cap \fg[\![x]\!] = \{0\}\), so \cref{lem:preleminary_facts_for_geometrization}.(4) implies \(i = 1\).
Moreover, \(W_+ \dotplus \fg[\![x]\!] = \fg(\!(x)\!) =  L \dotplus \fg[\![x]\!]\)
and \(W_+ \subseteq L\) imply \(L = W_+\), so \(M = N\).

Assume \(g= 0\), i.e.\ \(F(\!(x)\!) = F[\![x]\!] + N\). Since \(N \cap F[\![x]\!] = \textnormal{H}^0(\mathcal{O}_X) = F\), we can see that \(N = F[s]\) for the unique \(s \in (x^{-1}+xF[\![x]\!]) \cap N \neq \{0\}\). 
The Killing form \(\kf\) of \(\fg(\!(x)\!)\) as \(F(\!(x)\!)\)-Lie algebra restricts to the Killing form \(L \times L \to N\) of the locally free \(N\)-Lie algebra \(L\). Let \(N^\bot\) be the orthogonal complement of \(N\) with respect to the bilinear form \(R \colon F(\!(x)\!) \times F(\!(x)\!) \to F\) defined by
\begin{equation}
    (f,g) \longmapsto \textnormal{res}_{0} \{ x^{1-i} f g \} = \textnormal{coeff}_{i-2}\{fg\}.
\end{equation}
For all \(f \in N^\bot\) and \(a,b\in W_+\) we have
\begin{equation}
    \mathcal{K}_i(fa,b) = \textnormal{res}_{x = 0} \{ x^{1-i} f(x) {\kf(a(x),b(x))} \} = R(\underbrace{f(x)}_{\in N^\bot}, \underbrace{\kf(a(x),b(x))}_{\in N}) = 0.
\end{equation}
In particular, \(fa \in W_+^\bot \subseteq W_+\). Therefore, \(N^\bot \subseteq \{f \in F(\!(x)\!)\mid f W_+ \subseteq W_+\} = M \subseteq N\). From 
\begin{equation}
    R(x^{i-1}s',s^k) = \textnormal{res}_{x = 0} \{ s'(x)s^k(x) \} = \frac{1}{k+1}\textnormal{res}_{x = 0}\left(s^{k+1}(x)\right)' = 0
\end{equation}
for all \(k \in \mathbb{Z}_{\ge 0}\), we can deduce that
\(x^{i-1}s' \in N^\bot \subseteq M\).

\textbf{Case \(i = 1\).} Since \(R\) is associative 
and \( s' \in N^\bot \) we have 
the inclusion \(s'N \subseteq N^\bot\). Furthermore, since
\(s' \in N^\bot \subseteq N = F[s]\), we obtain 
\(F[s',s's] \subseteq F + s'N\). 
Combining these results we get 
\begin{equation}
    F[s',s's] \subseteq F + s'N \subseteq F+N^\bot \subseteq \{f \in F(\!(z)\!)\mid fW_+ \subseteq W_+\} = M.
\end{equation}

\textbf{Case \(i = 2\).} In this case we have \(xs' \in N^\bot \subseteq F[s]\).
Since we assumed that \( s \) has no
constant part the following equality must hold \(xs' = -s \). By coefficient
comparison we obtain \(s = x^{-1}\).
Therefore, \(N = F[x^{-1}]\) and
\(x^{-1}F[x^{-1}] = N^\bot \subseteq M\)
implying \(M = N = F[x^{-1}]\). 

\textbf{Case \(i = 3\).} The fact that
\(x^2s' \in N^\bot \cap F[\![x]\!] \subseteq N \cap F[\![x]\!] = F\) 
implies that \(s' = -x^{-2}\). Consequently, \(N = F[x^{-1}] = N^\bot \subseteq M\) concluding the proof.
\end{proof}

\subsection{Some results on sheaves of Lie algebras.}
\label{prop:weakly_g_locally_free_modules}
\begin{enumerate}
    \item Let \(\mathcal{A}\) be a locally free sheaf of Lie algebras  on an algebraic variety \(X\) and \(\mathcal{A}|_p\) be finite-dimensional and semi-simple for some closed point \(p \in X\). Then \(\mathcal{A}|_q \cong \mathcal{A}|_p\) for all closed points \(q\) in some open neighbourhood of \(p\);
    
    \item Let \(\mathcal{A}\) be a sheaf of Lie algebras on an irreducible cubic plane curve \(X\) such that: 
    \begin{enumerate}
        \item \(\textnormal{H}^0(\mathcal{A}) = 0 = \textnormal{H}^1(\mathcal{A})\);
        
        \item \(\mathcal{A}|_p\cong \fg\) for some smooth point \(p \in X\);
        
        \item There is a bilinear form
        \(K \colon \mathcal{A} \times \mathcal{A} \to \mathcal{O}_X\) which extends the Killing form of \(\mathcal{A}|_C\), 
        where \(C \subseteq X\) is the set of smooth points on \(X\).
    \end{enumerate}
    Then \(\mathcal{A}|_p \cong \fg\) for all smooth closed points \(p \in X\);
   
    \item Let \(A\) be a \(F[x]\)-Lie algebra such that \(A/(x-a)A \cong \fg\) for all \(a \in F\). Then \(A \cong \fg[x]\) as \(F[x]\)-Lie algebras.
\end{enumerate}

\begin{proof} 
We can assume that \(X\) is an affine variety and that \(\mathcal{A}\) is free of rank \(n\). Let \(X(F) \subseteq X\) be the affine algebraic set of closed points (which coincides with the set of \(F\)-rational points, since \(F\) is algebraically closed) and \(B\subseteq \textnormal{Hom}_F(F^n\otimes F^n,F^n)\) be the affine algebraic subset of all possible Lie brackets on \(F^n\).
Then \(\Gamma(X,\mathcal{A})\) can be identified with the \(\Gamma(X,\mathcal{O}_X)\)-module of all regular maps \(X(F) \to F^n\) equipped with the Lie bracket \(\mu_{\mathcal{A}}\) defined by a regular map \(\theta \colon X(F) \to B\) via \(\mu_{\mathcal{A}}(a \otimes b)(q) = \theta(q)(a(q) \otimes b(q))\)
for all \(a,b\colon X(F) \to F^n\) regular and \(q \in X(F)\). 
The group \(G = \textnormal{GL}(n,F)\) acts on \(B\) via
\begin{align}\label{eq:actionofG}
    &(L \cdot \vartheta)(v \otimes w) = L^{-1}\vartheta(Lv \otimes Lw) &\forall L \in G, \vartheta \in M, v,w \in F^n.
\end{align}
The orbit \(G\cdot\theta(p)\) coincides with the set of Lie brackets on \(F^n\) determining Lie algebra structures isomorphic to \(\mathcal{A}|_p = (F^n,\theta(p))\). Combining \cite[Theorem 7.2]{nijenhuis_richardson} with Whitehead's Lemma and the fact that \(\mathcal{A}|_p\) is semi-simple, 
we see that \(G\cdot \theta(p) \subseteq B\) is open and as a consequence \(U \coloneqq \theta^{-1}(G\cdot \theta(p)) \subseteq X(F)\) is an open neighbourhood of \(p\). The observation \(\mathcal{A}|_q \cong (F^n,\theta(p)) = \mathcal{A}|_p\) for all \(q \in U\) concludes the proof of Part (1).

Let us turn to Part (2).
It is well-known that the dualizing sheaf of any irreducible cubic plane curve is trivial. Therefore, Serre duality implies that \(\textnormal{H}^0(\mathcal{A}^*) = \textnormal{H}^1(\mathcal{A}) = 0\).
Similar to the proof of Lemma \ref{lemm:geometric_genus}, the condition (b) implies that there is an exact sequence
\begin{align}\label{eq:tildeKexactseq}
    0 \longrightarrow \textnormal{H}^0(\mathcal{A}) {\longrightarrow} \textnormal{H}^0(\mathcal{A}^*) \longrightarrow \textnormal{H}^0(\textnormal{Cok}(K^\textnormal{a}))\longrightarrow \textnormal{H}^1(\mathcal{A}) {\longrightarrow} \textnormal{H}^1(\mathcal{A}^*) \longrightarrow \textnormal{H}^1(\textnormal{Cok}(K^\textnormal{a})) \longrightarrow 0,
\end{align}
where \(K^\textnormal{a}\colon\mathcal{A} \to \mathcal{A}^*\) is the morphism induced by \(K\) and \(\textnormal{Cok}(K^\textnormal{a})\) is torsion. Consequently, Serre duality and condition (a) gives \(0 = \textnormal{H}^1(\mathcal{A}) =  \textnormal{H}^0(\mathcal{A}^*) \), so \(\textnormal{H}^0(\textnormal{Cok}(K^\textnormal{a})) = 0\). Since \(\textnormal{Cok}(K^\textnormal{a})\) is a torsion sheaf, we see that \(K^\textnormal{a} \colon \mathcal{A} \to \mathcal{A}^*\) is an isomorphism. Therefore, the fiber of \(K\) at any smooth point \(q \in X\), which coincides with the Killing form of \(\mathcal{A}|_q\) by assumption, is non-degenerate. 
By virtue of Cartan's criterion
\(\mathcal{A}|_q\) is semi-simple.
Using (1) and the fact that the set of smooth points \(C\) is connected, we are left with \(\mathcal{A}|_q \cong \mathcal{A}|_p \cong \fg\) for all closed points \(q \in C\).

Part (3) is a simple reformulation of \cite[Theorem 4.12.(2)]{abedin2021geometrization}. 
\end{proof}

\subsection{Twisting class \( \fg \ot A(0,0) \).} 
Let
\(
W = W_+ \subseteq \fg \otimes A(0,0) = \fg(\!(x)\!)   
\)
be a Lagrangian subalgebra complementary
to \( \fg[\![x]\!] \).
The proof of \cref{thm:raschid_classification} given in \cite{abedin2021geometrization} proceeds in the following steps:

\begin{enumerate}
    \item According to (2) and (3) of \cref{lemm:geometric_genus}
    there is a unital subalgebra \(O \subseteq M_{W}\) such that 
    \begin{equation}
        \textnormal{h}^1(\mathcal{O}_X) = \dim(F(\!(x)\!)/(F[\![x]\!]+ O)) = 1,
    \end{equation}
    where \(X\) is taken from the geometrization \(\mathbb{G}(O,W) = ((X,\mathcal{A}),(p,c,\zeta))\). 
    It is easy to see that \(X\) is an irreducible cubic plane curve (see \cite[Remark 3.9]{abedin2021geometrization}), i.e.\ it is defined by an equation of the form \(v^2 = u^3+au+b\) for some \(a,b \in F\). Moreover, \(X\) is smooth if \(4a^3+27b^2 \neq 0\), in which case its an elliptic curve, and has a unique singularity \(s\) otherwise. This singularity is nodal if \(0 \neq 4a^3 = -27b^2\) and cuspidal otherwise;

    \item The fact that \(W \subseteq \fg(\!(x)\!)\) is a Lagrangian subalgebra
    with respect to \(B_1 = \mathcal{K}_1\) satisfying \(\fg[\![x]\!] \dotplus W = \fg(\!(x)\!)\) implies that \(\mathcal{A}\) satisfies the conditions (a)-(c) in \cref{prop:weakly_g_locally_free_modules}.(2) (see \cite[Subsection 3.2]{abedin2021geometrization}).
    Therefore, \(\mathcal{A}|_q \cong \fg\) for all smooth closed points \(q \in X\). 
    This can be used to describe \(\mathcal{A}|_C\) explicitly in all three cases, where \(C \subseteq X\) is the set of smooth points; see \cite[Subsection 4.2]{abedin2021geometrization}. For instance, if \(X\) is cuspidal, \(C\) is isomorphic to the affine line \(\mathbb{A}^1 = \textnormal{Spec}(F[x])\) and
    then
    \(\Gamma(C,\mathcal{A}) \cong \fg[x]\) according to \cref{prop:weakly_g_locally_free_modules}.(3). 
    The explicit description in the elliptic case was thereby only achieved
    for \(F = \mathbb{C}\), that is why the base field is restricted to the field of complex numbers;
    
    \item In \cite{burban_galinat}, the authors construct a distinguished section \(\rho \in \Gamma(C \times C \setminus \Delta, \mathcal{A}\boxtimes \mathcal{A})\) called geometric \(r\)-matrix. This section has the initial formal \(r\)-matrix \(r\) as Taylor series; see \cite[Theorem 3.17]{abedin2021geometrization}. 
    Combined with the explicit description of \(\mathcal{A}|_C\) in (2) above, 
    this implies that \(r\) is elliptic (resp.\ trigonometric, resp.\ rational) in the sense of
    \cref{thm:raschid_classification} if and only if \(X\) is smooth (resp.\ nodal, resp.\ cuspidal); see \cite[Subsection 4.3]{abedin2021geometrization} for details. 
\end{enumerate}

\subsection{Twisting class \( \fg \ot A(1,0) \)}
Fix a Lagrangian Lie subalgebra
\begin{equation}
W \subseteq \fg \otimes A(1,0) = \fg(\!(x)\!) \times \fg,
\end{equation}
such that 
\( 
\iota(\fg[\![x]\!]) \add W = \fg(\!(x)\!) \times \fg
\)
and let \(W_{\pm}\) be its projection
on the first and second components
respectively.
Then 
\begin{equation}
((Y,\mathcal{W}),(p,c,\zeta)) \coloneqq \mathbb{G}(M_{W_+},W_+)    
\end{equation}
satisfies \(Y = \mathbb{P}^1\) because of \cref{lemm:geometric_genus}.(4). 
Put \(s_-\coloneqq p\) and let \(s_+ \in \mathbb{P}^1\) be the point corresponding to the ideal \((x^{-1}) \subseteq F[x^{-1}]\) via \(c(\Gamma(\mathbb{P}^1\setminus\{s_-\},\mathcal{O}_X)) = F[x^{-1}]\).

\begin{lemma}\label{lemm:thetas}
Let \( \mathcal{W} \) be defined as above.
Let the sheaf of Lie algebras \(\mathcal{V}\) be defined by the pull-back diagram
\begin{equation}
    \xymatrix{\mathcal{V} \ar[r] \ar[d] 
    & W_- \ar[d] \\
    \mathcal{W} \ar[r] &  \mathcal{W}|_{s_-} \cong \fg}
\end{equation}
%
where \(\fg\), \(W_-\), and \(\mathcal{W}|_{s_-}\) are understood as skyscraper sheaves at \(s_-\).
Then the following is true:
\begin{enumerate}
    \item \(\mathcal{V}\) can be identified with a subsheaf of \(\mathcal{W}\);
    
    \item \(\textnormal{H}^0(\mathcal{V}) \cong \iota(\fg[\![x]\!]) \cap (W_+ \times W_-) \), \(\textnormal{H}^1(\mathcal{V}) = 0\)
    and \(\mathcal{V}|_{\mathbb{P}^1\setminus\{s_-\}} = \mathcal{W}|_{\mathbb{P}^1\setminus\{s_-\}}\);
    
    \item Let \(\widetilde{K}\colon \mathcal{V} \times \mathcal{V}\to \mathcal{O}_{\mathbb{P}^1}\) be the restriction of the Killing form of \(\mathcal{W}\) to \(\mathcal{V}\). 
    There exist canonical surjective morphisms \(\mathcal{V}|_{s_\pm} \to W_\pm/W_\pm^\bot\)
    intertwining the corresponding forms
    if
    \(\mathcal{V}|_{s_\pm}\) is equipped with the bilinear form \(\widetilde{K}|_{s_\pm}\). 
\end{enumerate}
\end{lemma}
\begin{proof}
By definition, \(\mathcal{V}\) fits into the short exact sequence
\begin{equation} \label{eq:short_exseq_VW}
    0 \longrightarrow \mathcal{V} \longrightarrow \mathcal{W} \oplus W_- \longrightarrow \fg \longrightarrow 0.
\end{equation}
The morphism \(W_- \to \fg\) is injective, so \(\mathcal{V}\to\mathcal{W}\) is too and we can identify \(\mathcal{V}\) with a subsheaf of \(\mathcal{W}\), proving (1).

Restricting \cref{eq:short_exseq_VW} to \(\mathbb{P}^1\setminus\{s_-\}\) yields \(\mathcal{V}|_{\mathbb{P}^1\setminus\{s_-\}} = \mathcal{W}|_{\mathbb{P}^1\setminus\{s_-\}}\).
Since the first cohomology group of torsion sheaves vanishes and \(\fg[\![x]\!] + W_+ = \fg(\!(x)\!)\) implies \(\textnormal{H}^1(\mathcal{W}) = 0\), the long exact sequence of \cref{eq:short_exseq_VW} in cohomology reads
\begin{equation}\label{eq:long_exact_cohomology_V}
    0 \longrightarrow \textnormal{H}^0(\mathcal{V}) \longrightarrow \textnormal{H}^0(\mathcal{W}) \oplus W_- \longrightarrow \fg \longrightarrow \textnormal{H}^1(\mathcal{V}) \longrightarrow 0.
\end{equation}
The fact that \(\textnormal{H}^0(\mathcal{W})  \cong \fg[\![x]\!]\cap W_+\) gives \[\textnormal{H}^0(\mathcal{V}) \cong \iota(\fg[\![x]\!]) \cap (W_+ \times W_-).\]
The image \(\overline{W}_+\) of \(\fg[\![x]\!]\cap W_+\) under the evaluation \(x = 0\) has the property \(\overline{W}_+ + W_- = \fg\). Therefore, the map \(\textnormal{H}^0(\mathcal{W}) \oplus W_- \to \fg\) in 
\cref{eq:long_exact_cohomology_V}
is surjective and hence \(\textnormal{H}^1(\mathcal{V}) = 0\). This concludes the proof of (2). 

Let us turn to the proof of the last
statement.
\Cref{lemm:geometric_genus}.(4) implies that \(W_+\) is a free Lie algebra over \(F[x^{-1}]\), so \(\kappa(a,b) \in F[x^{-1}]\) for all \(a,b \in W_+ \subseteq \fg(\!(x)\!)\). This implies that
\begin{equation}
    B_2(x^{-1}a,b) = \textnormal{res}_0 \{ \underbrace{x^{-2}\kappa(a,b)}_{\in  x^{-2}F[x^{-1}]} \} = 0
\end{equation}
for all \(a,b \in W_+\).
Therefore, \(x^{-1}W_+ \subseteq W_+^\bot\) and we have a surjective morphism 
\[
\theta_+\colon \mathcal{V}|_{s_+} \cong W_+/x^{-1}W_+ \longrightarrow W_+/W_+^\bot
\] 
intertwining the corresponding forms.
By construction of \(\mathcal{V}\) we have a canonical morphism \(\mathcal{V} \to W_-\) which is surjective since
\(\mathcal{W} \to \mathcal{W}|_{s_+}\)
is surjective. This morphism factors through
a surjective morphism \(\mathcal{V}|_{s_-} \to W_-\) which intertwines the forms. The latter induces the desired morphism \(\theta_-\). 
\end{proof}

\begin{lemma}\label{lemm:A_weakly_g_locally_free_qtcase}
Let \(X\) be an irreducible cubic plane curve with nodal singularity \(s\) and chose the normalization \(\nu \colon \mathbb{P}^1 \to X\) in such a way that \(\nu^{-1}(s) = \{s_+,s_-\}\). 
Let \(\mathcal{A}\) be defined by the pull-back diagram 
\begin{equation}
    \xymatrix{\mathcal{A}\ar[r]\ar[d]&
    W/(W_+^\bot \times W_-^\bot)\ar[d]\\ \nu_*\mathcal{V}\ar[r]_-{\theta}&W_+/W_+^\bot \times W_-/W_-^\bot}
\end{equation}
where \(W/(W_+^\bot \times W_-^\bot)\) and \( W_+/W_+^\bot \times W_-/W_-^\bot\) are viewed as skyscraper sheaves at \(s\in X\) and \(\theta\) is the direct image under \(\nu\) of the morphism 
\begin{equation}
     \mathcal{V} \longrightarrow \mathcal{V}|_{s_+}\times \mathcal{V}|_{s_-} \stackrel{(\theta_+,\theta_-)}\longrightarrow W_+/W_+^\bot \times W_-/W_-^\bot
\end{equation}
for \(\theta_\pm\) from  \cref{lemm:thetas}.
Then \(\mathcal{A}|_q \cong \fg\) for all smooth closed points \(q\in X\).
\end{lemma}
\begin{proof}
It suffices to show that \(\mathcal{A}\) satisfies the conditions (a)-(c) of \cref{prop:weakly_g_locally_free_modules}.(2).

\emph{\(\bullet\) \(\mathcal{A}\) satisfies condition (a) of \cref{prop:weakly_g_locally_free_modules}.(2): }
The long exact sequence in cohomology of
\begin{equation}
    0 \longrightarrow \mathcal{A} \longrightarrow \nu_*\mathcal{V} \oplus (W/(W_+^\bot \times W_-^\bot)) \longrightarrow W_+/W_+^\bot \times W_-/W_-^\bot \longrightarrow 0,
\end{equation}
which is the short exact sequence that defines \(\mathcal{A}\), is given by
\begin{equation}\label{eq:long_exact_sequence_cohomology_A_qt}
    0 \longrightarrow \textnormal{H}^0(\mathcal{A}) \longrightarrow \textnormal{H}^0(\mathcal{V}) \oplus (W/(W_+^\bot \times W_-^\bot)) \longrightarrow W_+/W_+^\bot \times W_-/W_-^\bot \longrightarrow \textnormal{H}^1(\mathcal{A})\longrightarrow 0.
\end{equation}
Here, we used that the first cohomology group of torsion sheaves vanishes and \(\textnormal{H}^1(\mathcal{V}) = 0\); see \cref{lemm:thetas}.(2). 
The canonical map \(\textnormal{H}^0(\mathcal{V})\to W_+/W_+^\bot \times W_-/W_-^\bot\) thereby coincides with the inclusion 
\[
\iota(\fg[\![x]\!]) \cap (W_+ \times W_-) \longrightarrow W_+/W_+^\bot \times W_-/W_-^\bot
\]
under the identification \(\textnormal{H}^0(\mathcal{V}) \cong \iota(\fg[\![x]\!]) \cap (W_+ \times W_-)\).
Therefore, \cref{lem:preleminary_facts_for_geometrization}.(3) implies that the middle arrow in \cref{eq:long_exact_sequence_cohomology_A_qt} is an isomorphism. Consequently,
\(\textnormal{H}^0(\mathcal{A}) = 0 = \textnormal{H}^1(\mathcal{A})\);

\emph{\(\bullet\) \(\mathcal{A}\) satisfies condition (b) of \cref{prop:weakly_g_locally_free_modules}.(2): }
\Cref{prop:weakly_g_locally_free_modules}.(1) and \(\mathcal{W}|_{s_-} \cong \fg\) imply that there exists a closed point \(p \in \mathbb{P}^1\setminus\{s_+,s_-\}\) such that \(\fg \cong \mathcal{W}|_p \cong \mathcal{A}|_{\nu(p)}\);

\emph{\(\bullet\) \(\mathcal{A}\) satisfies condition (c) of \cref{prop:weakly_g_locally_free_modules}.(2): }
Let us identify \(\mathcal{A}\) with a subsheaf of \(\nu_*\mathcal{V}\) and let \(K\colon \mathcal{A}\times \mathcal{A}\to \nu_*\mathcal{O}_{\mathbb{P}^1}\) be the restriction of \(\nu_*\widetilde{K}\) to \(\mathcal{A}\), where we recall that \(\widetilde{K}\colon \mathcal{V}\times \mathcal{V} \to \mathcal{O}_{\mathbb{P}^1}\) is the restriction of the Killing form of \(\mathcal{W}\) to \(\mathcal{V}\). We want to see that \(K\) actually takes values in \(\mathcal{O}_X \subseteq \nu_*\mathcal{O}_{\mathbb{P}^1}\). Let \(a,b\in\mathcal{A}|_s\) and \(a_\pm,b_\pm \in W_\pm\) be representatives of the images of \(a,b\) under the canonical maps \(\mathcal{A}|_s \to \mathcal{V}|_{s_\pm} \to W_\pm/W_\pm^\bot\). Then
\begin{equation}
    K|_s(a,b) = (\mathcal{K}_2(a_+,b_+),\mathcal{K}_2(a_-,b_-)) \in F \times F \cong \nu_*\mathcal{O}_{\mathbb{P}^1}|_s
\end{equation}
holds, since \(\theta_\pm \colon \mathcal{V}|_{s_\pm} \to W_\pm/W_\pm^\bot\) intertwine
the forms. 
The definition of \(\mathcal{A}\) implies that
\begin{equation}
    (a_+,a_-),(b_+,b_-) \in W,
\end{equation}
and the Lagrangian property of \(W\) gives
\begin{equation}
    0 = B_2((a_+,a_-),(b_+,b_-)) = \mathcal{K}_2(a_+,b_+) - \mathcal{K}_2(a_-,b_-).
\end{equation}
We obtain \(K|_s(a,b)\in \{(\lambda,\lambda)\mid \lambda \in F\}\). This implies that \(K\) takes values in \(\mathcal{O}_X\). The restriction of \(K \colon \mathcal{A}\times \mathcal{A} \to \mathcal{O}_X\) to the set \(C \coloneqq X\setminus\{s\}\) of smooth points of \(X\) coincides with the Killing form of \(\mathcal{A}|_C\) since \(\mathcal{A}|_C = \nu_*(\mathcal{V}|_{\mathbb{P}^1\setminus\{s_+,s_-\}}) = \nu_*(\mathcal{W}|_{\mathbb{P}^1\setminus\{s_+,s_-\}})\).
\end{proof}

\begin{lemma}\label{lemm:concluding_qtcase}
There exists an
\(F[\![x]\!]\)-linear Lie algebra automorphism \(\varphi \colon \fg[\![x]\!] \to \fg[\![x]\!]\) such that \(\varphi(W_+) \subseteq \fg[x,x^{-1}]\). Furthermore, \((\varphi\times [\varphi])(W) \subseteq \fg(\!(x)\!)\times \fg\) is commensurable with \(W_2\).
\end{lemma}
\begin{proof}
\Cref{lemm:A_weakly_g_locally_free_qtcase} implies \( \mathcal{W}|_{q}\cong \mathcal{A}|_{\nu(q)} \cong \fg\) for all \(q \in \mathbb{P}^1\setminus\{s_+,s_-\}\). Combined with \(\mathcal{W}|_{s_-} \cong \fg\), this implies that \(A \coloneqq \zeta(\Gamma(\mathbb{P}^1\setminus\{s_+\},\mathcal{W}))\subseteq \fg[\![x]\!]\) is a free \(F[x] = c(\Gamma(\mathbb{P}^1\setminus\{s_+\},\mathcal{O}_{\mathbb{P}^1}))\)-Lie algebra satisfying \(A/(x-a)A \cong \fg\) for all \(a \in F\). Therefore, \cref{prop:weakly_g_locally_free_modules}.(3) provides an isomorphism \(A \cong \fg[x]\). Completing said automorphism in the \((x)\)-adic topology yields \(\varphi \in \textnormal{Aut}_{F[\![x]\!]\textnormal{-LieAlg}}(\fg[\![x]\!])\) with the property \(\varphi(A) = \fg[x]\). Since \(\mathcal{W}\) is a sheaf, we have
\begin{equation}
    \varphi(W_+) = \varphi(\zeta(\Gamma(\mathbb{P}^1\setminus\{s_-\},\mathcal{W})) \subseteq \varphi(\zeta(\Gamma(\mathbb{P}^1\setminus\{s_+,s_-\},\mathcal{W}))) = \varphi(A)[x^{-1}] = \fg[x,x^{-1}].
\end{equation}
Combining \(\dim(\varphi(W_+) \cap \fg[\![x]\!]) < \infty\) with \(\varphi(W_+) \subseteq \fg[x,x^{-1}]\) results in 
\begin{equation}
    x^{-N}\fg[x^{-1}] \subseteq \varphi(W_+) \subseteq x^{N}\fg[x^{-1}]    
\end{equation}
for a sufficiently large integer \(N\). This implies that
\( (\varphi\times [\varphi])(W) \subseteq \fg(\!(x)\!) \times \fg \) is commensurable with \(W_2\) defined in \cref{subsec:realization_of_doubles_polynomials}.
\end{proof}

\subsection{Twisting class \( \fg \ot A(2,0) \)}
As in the previous case
we fix a Lagrangian Lie subalgebra
\begin{equation}
W \subseteq \fg \otimes A(2,0) = \fg(\!(x)\!) \times \fg[x]/x^2\fg[x]    
\end{equation}
complementary to
\( \iota(\fg[\![x]\!]) \)
with projections
\( W_{\pm} \) on the components
\( \fg(\!(x)\!)\) and
\(\fg[x]/x^2\fg[x]\) respectively.

\begin{lemma}\label{lemm:quasirational_lemma}
The following facts are true.
\begin{enumerate}
    \item \(W = W_+ \times W_-\);
    
    \item \(W_+ \cap x^2\fg[\![x]\!] = \{0\}\), so \(W_+\cap \fg[\![x]\!]\) can be identified with a subalgebra of \(\fg[x]/x^2\fg[x]\);
    
    \item \((W_+ \cap \fg[\![x]\!]) \dotplus W_- = \fg[x]/x^2\fg[x] \).
\end{enumerate}
\end{lemma}
 \begin{proof}
For (1) observe that \(\kappa(a,b) \in F[x^{-1}]\) for all \(a,b\in W_+\), since \(W_+\) is a free Lie algebra over \(F[x^{-1}]\) by virtue of \cref{lemm:geometric_genus}.(4). Therefore, \(x^{-2}\kappa(a,b) \in x^{-2}F[x^{-1}]\) implies
\begin{equation}
    B_3(a,b) = \textnormal{res}_{x = 0} \{ x^{-2}\kappa(a(x),b(x)) \} = 0
\end{equation}
and hence \(W_+ \subseteq W_+^\bot\). Together with \(W_+^\bot \subseteq W_+\) we arrive at \(W_+ = W_+^\bot\). \Cref{lem:preleminary_facts_for_geometrization}.(3) implies \(W_- = W_-^\bot\), so \(W_+^\bot \times W_-^\bot \subseteq W \subseteq W_+ \times W_-\) concludes the proof of (1).

The identities \(\{0\} = (\fg[\![x]\!] + W_+)^\bot = x^2\fg[\![x]\!] \cap W_+^\bot = x^2\fg[\![x]\!] \cap W_+\) imply (2). Part (3) now follows from (2) and \(\iota(\fg[\![x]\!]) \dotplus (W_+ \times W_-) = \fg(\!(x)\!) \times \fg[x]/x^2\fg[x]\).
\end{proof}

\noindent
Consider \(((Y,\mathcal{W}),(p,c,\zeta)) \coloneqq \mathbb{G}(M_{W_+},W_+)\). As in the last section we have \(Y = \mathbb{P}^1\). Let us write \(0 \coloneqq p\) and \(\infty \in \mathbb{P}^1\) for the point corresponding to the ideal \((x^{-1})\) in \(F[x^{-1}] = c(\Gamma(\mathbb{P}^1\setminus\{0\},\mathcal{O}_X))\).

\begin{lemma}
Let \(X\) be an irreducible plane cubic curve with cuspidal singularity \(s\) and chose the normalization \(\nu \colon \mathbb{P}^1 \to X\) in such a way that \(\nu(s) = 0\). 

\begin{enumerate}
    \item The isomorphism \(\zeta \colon \widehat{\mathcal{W}}_p \to \fg[\![x]\!]\) induces a surjective morphism \(\nu_*\mathcal{W} \to \fg[x]/x^2\fg[x]\);
    
    \item Let the sheaf of Lie algebras \(\mathcal{A}\) be defined by the pull-back of
    \begin{equation}
        \xymatrix{\mathcal{A} \ar[r]\ar[d] & W_- \ar[d] \\ \nu_*\mathcal{W} \ar[r] & \fg[x]/x^2\fg[x]}
    \end{equation}
    where \(\fg[x]/x^2\fg[x]\) and \(W_-\) are understood as skyscraper sheaves at \(s\). Then \(\mathcal{A}|_q \cong \fg\) for all smooth points \(q \in X\).
\end{enumerate}
\end{lemma}
\begin{proof}
The isomorphism \(\zeta \colon \widehat{\mathcal{W}}_p \to \fg[\![x]\!]\) implies that \(\nu_*\mathcal{W}|_s \cong \zeta(\widehat{\mathcal{W}}_0)/x^2\zeta(\widehat{\mathcal{W}}_0) = \fg[x]/x^2\fg[x]\). This yields a surjective morphism \(\nu_* \mathcal{W} \to \fg[x]/x^2\fg[x]\).
In order to prove the second part of the statement, it suffices to show that \(\mathcal{A}\) satisfies the conditions (a)-(c) of \cref{prop:weakly_g_locally_free_modules}.(2).

\emph{\(\bullet\) \(\mathcal{A}\) satisfies condition (a) of \cref{prop:weakly_g_locally_free_modules}.(2): }
Globally, the morphism \(\nu_* \mathcal{W} \to \fg[x]/x^2\fg[x]\) coincides with the canonical morphism \(\fg[\![x]\!]\cap W_+ \to \fg[x]/x^2\fg[x]\) after identifying \(\textnormal{H}^0(\mathcal{W})\) with \(\fg[\![x]\!] \cap W_+\). Therefore, the middle arrow in the long exact sequence in cohomology 
\begin{equation}\label{eq:long_exact_sequence_cohomology_A_qr}
    0 \longrightarrow \textnormal{H}^0(\mathcal{A}) \longrightarrow \textnormal{H}^0(\mathcal{W}) \oplus W_- \longrightarrow \fg[x]/x^2\fg[x] \longrightarrow \textnormal{H}^1(\mathcal{A})\longrightarrow 0,
\end{equation}
of the short exact sequence
\begin{equation}
    0 \longrightarrow \mathcal{A} \longrightarrow \nu_*\mathcal{W} \oplus W_- \longrightarrow \fg[x]/x^2\fg[x] \longrightarrow 0,
\end{equation}
which defines \(\mathcal{A}\), is an isomorphism by virtue of \cref{lemm:quasirational_lemma}.(3). 
Here we used again that the first cohomology group of torsion sheaves vanishes and that \(\textnormal{H}^1(\mathcal{W}) = 0\) by virtue of \cref{lem:preleminary_facts_for_geometrization}.(2). Consequently, \(\textnormal{H}^0(\mathcal{A}) = 0 = \textnormal{H}^1(\mathcal{A})\);

\emph{\(\bullet\) \(\mathcal{A}\) satisfies condition (b) of \cref{prop:weakly_g_locally_free_modules}.(2): }
Part (1) in \Cref{prop:weakly_g_locally_free_modules}
and \(\mathcal{W}|_{0} \cong \fg\) imply
that there exists 
\(p \in \mathbb{P}^1\setminus\{0\}\) 
such that 
\(\fg \cong \mathcal{W}|_p \cong \mathcal{A}|_{\nu(p)}\);

\emph{\(\bullet\) \(\mathcal{A}\) satisfies condition (c) of \cref{prop:weakly_g_locally_free_modules}.(2): }
Let \(\widetilde{K}\) be the Killing form of \(\mathcal{W}\). Identify \(\mathcal{A}\) with a subsheaf of \(\nu_*\mathcal{W}\) and let \(K\colon \mathcal{A}\times \mathcal{A} \to \nu_*\mathcal{O}_{\mathbb{P}^1}\) be the the restriction of \(\nu_*\widetilde{K}\) to \(\mathcal{A}\). For any \(a,b \in \mathcal{A}|_s\) we have  
\begin{equation}
    K|_{s}(a,b) = \kappa(a_1,b_1) + [x](\kappa(a_1,b_2) +  \kappa(a_2,b_1))  \in F[x]/(x^2),
\end{equation}
where \(a_1 + [x]a_2\) and \(b_1 + [x]b_2 \in \fg[x]/x^2\fg[x]\) are the images of \(a \) and \(b\) respectively
under 
\[\mathcal{A}|_s \to \nu_*\mathcal{W}|_s \cong \fg[x]/x^2\fg[x].
\]
By definition of \(\mathcal{A}\), \(a_1 + [x]a_2\), \(b_1 + [x]b_2 \in W_- \) and \(\kappa(a_1,b_2) +  \kappa(a_2,b_1) = 0\) since \(W_-\subseteq \fg[x]/x^2\fg[x]\) is Lagrangian. Therefore, \(K|_{s}(a,b) = \kappa(a_1,b_1)\in F\),
implying that \(K\) takes values \(\mathcal{O}_X \subseteq \nu_*\mathcal{O}_{\mathbb{P}^1}\). The restriction of the bilinear form \(K \colon \mathcal{A}\times \mathcal{A} \to \mathcal{O}_X\) to the set \(C \coloneqq X\setminus\{s\}\) of smooth points of \(X\) coincides with the Killing form of \(\mathcal{A}|_C\) since \(\mathcal{A}|_C = \nu_*(\mathcal{W}|_{\mathbb{P}^1\setminus\{0\}})\).
\end{proof}

\noindent
The proof of the following statement is completely analogous to the proof of \cref{lemm:concluding_qtcase}.

\begin{lemma}
There exists an \(F[\![x]\!]\)-linear Lie algebra automorphism \(\varphi \colon \fg[\![x]\!] \to \fg[\![x]\!]\) such that \(\varphi(W_+) \subseteq \fg[x,x^{-1}]\). Furthermore, \((\varphi \times [\varphi])W \subseteq \fg(\!(x)\!)\times \fg[x]/x^2\fg[x]\) is commensurable with \(W_3\).
\end{lemma}

\newpage

\appendix
\section{Different types of equivalence}%
\label{sec:appendix}
Now we give a brief description of different
equivalence notions used in this paper.
We denote by \(F\) an algebraically closed field
of characteristic \(0\).
\vspace{5px}

\begin{adjustbox}{center}
\begin{tabular}[c]{ |c|c|c|  }
 \hline 
 \((L, \delta_1) \cong (L, \delta_2)\) & \makecell{Isomorphism of two \\ topological Lie bialgebra \\ structures} & \makecell{There exists \(\varphi \in \textnormal{Aut}_{F\textnormal{-LieAlg}}(L)\) 
\\ such that \(\varphi\) and its dual \(\varphi' \) are  \\ homeomorphisms and \\
 \( (\varphi \fot \varphi)\delta_1 = \delta_2 \varphi\)}  \\
 \hline
 \((L, \delta_1) \sim (L, \delta_2)\) & \makecell{Equivalence of two \\ topological Lie bialgebra \\ structures} & \makecell{There exists a constant
 \(\xi \in F^\times\)\\ such that \( (L, \xi \delta_1) \cong (L, \delta_2)\) }  \\
\hline
\( (L, L_+, L_-) \cong (M, M_+, M_-)\) & \makecell{Isomorphism of two \\ topological Manin triples}  & \makecell{There exists a Lie algebra  \\ isomorphism \(\varphi \colon L \to M\) such that \\ \(\varphi\) is a
homeomorphism,   \\ \(\varphi(L_\pm) = M_\pm\) and it intertwines \\
the corresponding forms}  \\
\hline
\( (L, L_+, L_-) \sim (M, M_+, M_-)\) & \makecell{Equivalence of two \\ topological Manin triples}  & \makecell{There exists a constant
 \(\xi \in F^\times\)\\ such that the Manin triple  \\ \((L, L_+, L_-)\) with form \(\xi B_L\)  \\ is isomorphic to \((M, M_+, M_-)\) \\ with form \(B_M\)}  \\
\hline
\( (A_1,t_1) \sim (A_2,t_2)\) & \makecell{Equivalence of two \\ trace extensions of \(F[\![x]\!]\)}  & \makecell{There exists an algebra isomorphism \\ \(T \colon A_1 \to A_2\),
identical on \(F[\![x]\!]\), \\ and 
a constant \(\xi \in F^\times\) such that \\
\(t_2 \circ T = \xi t_1 \)} \\
\hline
\end{tabular}
\end{adjustbox}

\vspace{5px}
Lie bialgebra isomorphisms of \(\fg[\![x]\!]\),
given by \(x \mapsto a_1x + a_2x^2 + \dots\) with \(a_i \in F \) and \(a_1 \neq 0\),
are called \emph{coordinate transformations}.
We show in \cref{thm:splitting_aut} that 
any Lie algebra automorphism of \(\fg[\![x]\!]\)
decomposes uniquely into a coordinate transformation
and an \(F[\![x]\!]\)-linear Lie algebra automorphism of \(\fg[\![x]\!]\).

In general, scaling and coordinate transformations
do not preserve the topological double of a topological
Lie bialgebra: they change the corresponding form.
These equivalences are used to place each topological Lie bialgebra
structure into a
particular topological double.
After that we work primarily with \emph{formal isomorphisms},
i.e.\  \(F[\![x]\!]\)-linear
Lie algebra automorphisms of \(\fg[\![x]\!]\) because they leave
topological doubles invariant.

Topological twists -- topological
analogue of classical twists --
are introduced in 
\cref{sec:topological_twists}.
This notion allows to reduce
the classification
of topological Lie bialgebra
structures to the classification
of twists within certain 
topological doubles.
We call two topological twists \(s_1, s_2 \in (\fg \ot \fg)[\![x,y]\!]\)
of \(\delta\)
\emph{formally isomorphic}
if the corresponding Lie bialgebra structures
\(\delta + ds_1\) and \(\delta + ds_2\) are formally isomorphic.

In \cref{sec:classification_of_doubles_MSZ}
we explain that there are only
three non-trivial topological doubles
of topological Lie bialgebra
structures on \( \fg[\![x]\!] \).
They are denoted by
\( \fD(\fg[\![x]\!], \delta_i) \),
\( i \in \{1,2,3 \} \).
Each topological twist of \( \delta_i\)
is completely determined 
by a certain Lagrangian Lie subalgebra of \( \fD(\fg[\![x]\!], \delta_i) \).
Conversely, every Lagrangian Lie
subalgebra of \( \fD(\fg[\![x]\!], \delta_i) \), complementary
to \( \fg[\![x]\!] \),
defines a topological twist
of \( \delta_i \).
Two Lagrangian Lie subalgebras
\( W_1, W_2 \subseteq \fD(\fg[\![x]\!], \delta_i) \) 
are said to be \emph{formally isomorphic}
if there is an 
\( F[\![x]\!] \)-linear
automorphism
\( \phi\) of \( \fg[\![x]\!]\)
such that
\[ 
W_2 = (\phi \times [\phi])(W_1).
\]

Another important notion
tightly related to
topological twists
is the notion of a formal 
\( r \)-matrix; see \cref{sec:formal_r_matrices}.
Again, each topological
twist of \(\delta_i\)
gives rise to a formal
\(r\)-matrix \(-r_i + t\) 
and conversely,
any formal
\(r\)-matrix of the form 
\(-r_i + t\) defines a
topological twist of \(\delta_i\).
Two formal \(r\)-matrices \(r_1\) and \(r_2\) are called
\emph{formally} (resp.\ \emph{polynomially}) \emph{gauge equivalent} if there is an element \(\varphi\) in
\(\Aut_{F[\![x]\!]\textnormal{-LieAlg}}(\fg[\![x]\!])\) (resp.\ in \(\Aut_{F[x]\textnormal{-LieAlg}}(\fg[x])\))
such that
\[(\varphi(x) \ot \varphi(y))r_1(x,y) = r_2(x,y).\] 
Similarly, classical \(r\)-matrices \(r_1\) and \(r_2\) are called
\emph{holomorphically gauge equivalent} if there is a holomorphic
function \(\varphi \colon U \subseteq \bC \to \Aut(\fg)\) such that
\((\varphi(x) \ot \varphi(y))r_1(x,y) = r_2(x,y)\).

The statement of
\cref{thm:equivalence} tells us
that the relations between
the three objects mentioned above
preserve equivalences.

\newpage
\printbibliography\end{document}